\documentclass[a4paper,12pt,reqno]{amsart}

\usepackage{tikz-cd}

\usepackage[latin1]{inputenc} 
\usepackage{graphicx}
\usepackage[matrix,arrow,tips,curve]{xy}
\usepackage[english]{babel}
\usepackage{amsmath}
\usepackage{amssymb}

\usepackage{enumerate,bm}

\newtheorem{thm}[equation]{Theorem}
\newtheorem{lemma}[equation]{Lemma}
\newtheorem{corollary}[equation]{Corollary}
\newtheorem{proposition}[equation]{Proposition}
\newtheorem{assumption}[equation]{Assumption}

\newtheorem{propdefi}[equation]{Proposition - Definition}

\newtheorem*{thm*}{Theorem}
\newtheorem{parg}[equation]{}

\numberwithin{equation}{section}
\numberwithin{figure}{section}
\numberwithin{table}{section}

\theoremstyle{definition}

\newtheorem{question}[equation]{Question}
\newtheorem{remark}[equation]{Remark}
\newtheorem{prg}[equation]{}

\newtheorem{notation}[equation]{Notation}

\numberwithin{figure}{section}

\newcommand{\ph}{\varphi}
\newcommand{\de}{\delta}
\newcommand{\w}{\widetilde}
\newcommand{\ma}{\mathcal}
\newcommand{\la}{\longrightarrow}
\newcommand{\ol}{\mathcal{O}}
\newcommand{\wi}{\widehat}
\newcommand{\pr}{\mathbb{P}}
\newcommand{\Q}{\mathbb{Q}}
\newcommand{\C}{\mathbb{C}}
\newcommand{\R}{\mathbb{R}}
\newcommand{\Z}{\mathbb{Z}}
\newcommand{\N}{\mathcal{N}_1}
\newcommand{\Nu}{\mathcal{N}^1}

\newcommand{\dom}{\operatorname{dom}}

\newcommand{\Spec}{\operatorname{Spec}}
\newcommand{\Sing}{\operatorname{Sing}}
\newcommand{\Pic}{\operatorname{Pic}}
\newcommand{\Rat}{\operatorname{Ratcurves}^n}

\newcommand{\NE}{\operatorname{NE}}
\newcommand{\Exc}{\operatorname{Exc}}
\newcommand{\Supp}{\operatorname{Supp}}
\newcommand{\Lo}{\operatorname{Locus}}
\newcommand{\codim}{\operatorname{codim}}
\newcommand{\Eff}{\operatorname{Eff}}

\newcommand{\Sym}{\operatorname{Sym}}

\newcommand{\Hom}{\operatorname{Hom}}

\newcommand{\coker}{\operatorname{coker}}
\newcommand{\Bl}{\operatorname{Bl}}

\newcommand{\Div}{\operatorname{div}}

\usepackage{etoolbox}
\patchcmd{\section}{\normalfont}{\normalfont\large}{}{}
\patchcmd{\subsection}{\bfseries}{\scshape\centering}{}{}
\patchcmd{\subsection}{-.5em}{.5em}{}{}

\makeatletter
\@namedef{subjclassname@2020}{\textup{2020} Mathematics Subject Classification}
\makeatother

 \title[On  Fano varieties with Lefschetz defect 2]{On the geometry of Fano varieties with Lefschetz defect 2}
\author{C.~Casagrande}
\address{Universit\`a di Torino,
Dipartimento di Matematica,
via Carlo Alberto 10,
10123 Torino - Italy}
\email{cinzia.casagrande@unito.it}
\date{\today}
\subjclass[2020]{14J45,14E30}

\calclayout

\begin{document}
\maketitle


\section{Introduction}
\noindent In the study of smooth, complex, higher dimensional Fano varieties $X$ with large Picard number, an important invariant is the {\bf Lefschetz defect} $\delta_X$, defined as follows. For any prime divisor $D\subset X$,  consider the natural map $\N(D)\to\N(X)$ between the real vector spaces of 1-cycles modulo numerical equivalence, and set $\N(D,X)\subset\N(X)$ to be its image. Then we define
$\delta(D):=\codim\N(D,X)$ and
$$\delta_X:=\max\{\delta(D)\,|\,D\text{ a prime divisor in }X\}.$$

The first strong property of the Lefschetz defect is the following.
\begin{thm}[\cite{codim}, Theorem 1.1]\label{geq 4}
  Let $X$ be a smooth Fano variety. If $\delta_X\geq 4$, then $X\cong S\times T$, where $S$ is a del Pezzo surface with $\rho_S=\delta_X+1$, and $T$ a smooth Fano variety with $\delta_T\leq \delta_X$.
\end{thm}

Thus every Fano variety $X$ that is not a product of a del Pezzo surface with another  Fano variety has $\delta_X\in\{0,1,2,3\}$.
When $\delta_X=3$, $X$ does not need to be a product, but it nonetheless has a very rigid structure, described as follows.
\begin{thm}[\cite{delta3}, Theorem 1.4]\label{delta=3}
  Let $X$ be a smooth Fano variety with $\delta_X=3$ and dimension $n$. Then there exist:
\begin{enumerate}[$\bullet$]
\item  a smooth Fano variety $T$ with $\dim T=n-2$, $\rho_T=\rho_X-4$, and $\delta_T\leq 3$;
  \item
  a $\pr^2$-bundle $\pi\colon Z\to T$, given by the projectivization of a totally decomposable rank 3 vector bundle on $T$;
\item
a birational map $\alpha\colon X\to Z$ that blows-up three pairwise disjoint smooth irreducible subvarieties $A,A',B\subset Z$ of codimension $2$, such that $A$ and $A'$ are sections of $\pi$, and $B$ is either a section of $\pi$, or is finite of degree $2$ onto $T$.
\end{enumerate}
\end{thm}
Note that the composition $\pi\circ\alpha\colon X\to T$ is a flat fibration in del Pezzo surfaces. We also note that the construction of $X$ from $T$ is completely explicit, in the sense that both $Z$ and the centers of the blow-up $\alpha$ are determined by the choice of some line bundles on $T$ satisfying suitable conditions; we refer the interested reader to \cite[Propositions 1.2 and 1.3]{delta3} for more details.

This paper is devoted to the next case, namely that of Lefschetz defect $\delta_X=2$. A first study of this case has been developed in \cite{codimtwo}; here we refine that study, and our first result is the following.
\begin{thm}\label{main}
  Let $X$ be a smooth Fano variety with $\de_X=2$. Then there exist:
\begin{enumerate}[$\bullet$]
\item a sequence of $K$-negative flips $\xi\colon X\dasharrow X'$, with 
  $X'$ smooth;
\item 
  a conic bundle\footnote{A conic bundle is a flat map such that every fiber is isomorphic to a plane conic.} $\zeta\colon X'\to Y$, such that $Y$ is smooth with $\rho_X-\rho_Y=2$;
  \item a factorization of $\zeta$ as  $\ph\circ \sigma$,
    where  $\sigma\colon X'\to X_2$ is the blow-up of a smooth, irreducible, codimension 2 subvariety $A_2\subset X_2$,  $X_2$ is smooth, and
    $\ph\colon X_2\to Y$ is an elementary\footnote{Namely with $\rho_{X_2}-\rho_Y=1$.} conic bundle.
   $$\xymatrix{ X\ar@{-->}[r]^{\xi}&{X'}\ar[r]^{\sigma}\ar[dr]_{\zeta}&{X_2}\ar[d]^{\ph}\\
    &&Y
  }$$
  Moreover $\ph$ is smooth over $A_Y:=\ph(A_2)\subset Y$, $A_2$ is a section of $\ph_{|\ph^{-1}(A_Y)}$, and  $\zeta^{-1}(A_Y)\subset\dom(\xi^{-1})$.
   \end{enumerate}
\end{thm}  

We note that flips do occur; this can be seen in examples of toric Fano $4$-folds, such as those of type $J$ (with $\rho_X=4$) or those of type $R$ (with $\rho_X=5$) in \cite{bat2}.

\medskip

Let us recall that a conic bundle either is a smooth morphism, or has singular fibers over a non-empty divisor in the base, called the discriminant of the conic bundle.
Next we study  whether the elementary conic bundle $\ph\colon X_2\to Y$ as in Theorem \ref{main} can have non-empty discriminant divisor $\Delta_\ph$. It turns out that, when $\ph$ is not smooth, we get a stronger version of the result above. In order to give the statement, let us introduce the following notation, needed for generic fibers: if $S$ is a smooth surface over a field, we denote by $\Bl(S;d_1,d_2)$ the blow-up of $S$ at two closed points of degrees $d_1$ and $d_2$ (see \S\ref{delpezzo} for more details).
\begin{thm}[case $\Delta_\ph\neq\emptyset$]\label{singular}
  In the situation of Theorem \ref{main}, suppose that the conic bundle $\ph$ has non-empty discriminant. Then $\xi$ is an isomorphism (so that $X=X'$), $A_Y\cap\Delta_\ph=\emptyset$, and $Y=\pr_T(\ma{F})$ for some smooth projective variety $T$ and
  rank $2$ vector bundle $\ma{F}$ on $T$, such that $A_Y$ is a section of the $\pr^1$-bundle $\pi_Y\colon Y\to T$.
  $$\xymatrix{X\ar[r]_{\sigma}\ar@/^{1pc}/[rr]^{\zeta}
    \ar@/_/[rrd]_{\psi}&{X_2}\ar[r]_-{\ph}&{Y=\pr_T(\ma{F})}\ar[d]^{\pi_Y}\\
    &&T
  }$$
  
    The composition $\psi:=\pi_Y\circ\zeta\colon X\to T$ is a flat fibration in del Pezzo surfaces, with integral fibers.  
    Let $X_t\subset X$ be the fiber of $\psi$ over a general closed point $t\in T$, and $X_\eta$ the generic fiber of $\psi$, where $\eta:=\Spec K$ is the generic point of $T$, $K:=\C(T)$ the field of rational functions on $T$. Then one of the following holds:
\begin{enumerate}[$(i)$]    
\item $X_\eta\cong\Bl(\pr^2_K;1,4)$ and $\rho_{X_t}=6$;
  \item  $X_\eta\cong\Bl(\pr^2_K;2,2)$ and $\rho_{X_t}=5$.
\end{enumerate}  
 We also have $\NE(\psi)\cong\NE(X_\eta)$, a $3$-dimensional cone with $4$ extremal rays, and every contraction of $X_\eta$ is the restriction of a contraction of $X$ over $T$.
\end{thm}  

Finally we study more in detail case $(i)$, which is the simplest from the point of view of the generic fiber $X_\eta$ of $\psi$. We get a complete description of $X$ as a blow-up of a $\pr^2$-bundle over $T$, as follows.
\begin{thm}\label{rho=6}
In the situation of Theorem \ref{singular}, suppose that we are in case $(i)$. 
Then $T$ is Fano with $\delta_T\leq 2$, and there is a factorization of $\psi$ as:
$$\xymatrix{
  X
  \ar[rrd]|{\,\psi\,}\ar@/^{1pc}/[rr]^{\alpha} \ar[r]|>>>>>>>{\,\tau'\,}
  &{X_2'}\ar[d]_{\ph'}\ar[r]|-{\,\sigma'\,}
  &{Z=\pr_T(\ol\oplus\ol(D)\oplus\ol(2D))}\ar[d]^{\pi_Z} \\
 & {W=\pr_T(\ol\oplus\ol(D))}\ar[r]|-{\,\pi_W}&{T}
}$$
where:
\begin{enumerate}[$(a)$]
  \item $D$ is a divisor on $T$ such that  $-K_T-2D$ and $-K_T+D$ are ample;
\item $\alpha$ is the blow-up of two disjoint smooth, irreducible subvarieties $A_Z,B_Z\subset Z$ of codimension $2$; $A_Z$ is the section of $\pi_Z$ corresponding to the projection  $\ol_T\oplus\ol_T(D)\oplus\ol_T(2D)\twoheadrightarrow\ol_T$, while $\pi_{Z|B_Z}\colon B_Z\to T$ is finite of degree $4$;
\item $\sigma'$ is the blow-up of $A_Z$, and $\ph'$ is a $\pr^1$-bundle;
\item let $B_W\subset W$ be the image under $\ph'$ of the transform of $B_Z\subset Z$ in $X_2'$. Then
$B_W$ is  a smooth connected divisor in $|4H_W|$, where $H_W$ is the tautological divisor of $W$.
\end{enumerate}
\end{thm} 
Let us note that in the statement above, $\ph'\colon X_2'\to W$ is smooth, while $\ph\colon X_2\to Y$ has non-empty discriminant by assumption, hence the factorization $\psi=\pi_W\circ\ph'\circ\tau'$ is actually different from the original factorization $\psi=\pi_Y\circ \ph\circ \sigma$ given by Theorem \ref{singular}.

In the setting of Theorem \ref{rho=6}, the Fano variety $X$ determines the pair $(T,D)$.
Conversely, we show that $X$ can be reconstructed from $(T,D)$.
\begin{proposition}\label{construction}
  Let $T$ be a smooth Fano variety, of dimension $n-2$, with $\delta_T\leq 2$. Let $D$ be a divisor on $T$ such that $-K_T-2D$ and $-K_T+D$ are ample, and such that, in $W=\pr_T(\ol\oplus\ol(D))$ with tautological divisor $H_W$, there exists a smooth connected divisor $B_W\in |4H_W|$.\footnote{If $D$ is base-point-free and $D\not\equiv 0$, the existence of such $B_W$ is automatic, see Remark \ref{bsptfree}.}

  Then there exists a smooth Fano variety $X$ with $\dim X=n$,
  $\delta_X=2$, and $\rho_X=\rho_T+3$,
  as in Theorem \ref{singular} case $(i)$, 
  and
  obtained from $(T,D)$ as in Theorem \ref{rho=6}.

  Moreover the  section $A_Y$ of $\pi_Y\colon Y=\pr_T(\ma{F})\to T$  (notation as in Theorem \ref{singular}) corresponds to an
  extension:
   $$0\la\ol_T(2D)\la\ma{F}\la\ol_T\la 0$$
   and $\Delta_\ph\sim 3A_Y+6\pi_Y^*(D)$.
 \end{proposition}

 In dimension 4, this construction yields one family of Fano $4$-folds with $\de_X=2$ and $\rho_X=4$, obtained as a blow-up of $\pr_{\pr^2}(\ol\oplus\ol(1)\oplus\ol(2))$ along the negative section of the $\pr^2$-bundle and a K3 surface having degree $4$ onto $\pr^2$; this corresponds to the pair $(T,D)=(\pr^2,\ol_{\pr^2}(1))$, see Section \ref{dim4}.

 \medskip

 Throughout the paper we use techniques from birational geometry in the framework of the MMP. We work over the field of complex numbers $\C$, but we also consider and study generic fibers $X_\eta$ of del Pezzo fibrations $\psi\colon X\to T$ ($\eta$ being the generic point of $T$), where $X$ is a smooth, complex Fano variety; such generic fiber is a smooth del Pezzo surface over the non-closed field $\C(T)$ of rational functions on $T$. In particular we use the classification of del Pezzo surfaces with Picard number $3$ over arbitrary fields in \cite{BFSZ}. In our setting, the knowledge of $X_\eta$ gives us information on the relative cone $\NE(\psi)$, and hence on the possible factorizations of $\psi$, see \S \ref{delpezzo} and \S \ref{fibrations}.

The paper is organized as follows. In Section \ref{prel} we set up the notation and give some preliminary results on different topics: in \S\ref{disjdiv} we consider properties of a pair of effective divisors with disjoint support in a variety $X$ with a smooth fibration with fiber $\pr^1$, that are interesting in their own right. 
In \S\ref{delpezzo} we state the results on del Pezzo surfaces over non-closed fields from \cite{BFSZ} that are needed in the sequel, and in \S\ref{fibrations} we relate some properties of contractions of fiber type over $\C$ to properties of the generic fiber.

In Section \ref{proof} first we recall what is a special MMP for $-D$, where $D$ is a prime divisor in a Fano variety $X$ such that $\de(D)=\de_X=2$, and  we recall some results from \cite{codimtwo}, in particular Proposition \ref{prop5.1}, and \ref{previous}. Then we prove Theorem \ref{main} on the structure of Fano varieties with $\de_X=2$; see \ref{outlinemain} for an outline of the proof.

 In Section \ref{discriminant} we study the case where the elementary conic bundle $\ph\colon X_2\to Y$ in Theorem \ref{main} is not smooth. Again, our starting point is a result from \cite{codimtwo} (Proposition \ref{giugno}), that in this setting shows the existence of a del Pezzo fibration $\psi\colon X\to T$. We study this del Pezzo fibration and prove Theorem \ref{singular}; see \ref{outline_singular} for an outline of the proof.

 Sections \ref{section_rho=6}
and \ref{reconstructing} are devoted to study case $(i)$ of Theorem \ref{singular}, where the generic fiber of the del Pezzo fibration $\psi\colon X\to T$ is $\Bl(\pr^2_{\C(T)};1,4)$. In Section \ref{section_rho=6} we show Theorem \ref{rho=6}, that gives an explicit description of $X$ as a blow-up of $\pr_T(\ol\oplus\ol(D)\oplus(2D))$, where $T$ is Fano, and $D$ is a suitable divisor on $T$; see \ref{outline_rho=6} for an outline of the proof. Then in Section \ref{reconstructing} we reverse this description, and show  that given a Fano variety $T$ and a divisor $D$ on $T$ with suitable properties, we can construct a Fano variety $X$ with $\de_X=2$ as above (Proposition \ref{construction}).
In Section \ref{dim4} we apply these results to get a family of Fano $4$-folds with Lefschetz defect $2$ as in  Theorem \ref{singular}$(i)$.

Finally in Section \ref{questions} we state some open questions on Fano varieties $X$ with $\de_X=2$, and consider the special case where in Theorem \ref{main} the elementary conic bundle $\ph$ is smooth and there are no flips, namely $X=X'$.

\medskip

{\footnotesize

\noindent{\bf Acknowledgements} I am grateful to Eric Jovinelly for useful conversations about this project. I also thank Saverio Secci for a careful reading of a preliminary version of this paper.

The author is a member of GNSAGA, INdAM.

  \tableofcontents}
 \section{Preliminaries}\label{prel}
 \subsection{Notation}\label{notation}
 \noindent Let $X$ be a smooth projective variety.\\
$\dom(f)$ is the domain of a rational map $f\colon X\dasharrow Y$;\\
$\sim$ denotes linear equivalence for Cartier divisors;\\
$\equiv$ denotes numerical equivalence both  for $1$-cycles and $\R$-divisors;\\
$\Nu(X)$ is the $\R$-vector space of $\R$-divisors up to numerical equivalence;\\
$\N(X)$ is the $\R$-vector space of $1$-cycles, with $\R$-coefficients, up to numerical equivalence;\\
 $\rho_X=\dim\Nu(X)=\dim\N(X)$ is the Picard number of $X$;\\
 $[D]\in\Nu(X)$ is the class of a divisor $D$, and $[C]\in\N(X)$ is the class of a curve $C$;\\ $D^{\perp}:=\{\gamma\in\N(X)\,|\,D\cdot\gamma=0\}\subset\N(X)$, $D$ a divisor in $X$.

 We refer the reader to \cite{kollarmori} for the standard terminology in birational geometry, and to \cite{hukeel} for the notion of Mori dream space. Every Fano variety is a Mori dream space \cite[Corollary 1.3.2]{BCHM}.

 Assume that $X$ is a Mori dream space. \\
 $\NE(X)\subset\N(X)$ is the convex cone spanned by classes of effective curves, and $\Eff(X)\subset\Nu(X)$ is the convex cone spanned by classes of effective divisors. Both cones are rational polyhedral.

 A {\bf contraction} of $X$ is a surjective morphism $f\colon X\to Y$, with connected fibers, where $Y$ is normal and projective; $f$ can be birational or of fiber type. If $Y$ is $\Q$-factorial, then $Y$ is still a Mori dream space. We set $\NE(f):=\NE(X)\cap\ker f_*$ where $f_*\colon\N(X)\to\N(Y)$ is the pushforward of 1-cycles; $\NE(f)$ is a face of $\NE(X)$ of dimension $\rho_X-\rho_Y$, and we say that $f$ is the contraction of $\NE(f)$.
A contraction is {\bf elementary} if $\rho_X-\rho_Y=1$; an elementary birational contraction can be divisorial or small.

An {\bf extremal ray} $R$ is a one-dimensional face of $\NE(X)$; we have $R=\NE(f)$ for an elementary contraction $f\colon X\to Y$. We set $\Lo(R):=\Exc(f)$, and we say that $R$ is birational, or small, if $f$ is.
If $D$ is a divisor on $X$, we write $D\cdot R>0$ (or $<0$, etc.) if $D\cdot \gamma>0$ for $\gamma\in R\smallsetminus\{ 0\}$ (or $<0$, etc.).

 A flip is the flip of a small elementary contraction (equivalently of a small extremal ray).

Let $f\colon X\to Y$ be an elementary contraction. We say that $f$, or $\NE(f)$, is {\bf of type $(n-1,n-2)^{sm}$} (respectively, of type $(n-1,n-2)\,$) if $f$ is the blow-up of a smooth, irreducible subvariety of codimension $2$ in $Y$, contained in $Y_{reg}$ (respectively, if $\dim\Exc(f)=n-1$ and $\dim f(\Exc(f))=n-2$).
 
Given a divisor $D$, a contraction $f\colon X\to Y$ is $D$-negative if $-D$ is $f$-ample. When $D=K_X$, we just say {\bf $K$-negative}.

We say that two extremal rays $R$, $R'$ of $\NE(X)$ are {\bf adjacent} if $R+R'$ is a face of $\NE(X)$.

Given a closed subset $Z$ of $X$, we set $\N(Z,X):=\iota_*(\N(Z))\subset\N(X)$, and 
$\NE(Z,X):=\iota_*(\NE(Z))\subset\NE(X)$, where $\iota\colon Z\hookrightarrow X$ is the inclusion. Note that if $f\colon X\to Y$ is a morphism, we have $\N(f(Z),Y)=f_*(\N(Z,X))$.

For every prime divisor $D\subset X$, we set $\de(D):=\codim\N(D,X)$.

A {\bf smooth $\pr^1$-fibration} is a smooth morphism with fiber $\pr^1$. A  {\bf $\pr^m$-bundle} is the projectivization of a vector bundle of rank $m+1$.
A {\bf conic bundle} is a flat contraction of fiber type such that every fiber is isomorphic to a plane conic.

An {\bf exceptional $\pr^1$-bundle} is a smooth prime divisor $E\subset X_{reg}$, with a $\pr^1$-bundle structure, such that if $e\subset E$ is a fiber, we have $E\cdot e=-1$. Then the class $[E]\in\Nu(X)$ generates a one-dimensional face of the cone $\Eff(X)$. The exceptional divisor of an elementary contraction of type $(n-1,n-2)^{sm}$ is an exceptional $\pr^1$-bundle.

Let $f\colon X\to Y$ be a contraction of fiber type. We say that $f$ is
{\bf quasi-elementary} if, for a general fiber $F$ of $f$, we have $\dim\N(F,X)=\rho_X-\rho_Y$. We refer the reader to \cite{fanos} for properties of quasi-elementary contractions of fiber type.

We work over the field of complex numbers $\C$, but we will also consider {\bf generic fibers} of some contraction of fiber type $f\colon X\to Y$. If $K=\C(Y)$ is the field of rational functions on $Y$, let $\eta=\Spec K$ be the generic point of $Y$. Then $X_\eta=\eta \times_{Y} X$ is the generic fiber of $f$, a projective variety over the field $K$.

\begin{remark}\label{disjoint}
  Let $X$ be a smooth projective variety, $Z\subset X$ a closed subset, and $D\subset X$ a prime divisor. If $Z\cap D=\emptyset$, then $\N(Z,X)\subset D^{\perp}$. Indeed for every curve $C\subset Z$ we have $D\cdot C=0$, hence $[C]\in D^{\perp}$.
  \end{remark}
  \begin{remark}\label{linindep}
Let $X$ be a smooth projective variety and $E_1,E_2,E_3\subset X$ distinct exceptional $\pr^1$-bundles. Then $[E_1],[E_2],[E_3]\in\Nu(X)$ generate distinct one-dimensional faces of $\Eff(X)$, therefore they are linearly independent.
    \end{remark}
    \begin{remark}\label{easy}
      Let $X$ be a smooth Fano variety with $\de_X=2$, $D\subset X$ a prime divisor, and  $E_1,E_2\subset X$ distinct exceptional $\pr^1$-bundles. If $D$ is disjoint from both $E_1$ and $E_2$, then $\de(D)=2$ and $\N(D,X)=E_1^{\perp}
      \cap E_2^{\perp}$. Moreover $D\cap E_3\neq\emptyset$ for every exceptional $\pr^1$-bundle $E_3$ distinct from $E_1$ and $E_2$.
    \end{remark}
    \begin{proof}
We have $\N(D,X)\subset E_1^{\perp}
      \cap E_2^{\perp}$ by Remark \ref{disjoint}, and $\codim (E_1^{\perp}
      \cap E_2^{\perp})=2$ by Remark \ref{linindep}. Since $\de_X=2$, we deduce
      $\de(D)=2$ and $\N(D,X)=E_1^{\perp}
      \cap E_2^{\perp}$. Finally $D\cap E_3\neq\emptyset$ again by Remarks \ref{disjoint} and  \ref{linindep}.
      \end{proof}
\begin{lemma}\label{pullback}
  Let $(X,\Delta)$ be a projective klt pair with $\Delta$ effective. Let $f\colon X\to Y$ be a $(K+\Delta)$-negative contraction, and $D$ an effective Cartier divisor on $X$ such that $\NE(f)\subset D^{\perp}$.

  Then there exists an effective Cartier divisor $D_Y$ on $Y$ such that $D=f^*(D_Y)$.
\end{lemma}
We remark that the statement above is an equality of divisors, not only of linear equivalence classes.
\begin{proof}
  By \cite[Theorem 3.7(4)]{kollarmori} there exists a Cartier divisor $E$ such that $\ol_X(D)\cong f^*(\ol_Y(E))=\ol_X(f^*(E))$. Then $D$ and $f^*(E)$ are linearly equivalent, and there exists a non-zero rational function $\ph\in\C(X)$ such that $D=f^*(E)+\Div(\ph)$. Since $D\geq 0$, this means that $\ph\in H^0(X,f^*(E))$. On the other hand $f_*\ol_X\cong\ol_Y$, hence $f_*\ol_X(f^*(E))\cong\ol_Y(E)$, and the pullback via $f$ gives an isomorphism
  $H^0(X,f^*(E))\cong H^0(Y,E)$. We conclude that there exists $\psi\in\C(Y)$ such that $\ph=f^*(\psi)$, therefore $D=f^*(E+\Div(\psi))$, and we get the statement with $D_Y:=E+\Div(\psi)$.
\end{proof}  
\subsection{Disjoint divisors in $\pr^1$-fibrations}\label{disjdiv}
\noindent
In this subsection we give several preliminary results on effective divisors with disjoint support in a variety that has a smooth $\pr^1$-fibration; these will be needed in the paper, but they are interesting in their own right. We also refer the reader to \cite{totaro2000, BPS2016} for related results and examples on effective divisors with disjoint support.

The following two lemmas are variations of \cite[Lemma 2.5]{delta3} and \cite[3.3.16]{codim}.
\begin{lemma}\label{section}
  Let $X$ be a smooth projective variety and $\pi\colon X\to Y$ a smooth $\pr^1$-fibration. Let $D_1$ and $D_2$ be effective divisors in $X$, with disjoint support, such that $\pi$ is finite on both supports.
Then one of the following holds:
  \begin{enumerate}[$(i)$]
\item 
  $D_j$ is supported on a section of $\pi$ for some $j\in\{1,2\}$;
\item
 there exists $\lambda\in \Q$ such that $D_1\equiv\lambda D_2$.
  \end{enumerate}
\end{lemma}
\begin{proof}
  Let $\Gamma\subset X$ be a fiber of $\pi$ and
  set $d_i:=D_i\cdot \Gamma$ for $i=1,2$, so that
  $(d_2D_1-d_1D_2)\cdot \Gamma=0$.
  Since $\pi$ is an elementary, $K$-negative contraction, we have
  $d_2D_1-d_1D_2\sim\pi^*(N)$ for some divisor $N$ on $Y$.
   If $N\equiv 0$, then $d_2D_1\equiv d_1D_2$, and we have $(ii)$.

Suppose that $N\not\equiv 0$. Then $N^{\perp}\subset\N(Y)$ is a hyperplane, 
and there exists
a curve $C\subset Y$, complete intersection of very ample divisors 
$H_1,\dotsc,H_{n-2}$, such that 
 $N\cdot C\neq 0$ (see for instance \cite[Remark 2.4]{delta3}).  We choose   
 $H_i$ general in their linear system.

 Set $S:=\pi^{-1}(C)\subset X$; by Bertini $C$ is a smooth irreducible curve, and
 $S$ is a smooth ruled surface.

 Suppose that $N\cdot C>0$,  let $C_2$ be an irreducible component of $D_{2\,|S}$, and note that $C_2\cap \Supp D_1=\emptyset$ hence $D_1\cdot C_2=0$.
 Then
 \begin{align*}
 D_2\cdot C_2=&\frac{1}{d_1}\bigl(d_2D_1-\pi^*(N)\bigr)\cdot C_2
 =-\frac{1}{d_1}\pi^*(N)\cdot C_2=-\frac{1}{d_1}N\cdot\pi_*(C_2)\\=&-
 \frac{\deg\pi_{|C_2}}{d_1}N\cdot C<0\end{align*}
 hence $D_{2\,|S}\cdot C_2<0$.

 Since $D_{2\,|S}$ is an effective divisor in $S$, it must be $(C_2)^2<0$.
Then $\NE(S)=\langle [\Gamma],[C_2]\rangle$ and, by the properties of ruled surfaces (see for instance \cite[Ch.~V, Propositions 2.20 and 2.21]{hartshorne}),
$C_2$ is the unique negative section of the ruled surface $S$. This shows that $D_{2\,|S}=d_2 C_2$.
By generality of $C$, we conclude that $D_2$ is irreducible and that $(\Supp D_2)\cdot \Gamma=1$, so that $\Supp D_2$
is a section of $\pi$, and we have $(i)$.

In the case $N\cdot C<0$, we argue in a similar way with $D_1$.
\end{proof}
\begin{lemma}\label{section2}
In the setting of Lemma \ref{section}, suppose furthermore that $X$ is simply connected. Then one of the following holds:
  \begin{enumerate}[$(i)$]
\item 
  $D_j$ is supported on a section of $\pi$ for some $j\in\{1,2\}$;
\item
  $X\cong Y\times\pr^1$, $\pi$ is the projection, and $D_i=Y\times\{pts\}$ for $i=1,2$.
  \end{enumerate}
\end{lemma}
\begin{proof}
  By Lemma \ref{section} either we have $(i)$ or there exists $\lambda\in\Q_{>0}$ such that $D_1\equiv \lambda D_2$. In this last case, we show that $(ii)$ holds.

  There exist positive integers $m_1,m_2$ such that $m_1D_1\equiv m_2D_2$. Since $X$ is simply connected, we have $h^1(\ol_X)=0$, and up to taking larger multiples, we can assume that $m_1D_1$ and $m_2D_2$ are linearly equivalent. Since they are effective with disjoint support, it is not difficult to see that the linear system $|m_1D_1|$ defines a morphism $X\to\pr^1$ (see for instance \cite[Lemma 2.6]{BPS2016}). By taking the Stein factorization, we get a contraction $f\colon X\to C$, with $C$ a smooth curve, that sends every connected component of $\Supp(D_i)$ to a point, for $i=1,2$.

  Let $\Gamma\subset X$ be a fiber of $\pi$. We have $\N(X)=\R[\Gamma]\oplus\N(\Supp(D_1),X)$ and $\N(\Supp(D_1),X)=\ker f_*$, thus $f(\Gamma)=\pr^1$. Then $\pi$ and $f$ induce a finite morphism $g\colon X\to Y\times\pr^1$.

Consider a general fiber $F$ of $f$ and a point $x_0\in F$; set $g(x_0)=(y_0,t_0)\in Y\times\pr^1$. Then both $f$ and $\pi$ are smooth at $x_0$, and this implies that $g$ is smooth at $x_0$ too. Indeed $d_{x_0}g\colon T_{x_0}X\to T_{g(x_0)}(Y\times\pr^1)=T_{y_0}Y\oplus T_{t_0}\pr^1$ has components precisely $d_{x_0}\pi$ and $d_{x_0}f$. This means that the ramification 
divisor $R$ of $g$ does not meet $F$, and $g_{|F}\colon F\to Y$ is étale, of the same degree as $g$.

On the other hand $X$ is simply connected and $\pi\colon X\to Y$ is topologically a fiber bundle (with respect to the complex topology), with simply connected fiber $\pr^1$, hence $\pi_1(Y)\cong\pi_1(X)$ by the long exact sequence of homotopy groups  (see for instance \cite[Theorem 4.41 and Proposition 4.48]{hatcher}), and $Y$ is simply connected too. We conclude that $g_{|F}$ and $g$ have degree one, so that $g$ is an isomorphism and we have $(ii)$.
\end{proof}  
\begin{lemma}\label{ample}
  Let $X$  be a projective variety
 and $\pi\colon X\to Y$ a surjective morphism such that every fiber of $\pi$ is one-dimensional. Suppose that $\ker\pi_*=\R[F]$, where $F\subset X$ is a fiber of $\pi$ and $\pi_*\colon\N(X)\to\N(Y)$ is the pushforward. Suppose also that there exist $A_1,A_2\subset X$ 
disjoint prime Cartier divisors such that $\pi(A_i)=Y$, $i=1,2$.

Let $H\in\Pic(X)$. The following are equivalent:
\begin{enumerate}[$(i)$]
  \item
    $H$ is ample;
  \item $H\cdot F>0$ and $H_{|A_i}$ is ample for $i=1,2$.
        \end{enumerate}
  \end{lemma}
  \begin{proof}
    The implication $(i)\Rightarrow (ii)$ is clear, let us show $(ii)\Rightarrow (i)$.

    For $i=1,2$ set $\ma{C}_i:=\overline{\NE}(X)\cap \{\gamma\in\N(X)\,|\,A_i\cdot\gamma\geq 0\}$; this is a closed convex cone in $\N(X)$, and the sum $\ma{C}_1+\ma{C}_2$ is still closed.
    We show that  $\ma{C}_1+\ma{C}_2=\overline{\NE}(X)$.

    We clearly have 
    $\ma{C}_1+\ma{C}_2\subset\overline{\NE}(X)$. For the converse, let $C\subset X$ be an irreducible curve. If $A_1\cdot C\geq 0$, then $[C]\in\ma{C}_1$. If instead $A_1\cdot C<0$, then $C\subset A_1$, thus $A_2\cap C=\emptyset$ and $A_2\cdot C=0$, therefore $[C]\in\ma{C}_2$. In any case $[C]\in \ma{C}_1+\ma{C}_2$, hence $\NE(X)\subset \ma{C}_1+\ma{C}_2$ and finally
    $\overline{\NE}(X)=\ma{C}_1+\ma{C}_2$.

    \medskip

    We show that $H$ is ample using Kleiman's criterion. Since   $\overline{\NE}(X)= \ma{C}_1+\ma{C}_2$, it is enough to show that
    $H\cdot\gamma>0$ for every $\gamma\in\ma{C}_i\smallsetminus\{0\}$ and $i=1,2$. For simplicity let us assume  $\gamma\in\ma{C}_1\smallsetminus\{0\}$.

If $\pi_*(\gamma)=0$, then $\gamma=\lambda [F]$ with $\lambda>0$, and $H\cdot\gamma=\lambda H\cdot F>0$.

Assume that $\pi_*(\gamma)\neq 0$; we have $\pi_*(\gamma)\in \overline{\NE}(Y)$. Since $\pi(A_2)=Y$, we have $(\pi_{|A_2})_*(\N(A_2))=\N(Y)$ and  $(\pi_{|A_2})_*(\overline{\NE}(A_2))=\overline{\NE}(Y)$. Let $\gamma_2\in \overline{\NE}(A_2)$ be such that $(\pi_{|A_2})_*(\gamma_2)=\pi_*(\gamma)$, and note that $\gamma_2\neq 0$, because $\pi_*(\gamma)\neq 0$.

Let also $\gamma_2'\in\overline{\NE}(X)$ be the image of $\gamma_2$ under the natural map $j_*\colon \N(A_2)\to\N(X)$ induced by the inclusion $j\colon A_2\hookrightarrow X$.
    We note that, by the projection formula,  $A_1\cdot \gamma_2'=A_{1|A_2}\cdot \gamma_2=0$, because $A_1$ and $A_2$ are disjoint.

    We have $$\pi_*(\gamma_2')=\pi_*(j_*(\gamma_2))=(\pi_{|A_2})_*(\gamma_2)=\pi_*(\gamma),$$
     hence 
     $\gamma=\gamma_2'+\mu[F]$ for some $\mu\in \R$.
     Intersecting with $A_1$ we get
     $$0\leq A_1\cdot \gamma=A_1\cdot\gamma_2'+\mu A_2\cdot F=\mu A_2\cdot F,$$
and since $A_2\cdot F>0$, we get $\mu\geq 0$. Finally
    $$H\cdot \gamma=H\cdot\gamma_2'+\mu H\cdot F\geq H\cdot\gamma_2'=H_{|A_2}\cdot \gamma_2>0,$$
because $H_{|A_2}$ is ample and $\gamma_2\in \overline{\NE}(A_2)\smallsetminus\{0\}$. 
  \end{proof}

\begin{lemma}\label{ample2}
  Let $X$ be a smooth projective variety and $\pi\colon X\to Y$ a smooth $\pr^1$-fibration.
Suppose that there exist a section
$A\subset X$ of $\pi$ and  a prime divisor $B\subset X$ disjoint from $A$.
Write $\ol_X(A)_{|A}\cong\ol_Y(D)$ for some divisor
  $D$  on $Y$.

  Let $H\in\Pic(X)$ and write $H\sim dA+\pi^*(M)$, where $d\in\Z$ and $M\in\Pic(Y)$. Then  the following are equivalent:
\begin{enumerate}[$(i)$]
  \item
    $H$ is ample;
  \item $d\geq 1$ and both $M$ and $M+dD$ are ample on $Y$.
        \end{enumerate}
\end{lemma}
\begin{proof}
  Since $\pi$ has a section, we have $X=\pr_Y(\ma{E})$ where $\ma{E}$ is a rank 2 vector bundle on $Y$; moreover we can assume that there is an exact sequence
  \begin{equation}\label{good}
  0\la \ol_Y\la\ma{E}\stackrel{\alpha}{\la} \ol_Y(D)\la 0\end{equation}
  such that the section $A$ corresponds to the surjection $\alpha$, so that $A$ is a tautological divisor and $\ol_X(A)_{|A}\cong\ol_Y(D)$.

  We have $d=H\cdot F$ where $F\subset X$ is a fiber of $\pi$. Moreover $H_{|A}\sim dA_{|A}+\pi^*(M)_{|A}\cong dD+M$ and
  $H_{|B}\sim  dA_{|B}+\pi^*(M)_{|B}=\pi_{|B}^*(M)$. Finally $\pi_{|B}\colon B\to Y$ is finite, because $A\cap B=\emptyset$, hence $B$ cannot contain any fiber of $\pi$. This implies that $H_{|B}$ is ample if and only if $M$ is ample, and the statement follows from Lemma \ref{ample}.
\end{proof}
\begin{lemma}\label{sectionsFano}
  Let $X$ be a smooth projective variety and $\pi\colon X\to Y$ a smooth $\pr^1$-fibration. Suppose that there exist a section
  $A\subset X$ of $\pi$ and  a prime divisor $B\subset X$ disjoint from $A$.
  Write $\ol_X(A)_{|A}\cong\ol_Y(D)$ for some divisor
  $D$  on $Y$.

  Then $X$ is Fano if and only if $-K_Y\pm D$ are ample on $Y$. Moreover in this case $X\cong\pr_Y(\ol\oplus\ol(D))$, and the section $A$ corresponds to the projection onto $\ol_Y(D)$.
\end{lemma}
\begin{proof}
  We keep the same notation as in the proof of Lemma \ref{ample2}. We have $\det\ma{E}\cong\ol_Y(D)$ and  $-K_X=2A+\pi^*(-K_Y-D)$, so the first statement follows from Lemma \ref{ample2}.

  Suppose now that $X$ is Fano. Then $\text{Ext}^1(\ol_Y(D),\ol_Y)\cong H^1(Y,\ol_Y(-D))=H^1(Y,K_Y-K_Y-D)=0$ by Kodaira vanishing, hence the sequence \eqref{good} splits and  $\ma{E}\cong \ol_Y\oplus\ol_Y(D)$.
\end{proof}
\begin{remark}\label{pisa0}
 Let $X$ be a smooth projective variety,
 $E\subset X$  an exceptional $\pr^1$-bundle with fiber $e\subset E$, and $D\subset X$ a prime divisor such that $D\cdot e>0$. Then  $\N(E,X)=\R[e]+\N(D\cap E,X)$. 
 Indeed $D\cap E$ is horizontal for the $\pr^1$-bundle on $E$, hence $\N(E)=\R[e]+\N(D\cap E,E)$, and we pushforward this relation to $X$.
\end{remark}  
\begin{lemma}\label{pisa}
 Let $X$ be a smooth projective variety,
  $E\subset X$  an exceptional $\pr^1$-bundle with fiber $e\subset E$, and $D_1,D_2\subset X$ disjoint prime divisors such that $D_1\cdot e>0$ and $D_2\cap E\neq\emptyset$. Then $D_2\cdot e>0$ and for $i=1,2$ we have $[e]\not\in\N(D_i,X)$ and $\N(D_i\cap E,X)$ has codimension one in $\N(E,X)$.
\end{lemma}
\begin{proof}
Since $D_1\cap D_2=\emptyset$ and $D_1\cdot e>0$, $D_2$ cannot contain any fiber of the $\pr^1$-bundle on $E$; then $D_2\cap E\neq\emptyset$ implies that $D_2\cdot e>0$.
  
  We have $\N(D_2,X)\subset D_1^{\perp}$ (see Remark \ref{disjoint}) and $D_1\cdot e>0$, hence $[e]\not\in\N(D_2,X)$,
  in particular $[e]\not\in \N(D_2\cap E,X)$. By Remark \ref{pisa0} we get $\N(E,X)=\R[e]\oplus\N(D_2\cap E,X)$. 
The same proof applied to $D_1$ gives $[e]\not\in \N(D_1,X)$ and $\N(E,X)=\R[e]\oplus\N(D_1\cap E,X)$.
\end{proof}
\begin{lemma}[\cite{codim}, Lemma 3.1.5]\label{varigotti}
  Let $X$ be a smooth Fano variety, $E\subset X$  an exceptional $\pr^1$-bundle with fiber $e\subset E$, and $D\subset X$ a prime divisor such that $D\cdot e>0$. Suppose that
  $$\N(D\cap E,X)=\N(E,X)\cap D^{\perp}\subset E^{\perp}.$$
  Then $E\cong\pr^1\times B$, $e=\pr^1\times\{pt\}$, and $D\cap E=\{pts\}\times B$.
\end{lemma}
  \subsection{Del Pezzo surfaces with Picard number $3$ over arbitrary fields}\label{delpezzo}
  \noindent Let $K$ be an arbitrary field (not necessarily algebraically closed). Smooth del Pezzo surfaces $S$ over $K$, with Picard number $\rho_S=3$, are classified in \cite{BFSZ} in more than 30 types.
They are described as $S=\Bl(Y_a;d_1,d_2)$ where $Y_a$ is a del Pezzo surface over $K$ with $\rho_{Y_a}=1$ and degree $(K_{Y_a})^2=a$, and $S$ is the blow-up of $Y_a$ at two closed points of degrees $d_1$ and $d_2$. This classification includes a complete description of the exceptional curves of $S$, of the cone of effective curves $\NE(S)$, and of the associated contractions, see
[\emph{ibidem}, Figures on pages 62--64].

We will need the following two results, that follow from this classification. We will apply them to the generic fiber of a del Pezzo fibration on a complex Fano variety.
\begin{lemma}[\cite{BFSZ}]\label{10surfaces}
  Let $S$ be a smooth del Pezzo surface over the field $K$, with $\rho_S=3$.
 If $S$ does not have a morphism onto a curve, then $S$ is one of the 10 surfaces in Table \ref{table1}.
    \begin{table}[h]
$\begin{array}{|c|c|c|c|}
\hline

\text{\em surface} & \text{\em figure in \cite{BFSZ}} & \text{\em degree} & \text{\em n.\ of extremal rays of $\NE(S)$}\\ 
  
  \hline
  \hline
  
  \Bl(\pr^2_K;2,3)   & \text{\em Fig.~16} & 4 & 5  \\

  \hline

 \Bl(\pr^2_K;2,5) &   \text{\em Fig.~18}     & 2 & 6 \\

        \hline

 \Bl(\pr^2_K;2,6) &  \text{\em Fig.~19}  &   1  &  8 \\

       \hline

 \Bl(Y_9;3,3) &   \text{\em Fig.~20}   &3  & 6 \\
       
  \hline

  \Bl(\pr^2_K;3,5) & \text{\em Fig.~22}       & 1 &  8\\

        \hline

\Bl(Y_8;3,3) & \text{\em Fig.~27} &     2   & 6 \\

\hline

  \Bl(Y_8;3,4)& \text{\em Fig.~28}  & 1 & 8 \\

  \hline

  \Bl(Y_6;2,2) & \text{\em Fig.~29}  & 2 & 6\\

  \hline

  \Bl(Y_6;2,3) & \text{\em Fig.~30}  & 1 & 8 \\
  
\hline

  \Bl(Y_3;1,1) & \text{\em Fig.~37} &  1 & 6\\
  
\hline
\end{array}$

\bigskip
  
  \caption{}\label{table1}
\end{table}
\end{lemma}
Note that in $\Bl(Y_9;3,3)$, we could have both $Y_9\cong\pr^2_K$ and $Y_9$ a non-trivial Severi-Brauer surface.
\begin{proof}
The list of possibilities for $S$ are given in \cite[Theorem 6.3(2)-(3)]{BFSZ}, so that $S$ must appear in [\emph{ibidem}, Tables 1 or 2]. For each of these surfaces, the cone $\NE(S)$ is explicitly described in [\emph{ibidem}, Figures 7 -- 39], and those with no morphism onto a curve are those with no ``black ray'' in the corresponding figure.
\end{proof}
\begin{lemma}[\cite{BFSZ}]\label{4surfaces}
   Let $S$ be a smooth del Pezzo surface over the field $K$, with $\rho_S=3$.
Suppose that there is a contraction $f\colon S\to C$ onto a curve, with $\NE(f)=R_1+R_2$, such that both extremal rays $R_1$ and $R_2$ correspond to blow-ups of points of degree one.
 Then $S$ is one of the 4 surfaces in Table \ref{table2}.
 \begin{table}[h]

   $\begin{array}{|c|c|c|c|}
\hline

      \text{\em surface} & \text{\em figure in \cite{BFSZ}} & \text{\em degree} &\text{\em n.\ of extremal rays of $\NE(S)$}\\ 
  
  \hline
  \hline
  
  \Bl(\pr^2_K;1,1)   & \text{\em Fig.~7} & 7 & 3  \\

  \hline

 \Bl(\pr^2_K;1,4) &   \text{\em Fig.~10}     & 4 & 4 \\

        \hline

 \Bl(\pr^2_K;2,2) &  \text{\em Fig.~15}  &   5  &  4 \\

       \hline

\Bl(Y_4;1,1) &   \text{\em Fig.~38}   &2  & 4 \\
         
\hline
\end{array}$

\bigskip
  
  \caption{}\label{table2}
\end{table}
\end{lemma}
\begin{proof}
  We use again the classification
  in \cite[Theorem 6.3(2)-(3)]{BFSZ} as in the previous proof.
  Now in [\emph{ibidem}, Figures 7 -- 39] we have to look for the figures with a ``black ray'' whose two adjacent rays are blue.
\end{proof}
  \subsection{Fibrations and generic fibers} \label{fibrations}
\noindent In this subsection we consider a fibration in del Pezzo surfaces $f\colon X\to Y$ and we relate some properties of the generic fiber of $f$ to some properties of $\NE(f)$. We also study properties of extremal rays of type $(n-1,n-2)$ in $\NE(f)$.
\begin{lemma}\label{restr}
  Let $X$ be a projective, normal, and $\Q$-factorial Mori dream space, and $f\colon X\to Y$ a quasi-elementary contraction of fiber type. Let $X_y$ be
the fiber of $f$ over   a general closed point $y\in Y$ and
  $X_\eta$ the generic fiber of $f$.
  
  Then the restriction $r\colon \Nu(X)\to \Nu(X_{\eta})$ is surjective with kernel $\N(X_y,X)^{\perp}$. Dually we have an isomorphism $j_*\colon \N(X_{\eta})\stackrel{\sim}{\to}\N(X_y,X)$, where $j\colon X_\eta\to X$ is the natural morphism, and $\rho_{X_\eta}=\rho_X-\rho_Y$.
\end{lemma}
\begin{proof}
  The restriction $r$ is surjective as $X$ is $\Q$-factorial, and $\ker r$
  is the linear subspace $W\subset \Nu(X)$ given by classes of $\R$-divisors $D$ whose support does not dominate $Y$.
  Then $W=\N(X_y,X)^\perp$ (see \cite[Lemma 3.3]{refinements}). Moreover $\dim
  \N(X_y,X)=\rho_X-\rho_Y$ because $f$ is quasi-elementary (see \S\ref{notation}).
  \end{proof}
  \begin{lemma}\label{NE}
    In the setting of Lemma \ref{restr}, suppose moreover that for every extremal ray $R$ of $\NE(f)$ we have $f(\Lo(R))=Y$.
    Then $\NE(X_\eta)\cong\NE(f)$ under $j_*$, and every contraction of $X_\eta$ is the restriction of a contraction of $X$ over $Y$.
  \end{lemma}
  \begin{proof}
We have $j_*(\NE(X_\eta))\subset\NE(f)$. Conversely, for every extremal ray 
    $R$ of $\NE(f)$, 
the associated contraction restricts to a non-trivial contraction 
of $X_\eta$, because $f(\Lo(R))=Y$. Thus $R\subset j_*(\NE(X_\eta))$, and finally
$j_*(\NE(X_\eta))=\NE(f)$.
\end{proof}
\begin{lemma}\label{fulton}
Let $X$ be a smooth projective variety, $f\colon X\to Y$ a contraction with $\dim X-\dim Y=2$, and $D_1,D_2$ divisors in $X$. Let $X_y$ be
the fiber of $f$ over a general closed point $y\in Y$, and $X_\eta$ the generic fiber. Then
$$D_{1|X_y}\cdot D_{2|X_y}=D_{1|X_\eta}\cdot D_{2|X_\eta}.$$
\end{lemma}
\begin{proof}
  We can assume that $y$ belongs to a smooth open subset of $Y$ over which $f$ is smooth, and over which $D_1$ and $D_2$ are flat.

  Let $Z\subset Y$ be a general very ample divisor through $y$, with $y\in Z_{reg}$; let $\eta_Z$ be the generic point of $Z$ and $X_{\eta_Z}$ the fiber of $f$ over $\eta_Z$. Let also $\ol_{Y,\eta_Z}$ be the local ring of $Y$ at $\eta_Z$; it is a discrete valuation ring with residue field $\C(Z)$ and quotient field $\C(Y)$. By considering the pullback of $f$ to $\Spec\ol_{Y,\eta}$ and using \cite[Corollary 20.3]{fultonint}, we see that the specialization map $\Pic(X_\eta)\to\Pic(X_{\eta_Z})$ preserves the intersection product. Then we get the statement by induction on the dimension.
\end{proof}  
\begin{lemma}\label{degreeR}
  Let $X$ be a projective and $\Q$-factorial Mori dream space, with terminal singularities and of dimension $n$, and $f\colon X\to Y$ a quasi-elementary $K$-negative contraction of fiber type with $\dim Y=n-2$. Let $X_y$ be
the fiber of $f$ over   a general closed point $y\in Y$, and $X_\eta$ the generic fiber; they are smooth del Pezzo surfaces over $\C$ and $\C(Y)$ respectively.

  Let $R$ be an extremal ray of $\NE(f)$ of type $(n-1,n-2)$, $E:=\Lo(R)$, 
  and $\sigma \colon X\to Z$ the contraction of $R$. $$\xymatrix{
    X\ar@/^{1pc}/[rr]^{f}\ar[r]_{\sigma} &{Z}\ar[r]_{\alpha}&Y}$$

  Then $f(E)=Y$ and there exists $d\in\mathbb{N}$ such that:
  \begin{enumerate}[(1)]
  \item $E_{|X_y}=e_1+\cdots+e_{d}$ where  $e_1,\dotsc,e_{d}$ are pairwise disjoint $(-1)$-curves in $X_y$
with $[e_1]=\cdots=[e_d]\in R$;
  \item $\sigma_\eta\colon X_\eta\to Z_\eta$ is the blow-up of a closed point of degree $d$, with exceptional curve $C=E_{|X_{\eta}}$, and $C^2=-d$. 
  \end{enumerate}
\end{lemma}
We will refer to $d$ as {\bf the degree of $R$}. \label{defidegree}
\begin{proof}
Set $B:=\sigma(E)\subset Z$.
Note first of all that $\dim\Sing(X)\leq n-3$, therefore $\Sing(X)$ cannot dominate $Y$ nor $B$; in particular $X_y$ and $X_\eta$ are smooth.  
 We also note that the general fiber of $\sigma_{|E}\colon E\to B$ is one-dimensional. By \cite[Theorem 1.2]{wisn} there is an open subset $Z_0\subset Z_{reg}$ such that $X_0:=\sigma^{-1}(Z_0)\subset X_{reg}$, $B_0:=B\cap Z_0$ is a smooth, codimension $2$ subvariety of $Z_0$, and $\sigma_{|X_0}\colon X_0\to B_0$ is the blow-up of $B_0$. 

We show that $f(E)=Y$. Otherwise we would have $X_y\cap E=\emptyset$, thus $\N(X_y,X)\subset E^{\perp}$ (see Remark \ref{disjoint}) and hence $\ker\sigma_*\not\subset\N(X_y,X)$ (note that $\ker\sigma_*=\R R\not\subset E^{\perp}$ because $E\cdot R<0$). Set $Z_y:=\alpha^{-1}(y)$. We have $Z_y=\sigma(X_y)$, hence $\N(Z_y,Z)=\sigma_*(\N(X_y,X))$, and $\dim\N(X_y,X)=\dim\N(Z_y,Z)\leq\rho_Z-\rho_Y=\rho_X-\rho_Y-1$, contradicting the fact that
$f$ is quasi-elementary (see \S\ref{notation}). Therefore we have $f(E)=Y$ and
$\alpha(B)=Y$.

Hence $\alpha_{|B}\colon B\to Y$ is generically finite; 
let $d$ be its degree, so that $\sigma_\eta\colon X_\eta\to Z_\eta$ is the blow-up of the closed point $B_\eta$, of degree $d$. 

Moreover $B$ intersects $Z_y$ transversally in $d$ smooth points, hence $\sigma_{|X_y}\colon X_y\to Z_y$ is the blow-up of $d$ reduced, distinct points; this gives the statement.
\end{proof}
\begin{lemma}\label{t}
  In the setting of Lemma \ref{degreeR}, let $R_1$ and $R_2$ be two extremal rays of $\NE(f)$ of type $(n-1,n-2)^{sm}$; for $i=1,2$ let $d_i$ be the degree of $R_i$, 
  set $E_i:=\Lo(R_i)$, and let $e_i\subset E_i$ be a non-trivial fiber of the contraction of $R_i$.
  
  Assume that
 $E_{1|E_2}=mA$ where $A\subset E_2$ is a section of the $\pr^1$-bundle and $m\in\Z_{>0}$. Then  $E_1\cdot e_2=m$ and one of the following holds:
  \begin{enumerate}[$(i)$]
    \item
$m=1$;
  \item
    $m=2$
    and $(K_{X_y})^2\leq 2$;
  \item $m=3$ and $(K_{X_y})^2=1$.
\end{enumerate}
\end{lemma}
\begin{proof}
We have $E_1\cdot e_2=E_{1|E_2}\cdot_{E_2} e_2=mA\cdot_{E_2} e_2=m$.

 Let $i\in\{1,2\}$ and write $E_{i|X_y}=e_{i,1}+\cdots+e_{i,d_i}$, where $e_{i,1},\dotsc,e_{i,d_i}$ are pairwise disjoint $(-1)$-curves in $X_y$ (see Lemma \ref{degreeR}).

  Since $A$ is a section of the $\pr^1$-bundle $\sigma_{2|E_2}\colon E_2\to B_2$, and $e_{2,1}$ is a fiber of such $\pr^1$-bundle,
  we have $A\cap e_{2,1}=p$ for some point $p\in e_{2,1}$. Then
  $$\{p\}=(E_1\cap e_{2,1})_{red}=\bigl((e_{1,1}\cup\cdots\cup e_{1,d_1})\cap
  e_{2,1}\bigr)_{red},$$
and
$$m=E_1\cdot e_{2,1}=E_{1|X_y}\cdot_{X_y} e_{2,1}=(e_{1,1}+\cdots+e_{1,d_1})\cdot_{X_y} e_{2,1}.$$
Since $e_{1,1},\dotsc,e_{1,d_1}$ are pairwise disjoint, there is an index $k\in\{1,\dotsc,d_1\}$ such that
$$e_{1,i}\cap e_{2,1}=\emptyset \text{ for every }i\neq k,\ (e_{1,k}\cap e_{2,1})_{red}=\{p\},\text{ and }e_{1,k}\cdot_{X_y} e_{2,1}=m.$$
Because these are $(-1)$-curves in a del Pezzo surface, we deduce that
 $m\leq 3$; moreover $m=3$ implies that $(K_{X_y})^2=1$, and $m=2$ implies that $(K_{X_y})^2\leq 2$ (this follows from the explicit description of $(-1)$-curves in a del Pezzo surface, see for instance \cite[Theorem 2.1]{batyrevpopov}).
\end{proof}

\section{The geometry of Fano varieties with $\delta_X=2$}\label{proof}
\noindent In this section we recall the notion and properties of a special MMP, and then we prove Theorem \ref{main}.

Let $X$ be a smooth Fano variety with $\de_X=2$, and $D\subset X$ a prime divisor with $\de(D)=2$.

 A {\bf special MMP for $-D$} is a sequence:
 \begin{equation}\label{specialMMP}
 X=X^1\stackrel{\sigma_1}{\dasharrow} X^2\dasharrow\cdots\dasharrow
X^{k-1}\stackrel{\sigma_{k-1}}{\dasharrow} X^k\stackrel{\ph}{\la}Y\end{equation}
where:
\begin{enumerate}[(1)]
\item every $X^i$ is a normal projective variety with terminal and $\Q$-factorial singularities;
\item for every $i=1,\dotsc,k-1$ there is a birational extremal ray $R_i$ of $\NE(X^i)$, with $D_i\cdot R_i>0$ and $K_{X^i}\cdot R_i<0$, such that $\sigma_i$ is either the contraction of $R_i$ (if divisorial) or its flip (if small);  we set $D_{i+1}\subset X^{i+1}$ to be the transform of $D_i\subset X^i$;
  \item the morphism $\ph\colon X^k\to Y$ is a $K$-negative elementary contraction of fiber type such that $\ph(D_k)=Y$.
  \end{enumerate}
  Special MMP's as above are the starting point of the study of Fano varieties $X$ with $\de_X>0$, see \cite[\S 2]{codim} and \cite[\S 2]{codimtwo}. A special MMP for $-D$ can have two different outcomes, similarly to the situation where in a del Pezzo surface we start contracting $(-1)$-curves having positive intersection with a given irreducible curve: we could end up with $\pr^2$, or with a ruled del Pezzo surface.  
We distinguish the two types of MMP  as follows. 
  \begin{propdefi}[\cite{codim}, Lemma 2.7(2); \cite{codimtwo}, \S 2.1]
    We have $\de(D_k)\in\{0,1\}$. We say that the special MMP is:

  --  {\bf of type $(a)$} if $\de(D_k)=0$;

  --  {\bf of type $(b)$} if $\de(D_k)=1$.
  \end{propdefi}  
The properties of MMP's of type $(a)$ and $(b)$ when $\de_X=2$ are studied in detail in \cite{codimtwo}; in particular we will need the following results.
  \begin{lemma}[\cite{codimtwo}, \S 2.2]\label{typea}
In a special MMP of type $(a)$ we have $k\geq 3$, namely there are at least two birational steps.
  \end{lemma}  
\begin{proposition}[\cite{codimtwo}]\label{prop5.1}
  Let $X$ be a smooth Fano variety with $\de_X=2$, $D\subset X$ a prime divisor with $\de(D)=2$, and  assume that there exists 
  a special MMP of type $(b)$ for $-D$ as in \eqref{specialMMP}.
  Then the following hold:
  \begin{enumerate}[$(a)$]
  \item  
    $X^k$ and $Y$ are smooth, and $\ph\colon X^k\to Y$ is a conic bundle;
  \item
  there exists $l\in\{1,\dotsc,k-1\}$ such that $\sigma_i$ is a flip for every $i\neq l$, while $\sigma_{l}\colon X^{l}\to X^{l+1}$
  is an elementary divisorial contraction of type $(n-1,n-2)^{sm}$;
  \item
    $\sigma_{l}(\Exc(\sigma_{l}))$ is contained in the open subset where the birational map  $X^{l+1}\dasharrow X^k$ is an isomorphism; let $A\subset X^k$ be its transform, and set $A_Y:=\ph(A)\subset Y$. Then $\ph$ is smooth over $A_Y$,
$A$ is a section of $\ph_{|\ph^{-1}(A_Y)}$,
  and $\ph^{-1}(A_Y)$ is contained in the open subset where the map $X^k\dasharrow X^{l+1}$ is an isomorphism;
\item
  let $\wi{E}_{l}\subset X^{l}$ be the transform of $\ph^{-1}(A_Y)$. Then
  $\Exc(\sigma_{l})\cup\wi{E}_{l}$ is contained in the open subset where the birational map $X^{l}\dasharrow X$  is an isomorphism, the map $\zeta_l\colon X^{l}\dasharrow Y$ is regular and a conic bundle in a neighborhood of
  $\Exc(\sigma_{l})\cup\wi{E}_{l}$, and $\Exc(\sigma_{l})+\wi{E}_{l}=\zeta_l^*(A_Y)$.
\end{enumerate}  
\end{proposition}
\begin{proof}
Part $(a)$ is \cite[Lemma 2.2]{codimtwo}. By [\emph{ibidem}, \S 2.3] 
there is a special index  $l\in\{1,\dotsc,k-1\}$ such that
 $\sigma_{l}\colon X^{l}\to X^{l+1}$
   is an elementary divisorial contraction of type $(n-1,n-2)^{sm}$.
   Then
 [\emph{ibidem}, Proposition 5.1] shows that $\sigma_i$ is a flip for every $i\neq l$, giving $(b)$. Finally $(c)$ and $(d)$ are again in 
[\emph{ibidem}, \S 2.3].
\end{proof}
\begin{prg}\label{outlinemain}
We are now ready to prove Theorem \ref{main} from the Introduction, on the structure of Fano varieties $X$ with $\de_X=2$. The proof will take all the rest of this section; let us give an outline.

If there exist a  prime divisor $D\subset X$ with $\de(D)=2$ and a special MMP of type $(b)$ for $-D$, then the statement follows rather directly from Proposition \ref{prop5.1} (see \ref{2014}). Thus the proof consists in showing that  this case is always verified, namely that
there always exists such $D$ with a special MMP of type $(b)$.\footnote{Differently from this case, let us note that
in general it can happen
that there are no special MMP's 
of type $(a)$, see Remark \ref{typeb}.}

We work by contradiction, and assume that
for every prime divisor $D\subset X$ with $\de(D)=2$, every special MMP for $-D$ is of type $(a)$
(see Assumption \ref{assumption}). This situation has already been studied in \cite{codimtwo}, where it is shown that under this assumption there exists a flat, quasi-elementary fibration in del Pezzo surfaces $\psi\colon X\to T$, with $\rho_X-\rho_T=3$, such that  every extremal ray of $\NE(\psi)$ is of type $(n-1,n-2)^{sm}$ (see \ref{previous}).

We consider the extremal rays $R_1,\dotsc,R_m$ of the $3$-dimensional cone $\NE(\psi)$ and their loci $E_1,\dotsc,E_m$, and study in detail the relative positions of the divisors 
$E_1,\dotsc,E_m$ and the intersections  $E_i\cdot e_j$ for $i,j=1,\dotsc,m$, where $e_j$ is a non-trivial fiber of the contraction of $R_j$  (see \ref{ilme} -- \ref{old}); here we use the assumption on the non-existence of special MMP's of type $(b)$. In particular, we show that whenever two rays $R_i$ span a face  of $\NE(\psi)$, their loci are disjoint (see \ref{cone_psi}), and the  contraction of this $2$-dimensional face  is birational (see \ref{birational}).

Then we consider the generic fiber $X_\eta$ of $\psi$,  a smooth del Pezzo surface over the field $K:=\C(T)$. Using results from \S\ref{fibrations}, we show that 
$\rho_{X_\eta}=3$, that $\NE(X_\eta)\cong\NE(\psi)$, and that every contraction of $X_\eta$ extends to a contraction of $X$ over $T$ (see \ref{turkish}). This implies that the surface $X_\eta$ does not have a morphism onto a curve, and hence that $X_\eta$ is one of the 10 surfaces in Table \ref{table1}; in particular the number $m$ of extremal rays of $\NE(\psi)$ is $5$, $6$, or $8$ (see \ref{nocurve}, \ref{list10}).

We use the previous results to show that  \emph{the classes $[E_1],\dotsc,[E_m]\in\Nu(X)$  span a $3$-dimensional subspace}, using various strategies. When $m=8$ the proof is simple (see \ref{m=8}), while in the cases where $m=5,6$ we need to use two different arguments, depending on $X_\eta$ (see \ref{milano}, \ref{dim3}, and Remark \ref{long}).

Finally, using similar techniques as in \cite[\S 3]{codim}, we use this information on the classes $[E_1],\dotsc,[E_m]$ to show that $E_i\cong\pr^1\times B$ for every $i$, and to construct a nef divisor $H$ that defines a contraction $X\to S$ onto a surface, ``transverse'' to $\psi$.  Finally we show that $X\cong S\times T$, but this is impossible by the previous classification of the possible generic fibers of $\psi$, and we get a contradiction.
\end{prg}
\begin{proof}[Proof of Theorem \ref{main}]
  Let $X$ be a Fano variety with $\de_X=2$.
\begin{parg}\label{2014}
 Suppose that there exist a  prime divisor $D\subset X$ with $\de(D)=2$ and a special MMP of type $(b)$ for $-D$:
$$X=X^1\stackrel{\sigma_1}{\dasharrow} X^2\dasharrow\cdots\dasharrow
X^{k-1}\stackrel{\sigma_{k-1}}{\dasharrow} X^k\stackrel{\ph}{\la}Y$$
  Then the statement follows from Proposition \ref{prop5.1}.
\end{parg}
\begin{proof}
By Proposition \ref{prop5.1}$(a)$
 $X^k$ and $Y$ are smooth, and $\ph$ is a conic bundle. Let $A\subset X^k$ be the transform of $\sigma_{l}(\Exc(\sigma_{l}))\subset X^{l+1}$ as in Proposition \ref{prop5.1}$(c)$. Then $\ph$ is smooth over $A_Y=\ph(A)\subset Y$, and
$A$ is a section of $\ph_{|\ph^{-1}(A_Y)}$.

Let $\sigma\colon X'\to X^k$ be the blow-up of $A$. Note that $A$ is smooth by Proposition \ref{prop5.1}, so that $X'$ is smooth too.
  We have a diagram:
  $$\xymatrix{ X\ar@{-->}@/^{1pc}/[rr]^{\xi}\ar@{-->}[r]&{X^{l}}\ar@{-->}[r]
    \ar[d]^{\sigma_{l}}&{X'}\ar[dr]^{\zeta}\ar[d]_{\sigma}&\\
    &X^{l+1}\ar@{-->}[r]&{X^k}\ar[r]^{\ph}&Y
  }$$
By Proposition \ref{prop5.1}$(c)$ and $(d)$, $\zeta:=\ph\circ\sigma\colon X'\to Y$ is a conic bundle. Moreover
 the composition $\xi\colon X\dasharrow X'$ is an isomorphism in codimension one, and $\zeta^{-1}(A_Y)\subset\dom(\xi^{-1})$.

Finally we notice that $\xi$ factors
 as a finite sequence of $K$-negative flips, see \cite[Proposition 3.7(1)]{codimtwo}. Thus we set $X_2:=X^k$ and $A_2:=A$, and we have the statement of Theorem \ref{main}.
\end{proof}
By \ref{2014}, to prove the statement it is enough to show that there  exist some $D\subset X$ with $\de(D)=2$ and a special MMP of type $(b)$ for $-D$.
 We proceed by contradiction and make the following:
\begin{assumption}\label{assumption}
For every prime divisor $D\subset X$ with $\de(D)=2$, every special MMP for $-D$ is of type $(a)$.
\end{assumption}
\begin{parg}[\cite{codimtwo}, proof of Theorem 4.1]\label{previous}
  Under Assumption \ref{assumption}  we have the following:
  \begin{enumerate}[$(a)$]
    \item\label{a}
There exists a  contraction
$$\psi\colon X\la T$$
which is a quasi-elementary, equidimensional fibration in del Pezzo surfaces, with $\dim T=n-2$ and $\rho_X-\rho_T=3$.
\item\label{b}
  Every extremal ray of $\NE(\psi)$ is of type $(n-1,n-2)^{sm}$.
\item\label{c} There are $4$ distinct extremal rays $R_0,R_1,R_2,R_3$ of $\NE(\psi)$ such that $R_0+R_1$, $R_1+R_2$,  and $R_2+R_3$ are faces of $\NE(\psi)$ (see Figure \ref{figura_raggi}), and if
  $E_{i}:=\Lo(R_i)$ for $i=0,1,2,3$, we have
  $E_{0}\cap E_{1}=\emptyset$,
  $E_{1}\cap E_{2}=\emptyset$, $E_{2}\cap E_{3}=\emptyset$, $E_0\cdot R_3>0$, and $E_3\cdot R_0>0$.
  \end{enumerate}
\end{parg}
\begin{proof}
  Statement $(\ref{a})$ is proved in \cite[proof of Theorem 4.1, 4.5 and 4.13]{codimtwo}, and $(\ref{b})$
  in [\emph{ibidem}, 4.3 -- 4.6].
  Four extremal rays of $\NE(\psi)$ are introduced  in [\emph{ibidem}, 4.3, 4.4,  4.7]:
 $R_1$, $\tilde{R}_2$, $\tilde{R}_3$, and $\tilde{R}_4$. For convenience of notation we set $R_2:=\tilde{R}_2$,  
 $R_3:=\tilde{R}_4$, and $R_0:=\tilde{R}_3$. Statement $(\ref{c})$ is shown in [\emph{ibidem}, 4.4, 4.7, 4.8, 4.9].
\end{proof}
 \begin{figure}[h]\caption{A section of $\NE(\psi)$.}\label{figura_raggi}

\bigskip
  
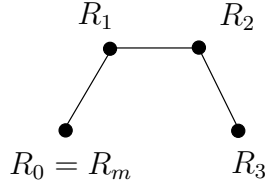
    \end{figure}
\begin{notation}\label{not}
   Let $m$ the number of extremal rays of $\NE(\psi)$; note that  $m\geq 4$ by \ref{previous}$(c)$.
  Let $R_0,R_1,R_2,R_3$ be as in \ref{previous}$(\ref{c})$, and set $R_m:=R_0$; recall that $R_0+R_1$, $R_1+R_2$, and $R_2+R_3$ are faces of $\NE(\psi)$.
Let 
 $R_4,\dotsc,R_{m-1}$ be the remaining extremal rays of $\NE(\psi)$, numbered in such a way that  $R_i$ and $R_{i+1}$ are adjacent for $i=3,\dotsc,m-1$. 
 Every $R_i$ is of type $(n-1,n-2)^{sm}$ by \ref{previous}$(\ref{b})$;  for every $i=1,\dotsc,m$ set $E_i:=\Lo(R_i)$, and let $e_i\subset E_i$ be a non-trivial fiber of the contraction of $R_i$, so that $e_i\cong\pr^1$ with $E_i\cdot e_i=-1$. 
 To simplify the notation, the index $i$ in $R_i$ is to be considered modulo  $m$, so that $R_m=R_0$,  $R_{m+1}=R_1$, etc.\, and similarly for the divisors $E_i$ and the curves $e_i$ (see Figure \ref{figura_raggi}).
\end{notation}
\begin{parg}\label{ilme}
 Notation as in \ref{not}. 
Suppose that for some $i\in\{1,\dotsc,m\}$ we have $E_{i-1}\cap E_i=E_i\cap E_{i+1}=\emptyset$. Then $E_i\cdot e_j>0$ for every $j\in\{1,\dotsc,m\}\smallsetminus \{i-1,i,i+1\}$.
\end{parg}
\begin{proof}
  Let $j\in\{1,\dotsc,m\}\smallsetminus \{i-1,i,i+1\}$, and consider the classes
  $[e_{i-1}],[e_i]$, $[e_{i+1}],[e_j]\in\N(X)$. They all belong to the $3$-dimensional subspace $\ker\psi_*$, so there must be a  relation of linear dependence among them. On the other hand each three of them are linearly independent, because they span distint extremal rays, so we can write
  $$e_j+a e_i\equiv be_{i-1}+ce_{i+1}$$
  with $a,b,c\neq 0$.
 We also note that $a,b,c$ must be positive, because $R_{i-1}$, $R_i$, $R_{i+1}$, $R_j$, $R_{i-1}+R_i$, and $R_i+R_{i+1}$ are faces of $\NE(\psi)$.

 We have $E_i\cdot e_{i-1}=E_i\cdot e_{i+1}=0$, while $E_i\cdot e_i<0$.
 By intersecting with $E_i$ in the relation above we deduce that $E_i\cdot e_j>0$.
\end{proof}
\begin{parg}\label{cone_psi}
  Notation as in \ref{not}. We have  $E_{i}\cap E_{i+1}=\emptyset$ for every $i=1,\dotsc,m$.
 \end{parg}  
\begin{proof}
  We proceed by induction on $i=1,\dotsc,m$. For $i=1,2,m$ the statement follows from   \ref{previous}$(\ref{c})$.

  Let  $i\in\{3,\dotsc,m-1\}$ and suppose that
  $E_j\cap E_{j+1}=\emptyset$ for every $j=1,\dotsc,i-1$. Since we also have
  $E_{j-1}\cap E_j=\emptyset$, it follows from \ref{ilme} that:
\begin{equation}\label{fabri}
 E_j\cdot e_h>0\text{ for every }j\in\{1,\dotsc,i-1\}\text{ and } h\in\{1,\dotsc,m\}\smallsetminus\{j-1,j,j+1\}.\end{equation}

By contradiction, suppose that $E_i\cap E_{i+1}\neq\emptyset$. By  \eqref{fabri} and the induction assumption and have $E_{i-1}\cdot R_{i+1}>0$ and $E_{i-1}\cap E_i=\emptyset$.
We apply Lemma \ref{pisa}  to the divisor $E_{i+1}$ with the pair of disjoint divisors $E_{i-1}$ and $E_i$, and we get
\begin{equation}\label{prima}
E_i\cdot R_{i+1}>0.\end{equation}

\medskip

We show that $E_{i+1}\cdot R_i>0$. Note that if $m=4$, then $i=3$, and $E_0\cdot R_3>0$ by \ref{previous}$(c)$, thus we can assume that $m\geq 5$.
By induction and by \eqref{fabri} we also have $E_{i-2}\cap E_{i-1}=\emptyset$ and $E_{i-2}\cdot R_i>0$, moreover 
$E_{i-2}$ and $E_{i-1}$ both have positive intersection with $R_{i+1}$ (because 
$i+1\neq i-3$, as $m\geq 5$).
Suppose that the contraction of $R_i$ is not finite on $E_{i+1}$. Then we apply 
\cite[Lemma 3.4]{codimtwo} to $E_{i+1}$ with the pair of disjoint divisors 
$E_{i-2}$, $E_{i-1}$ and the extremal ray $R_i$, and we get that $E_{i-2}\cdot R_i\leq 0$, a contradiction.

Therefore  the contraction of $R_i$ must be finite on $E_{i+1}$, namely $E_{i+1}$ cannot contain curves with class in $R_i$. On the other hand we have assumed that $E_i\cap E_{i+1}\neq\emptyset$, and we conclude that
 \begin{equation}\label{seconda}E_{i+1}\cdot R_i>0.\end{equation}

\medskip
 
Let us consider the contraction $\sigma_i\colon X\to X_i$ of $R_{i}$ and
the contraction $X\to Z$
of the face $R_i+R_{i+1}$, which factors through $\sigma_i$:
 \begin{equation}\label{MMP}
  X\stackrel{\sigma_{i}}{\la} X_{i}\stackrel{\xi}{\la}Z.
\end{equation}
Let us show that this is a special MMP of type $(b)$ for $-E_1$ (see p.\ \pageref{specialMMP}), contradicting Assumption \ref{assumption}.
Note that  $E_0\cap E_1=\emptyset$ and $E_1\cap E_2=\emptyset$ by \ref{previous}$(\ref{c})$, and $E_0\neq E_2$ because $E_2\cap E_3=\emptyset$ while $E_0\cdot R_3>0$ still by \ref{previous}$(\ref{c})$, therefore $\de(E_1)=2$ (see Remark \ref{easy}).

First of all
$\sigma_{i}$ is birational of type $(n-1,n-2)^{sm}$ (see \ref{previous}$(\ref{b})$), in particular it is $K$-negative, and $X_i$ is smooth. Moreover
 \begin{equation}\label{terza}
  E_1\cdot R_i>0\quad\text{and}\quad E_1\cdot R_{i+1}\geq 0,
\end{equation}
 by \eqref{fabri} and \ref{previous}$(\ref{c})$, because $i\in\{3,\dotsc,m-1\}$.

Concerning the second step of the MMP, we have $[\sigma_i(e_{i+1})]\in\NE(\xi)$, hence $\sigma_{i}(E_{i+1})$ is a divisor contained in $\Exc(\xi)$. Using \eqref{prima} and \eqref{seconda}:
\begin{align*} \sigma_{i}(E_{i+1})\cdot\sigma_{i}(e_{i+1})&=\sigma_{i}^*\bigl(\sigma_{i}(E_{i+1})
                                                            \bigr)\cdot e_{i+1}\\
&  =\bigl(E_{i+1}+(E_{i+1}\cdot e_{i})E_i\bigr)\cdot e_{i+1}=(E_{i+1}\cdot e_i)(E_i\cdot e_{i+1})-1\geq 0
  \end{align*}
  so that $\xi$ is of fiber type. By \cite[Lemma 2.8]{prokshok}, there exists a $\Q$-divisor $\Delta_i$ such that $(X_i,\Delta_i)$ is a klt log Fano, in particular $-K_{X_i}$ is big and $[-K_{X_i}]\in\Eff(X_i)$; therefore
$\xi$ is $K$-negative.
  We also have, by \eqref{prima} and \eqref{terza}:
  \begin{align*}
    \sigma_{i}(E_1)\cdot\sigma_{i}(e_{i+1})&=
                                         \sigma_{i}^*\bigl(\sigma_{i}(E_{1})\bigr)\cdot e_{i+1} =   \bigl(E_1+(E_1\cdot e_{i})E_{i}\bigr)\cdot e_{i+1}\\ & =E_1\cdot e_{i+1}+(E_1\cdot e_{i})(E_{i}\cdot e_{i+1})
> E_1\cdot e_{i+1}\geq 0,\end{align*}
so that \eqref{MMP} 
 is a special MMP for $-E_1$. Moreover by Lemma \ref{typea} it must be of type $(b)$, because there is only one birational step; this contradicts Assumption \ref{assumption}.
We conclude that $E_i\cap E_{i+1}=\emptyset$.
\end{proof}
\begin{parg}\label{positive}
  Let $i,j\in\{1,\dotsc,m\}$; we have $E_i\cdot R_j>0$ whenever $R_i\neq R_j$ and $R_i$ and $R_j$ are not adjacent. Moreover
 the divisors
  $E_1,\dotsc,E_m$ are all distinct.
\end{parg}
\begin{proof}
The first statement follows from \ref{cone_psi} and \ref{ilme}. Now consider $a,b$ with $1\leq a<b\leq m$.
If $R_a$ and $R_b$ are adjacent, then
$E_{a}\cap E_b=\emptyset$ by \ref{cone_psi}, in particular $E_{a}\neq E_b$. Otherwise, $E_a\cdot R_b>0$ by the first statement; this implies that $E_{a}\neq E_b$, because $E_b\cdot R_b<0$.
\end{proof}  
\begin{parg}\label{cone_psi2}
  For every $i=1,\dotsc,m$ we have
      $\de(E_i)=2$ and $\N(E_i,X)=E_{i-1}^{\perp}\cap E_{i+1}^{\perp}$. 
\end{parg}
\begin{proof}
  Let $i\in\{1,\dotsc,m\}$. Since $E_{i-1}\cap E_i=E_i\cap E_{i+1}=\emptyset$
  by \ref{cone_psi}, and $E_{i-1}\neq E_{i+1}$ by \ref{positive},
  the statement follows from Remark \ref{easy}.
\end{proof}
\begin{parg}\label{bergeggi}
  Suppose that $m\geq 6$, and let $i,j\in\{1,\dotsc,m\}$ be such that $\{i-1,i,i+1\}\cap\{j-1,j,j+1\}=\emptyset$.
 Then  the classes $[E_{i-1}],[E_{i+1}],[E_{j-1}],[E_{j+1}]$ are linearly dependent in $\Nu(X)$. 
\end{parg}
\begin{proof}
  We have $E_i\cdot R_j>0$ by \ref{positive}; in particular $E_i\cap E_j\neq\emptyset$ and $\N(E_j,X)=\R R_j+\N(E_i\cap E_j,X)$ (see Remark \ref{pisa0}). Since $\de(E_j)=2$ (see \ref{cone_psi2}), we get $\codim\N(E_i\cap E_j,X)\leq 3$.
  On the other hand, by \ref{cone_psi}, $E_i\cap E_j$ is disjoint from $E_{i-1}$, $E_{i+1}$, $E_{j-1}$, and $E_{j+1}$, therefore $\N(E_i\cap E_j,X)\subset (E_{i-1})^{\perp}\cap (E_{i+1})^{\perp} \cap (E_{j-1})^{\perp}\cap (E_{j+1})^{\perp}$ (see Remark \ref{disjoint}). We conclude that  the classes $[E_{i-1}],[E_{i+1}],[E_{j-1}],[E_{j+1}]$ are linearly dependent in $\Nu(X)$. 
\end{proof}  
\begin{parg}\label{old}
 Let $R_i,R_j,R_{j+1}$ be three distinct extremal rays of $\NE(\psi)$ such that
  $R_i$ is not adjacent to $R_j$ nor $R_{j+1}$, with $i,j\in\{1,\dotsc,m\}$. Then one of the following holds:
\begin{enumerate}[$(i)$]  
\item $E_{h}\cdot e_i=1$ for some $h\in\{j,j+1\}$;
\item $E_{h}\cdot e_{i}=2$ for some $h\in\{j,j+1\}$,
and the general fiber of $\psi$ is a del Pezzo surface of degree one or two;
  \item $E_{h}\cdot e_{i}=3$ for some $h\in\{j,j+1\}$, and the general fiber of $\psi$ is a del Pezzo surface of degree one;
  \item the classes $[E_{i-1}],[E_{i+1}],[E_j],[E_{j+1}]$ are linearly dependent in $\Nu(X)$.
\end{enumerate}  
\end{parg}
\begin{proof}
  Let $\pi\colon E_i\to B$ be the $\pr^1$-bundle induced by the contraction of $R_i$. By \ref{cone_psi} and \ref{positive} we have $E_j\cap E_{j+1}=\emptyset$, $E_j\cdot R_i>0$, and $E_{j+1}\cdot R_i>0$.  Then
 Lemma \ref{pisa} implies that  $R_i\not\subset \N(E_j,X)$ and $R_i\not\subset \N(E_{j+1},X)$, thus  $\pi$ is finite on $E_i\cap E_j$ and $E_i\cap E_{j+1}$.

 We apply Lemma \ref{section} to $E_i$, with the divisors $E_{j|E_i}$ and $E_{j+1|E_i}$. It implies that either the classes of these divisors are proportional in $\Nu(E_i)$, or one of these divisors
  is supported on a section of $\pi$. 

 Suppose first the classes of  $E_{j|E_i}$ and $E_{j+1|E_i}$ are proportional in $\Nu(E_i)$. Then there exists $\lambda\in\Q_{>0}$ such that $(E_j-\lambda E_{j+1})_{|E_i}\equiv 0$, so that   $(E_j-\lambda E_{j+1})\cdot C=0$ for every curve $C\subset E_i$. This implies that $\N(E_i,X)\subset (E_j-\lambda E_{j+1})^{\perp}$.
 On the other hand we have $\N(E_i,X)=E_{i-1}^{\perp}\cap E_{i+1}^{\perp}$ by \ref{cone_psi2}, therefore $[E_j]-\lambda [E_{j+1}]$ is a linear combination of $[E_{i-1}]$ and $[E_{i+1}]$ in $\Nu(X)$, and we get $(iv)$.

  Otherwise, there is $h\in\{j,j+1\}$ such that 
  $E_{h|E_i}$ is supported on a section of $\pi$. We apply Lemma \ref{t} to the two extremal rays $R_i$ and $R_h$, and get $(i)$, $(ii)$, or $(iii)$.
\end{proof}
\begin{parg}\label{birational}
The contraction of any proper face of $\NE(\psi)$ is birational. 
\end{parg}
\begin{proof}
  For extremal rays, this follows from \ref{previous}$(b)$. The 2-dimensional faces of $\NE(\psi)$ are  $R_i+R_{i+1}$  ($i\in\{1,\dotsc,m\}$),
and $E_i\cdot R_{i+1}=E_{i+1}\cdot R_i=0$
  by \ref{cone_psi}, therefore $(E_i+E_{i+1})\cdot R_i<0$ and $(E_i+E_{i+1})\cdot R_{i+1}<0$. This implies that every irreducible curve with class in the face $R_i+R_{i+1}$ is contained in $E_i\cup E_{i+1}$, hence the contraction associated to this face is birational with exceptional locus $E_i\cup E_{i+1}$.
\end{proof}
 Consider now the generic  fiber $X_\eta$ of $\psi$; it is a smooth del Pezzo surface over $K:=\C(T)$. Note that the degree $(K_{X_\eta})^2$ of $X_\eta$ is the same as the degree $(K_{X_t})^2$ of a general closed fiber $X_t$ of $\psi$ (see Lemma \ref{fulton}).
\begin{parg}\label{turkish}
We have $\rho_{X_\eta}=3$ and $\psi(E_i)=T$ for every $i=1,\dotsc,m$. Moreover the natural morphism $X_\eta\to X$ gives an isomorphism $\NE(X_\eta)\cong\NE(\psi)$, and every contraction of $X_\eta$ is the restriction of a contraction of $X$ over $T$.
\end{parg}
\begin{proof}
This follows from Lemmas \ref{degreeR}, \ref{restr}, and \ref{NE}.
\end{proof}  
      \begin{parg}\label{nocurve}
The surface  $X_{\eta}$ does not have a morphism onto a curve over $K$.
\end{parg}
\begin{proof}
 This is because the contraction of every proper face of $\NE(X_{\eta})$ is birational, by \ref{birational} and \ref{turkish}.
\end{proof}
\begin{parg}\label{list10}
  The surface $X_\eta$ is one of the 10 surfaces in Table \ref{table1}; in particular $m\in\{5,6,8\}$.
\end{parg}
\begin{proof}
  Straightforward from \ref{nocurve} and Lemma \ref{10surfaces}.
\end{proof}
We are now going to prove that the linear span of the classes $[E_1],\dotsc,[E_m]$ in $\Nu(X)$ has dimension $3$; we give different proofs depending on the value of $m$ and on the generic fiber $X_\eta$.
\begin{parg}\label{m=8}
Suppose that $m=8$. Then  the linear span of the classes $[E_1],\dotsc,[E_m]$ in $\Nu(X)$ has dimension $3$.
\end{parg}
\begin{proof}
  Note that for distinct $i,j,k\in\{1,\dotsc,m\}$, the classes $[E_i],[E_j],[E_k]$ are always linearly independent, see Remark \ref{linindep}.
  Let $W\subset\Nu(X)$ be the linear span of $[E_1],[E_3],[E_6]$.
  
  We apply \ref{bergeggi} several times:
  \begin{enumerate}[--]
    \item
  for $i=2$, $j=5$ we get that $[E_1],[E_3],[E_4],[E_6]$ are linearly dependent, thus $[E_4]\in W$;
\item
for  $i=2$, $j=7$ we get that $[E_1],[E_3],[E_6],[E_8]$ are linearly dependent, thus $[E_8]\in W$;
\item
for  $i=3$, $j=7$ we get that $[E_2],[E_4],[E_6],[E_8]$ are linearly dependent, thus $[E_2]\in W$;
\item
for  $i=3$, $j=8$ we get that $[E_2],[E_4],[E_7],[E_1]$ are linearly dependent, thus $[E_7]\in W$;
  \item
for   $i=2$, $j=6$
 we get that $[E_1],[E_3],[E_5],[E_7]$ are linearly dependent, thus $[E_5]\in W$ and we have the statement.\qedhere
\end{enumerate}
\end{proof}
\begin{prg}\label{tenerife}
  Suppose that $m=6$. Then each of the following sets are linearly dependent:
  $$\{[E_{1}],[E_{3}],[E_{4}],[E_{6}]\},\quad \{[E_{1}],[E_{2}],[E_{4}],[E_{5}]\},
  \quad
   \{[E_{2}],[E_{3}],[E_{5}],[E_{6}]\}.$$
This is straightforward from \ref{bergeggi}.
\end{prg}
\begin{parg}\label{milano}
  Suppose that $X_\eta\cong\Bl(Y_6;2,2)$ or $X_\eta\cong\Bl(Y_3;1,1)$; in particular $m=6$.
  Then the linear span of the classes $[E_1],\dotsc,[E_6]$ in $\Nu(X)$ has dimension $3$.
  \end{parg}  
\begin{proof}
 By \cite[Figures 29 and 37]{BFSZ} we see that all the contractions associated to the $2$-dimensional faces of $\NE(X_\eta)$ are of the same type, and express $X_\eta$ as the blow-up of some $Y_6$ or $Y_3$.
 
  Let us consider the contraction $\alpha\colon X\to Z$ of the face $R_1+R_2$ of $\NE(\psi)$; there is a contraction $\pi\colon Z\to T$ such that $\psi=\pi\circ\alpha$:
   $$\xymatrix{X\ar[r]_{\alpha}\ar@/^{1pc}/[rr]^{\psi}
    &{Z}\ar[r]_{\pi}&T
  }$$
  The map $\alpha$ is birational with exceptional locus $E_1\cup E_2$ (see \ref{birational}), and since $R_1$ and $R_2$ are of type $(n-1,n-2)^{sm}$ (see \ref{previous}$(b)$) and $E_1\cap E_2=\emptyset$ (see \ref{cone_psi}), $Z$ is smooth and $\alpha$ is the blow-up of two disjoint smooth, irreducible, codimension $2$ subvarieties $A_1,A_2\subset Z$, with $A_i=\alpha(E_i)$ for $i=1,2$. Since $\psi(E_i)=T$ (see \ref{turkish}), we have $\pi(A_i)=T$.

  We show that $\pi$ is finite on $A_i$ for $i=1,2$. For simplicity let $i=1$. 
  We have $\de(E_1)=2$ and $\N(E_1,X)=E_6^{\perp}\cap E_2^{\perp}$ (see \ref{cone_psi2}), so that $[e_2]\not\in\N(E_1,X)$ and $\N(E_1,X)\cap\ker\alpha_*=\R[e_1]$. Since $\N(A_1,Z)=\alpha_*(\N(E_1,X))$, we get $\dim\N(A_1,Z)=\dim\N(E_1,X)-1=\rho_X-3=\rho_T$. On the other hand $\pi_*(\N(A_1,Z))=\N(T)$, therefore $\N(A_1,Z)\cap\ker\pi_*=\{0\}$, and $\pi$ is finite on $A_1$.

  Over the generic point $\eta$ of $T$, $\alpha$ restricts to a blow-up $X_\eta\to Z_\eta$, with $Z_\eta\cong Y_6$ or $Z_\eta\cong Y_3$. Therefore the general closed fiber $Z_t$ of $\pi$ is a (complex) del Pezzo surface of degree $6$ or $3$ (see Lemma \ref{fulton}), and contains a $(-1)$-curve, namely a smooth rational curve $\ell\subset Z_t$ with $-K_Z\cdot\ell=-K_{Z_t}\cdot\ell=1$.

  Let us now apply to $\ell$ some results on deformations of rational curves. Let $V$ be an irreducible component of $\Rat(Z)$ parametrizing $\ell$, see \cite[\S II.2]{kollar}. Note that every curve parametrized by $V$ is an irreducible and reduced rational curve algebraically equivalent to $\ell$, hence numerically equivalent to $\ell$, and contracted by $\pi$.
  Let $\Lo(V)\subset Z$ be the union 
  of the curves parametrized by $V$. Since $-K_Z$ is $\pi$-ample and  $-K_Z\cdot\ell=1$, $V$ is proper, and $\Lo(V)$ is an irreducible subvariety of $Z$. Note also that $\Lo(V)\subsetneq Z$, otherwise  the general member of $V$ should be a free curve and would have anticanonical degree $\geq 2$.

  We have $\dim V\geq n-2$ from the standard lower bound on $\dim\Hom(\pr^1,Z)$, see \cite[Theorems II.1.2 and II.2.15]{kollar}.
  By \cite[Proposition IV.2.5]{kollar} we also have, for $z\in\Lo(V)$ general:
  $$\dim\Lo(V)+\dim\Lo(V_z)=\dim V+2\geq n,$$
  where $\Lo(V_z)\subset Z$ is the union of the curves of the family $V$ that contain $z$, and is again a closed subset of $Z$. Since $\Lo(V_z)$ is contained in the fiber of $\pi$ containing $z$, and $\psi$ is equidimensional (see \ref{previous}$(a)$), we have $\dim\Lo(V_z)\leq 2$, and we get $\dim\Lo(V)\geq n-2$.
Moreover 
if $\dim\Lo(V)=n-2$, then $\dim\Lo(V_z)=2$, and by semicontinuity  $\dim\Lo(V_{z'})=2$ for every $z'\in\Lo(V)$. However this is impossible, because for $z'\in\ell$ we get 
 $\Lo(V_{z'})=Z_t$, and
 the del Pezzo surface $Z_t$ contains only finitely many rational curves of anticanonical degree one. 
We conclude that $\Lo(V)\subset Z$ is a prime divisor.

We show that $\Lo(V)\cap A_i=\emptyset$ for $i=1,2$. Indeed otherwise there exists a curve $\ell_0$ of the family $V$ such that $A_i\cap\ell_0\neq\emptyset$. On the other hand $\ell_0\not\subset A_i$, because $\pi$ is finite on $A_i$. Then if $\tilde\ell_0\subset X$ is the transform of $\ell_0$, we have $-K_X\cdot\tilde\ell_0< -K_Z\cdot\ell_0=1$, a contradiction.

Let $D\subset X$ be the transform of $\Lo(V)$; then $D\cap(E_1\cup E_2)=\emptyset$. Thus
for $i=1,2$ we have
$\N(E_i,X)\subset D^{\perp}$ (see Remark \ref{disjoint}) and $\N(E_i,X)=E_{i-1}^{\perp}\cap E_{i+1}^{\perp}$ (see \ref{cone_psi2}), therefore $[D]$ is a linear combination of $[E_{i-1}]$ and $[E_{i+1}]$. We conclude that 
 $[E_1],[E_2],[E_3],[E_6]$ are linearly dependent, and together with \ref{tenerife} this gives the statement. 
\end{proof}     
 \begin{parg}\label{inttables}
Suppose that $m\in\{5,6\}$ and that $X_\eta\not\cong\Bl(Y_6;2,2)$, $X_\eta\not\cong\Bl(Y_3;1,1)$. Up to  a cyclic permutation of the $E_i$'s, the possible intersection tables $E_i\cdot e_j$ in $X$, for $i,j=1,\dotsc,m$, are given in the tables below, depending on the generic fiber $X_\eta$ (see Table \ref{table1}).
\end{parg}  

{\tiny
$$\begin{array}{|c||r|r|r|r|r|}
   \hline
   
 X_\eta=\Bl(\pr^2_K;2,3)  & E_1 & E_2 & E_3 & E_4 & E_5 \\
      \hline\hline

e_1 &       -1 & 0 & 1 & 1 & 0 \\

        \hline

e_2 &   0 &    -1 & 0 & 1 & 2 \\

       \hline

e_3 &     2 & 0 & -1 & 0 & 3 \\
       
  \hline

e_4 &      2 & 3 & 0 & -1  & 0 \\

        \hline

e_5 &      0  & 2 & 1 & 0  & -1\\
      
\hline
  \end{array}\qquad  
\begin{array}{|c||r|r|r|r|r|r|}
   \hline
   
X_\eta= \Bl(\pr^2_K;2,5) & E_1 & E_2 & E_3 & E_4 & E_5 & E_6\\
      \hline\hline

e_1 &       -1 & 0 & 1 & 3 & 5 & 0 \\

        \hline

e_2 &   0 &    -1 & 0 & 2 & 6 & 1 \\

       \hline

e_3 &     2 & 0 & -1 & 0 & 5 & 2 \\
       
  \hline

e_4 &      3 & 5 & 0 & -1  & 0 & 1\\

       \hline

e_5 &      2 & 6 & 1 & 0  & -1 & 0\\
      
        \hline

e_6 &      0  & 5 & 2 & 2 & 0  & -1\\
      
\hline
  \end{array}$$     

 \medskip
     
   $$\begin{array}{|c||r|r|r|r|r|r|}
   \hline
   
 X_\eta=\Bl(Y_9;3,3) & E_1 & E_2 & E_3 & E_4 & E_5 & E_6\\
      \hline\hline

e_1 &       -1 & 0 & 2 & 3 & 2 & 0 \\

        \hline

e_2 &   0 &    -1 & 0 & 2 & 3 & 2 \\

       \hline

e_3 &     2 & 0 & -1 & 0 & 2 & 3 \\
       
  \hline

e_4 &      3 & 2 & 0 & -1  & 0 & 2\\

       \hline

e_5 &      2 & 3 & 2 & 0  & -1 & 0\\
      
        \hline

e_6 &      0  & 2 & 3 & 2 & 0  & -1\\
      
\hline
     \end{array}\qquad\begin{array}{|c||r|r|r|r|r|r|}
   \hline
   
X_\eta=\Bl(Y_8;3,3)  & E_1 & E_2 & E_3 & E_4 & E_5 & E_6\\
      \hline\hline

e_1 &       -1 & 0 & 1 & 4 & 3 & 0 \\

        \hline

e_2 &   0 &    -1 & 0 & 3 & 4 & 1 \\

       \hline

e_3 &     3 & 0 & -1 & 0 & 3 & 2 \\
       
  \hline

e_4 &      4 & 3 & 0 & -1  & 0 & 1\\

       \hline

e_5 &      3 & 4 & 1 & 0  & -1 & 0\\
      
        \hline

e_6 &      0  & 3 & 2 & 3 & 0  & -1\\
      
\hline
  \end{array}$$ }

\bigskip

\begin{proof}
We prove the statement 
in the case $X_\eta=\Bl(\pr^2_K;2,3)$,
the other cases being analogous.

Let $C_1$ and $C_2$ be the exceptional curves of the blow-up $X_\eta\to\pr^2_K$, respectively over the point of degree $2$ and of degree $3$, and let $H$ be the pullback of $\ol_{\pr^2_K}(1)$. We have $(C_1)^2=-2$, $(C_2)^2=-3$, $H^2=1$, $C_1\cdot C_2=C_1\cdot H=C_2\cdot H=0$.

We see from  \cite[Fig.~16]{BFSZ} that the cone $\NE(X_\eta)$ has $m=5$ extremal rays, generated by the classes $C_1,C_2,H-C_1,2H-C_1-C_2,3H-2C_2$ (the classes of the $5$ exceptional curves in $X_\eta$, ordered in such a way that two consecutives ones span a $2$-dimensional face of $\NE(X_\eta)$). We get the following intersection table in $X_\eta$:

{\tiny
$$\begin{array}{|c||r|r|r|r|r|}
   \hline
   
 \Bl(\pr^2_K;2,3)  & C_1 & C_2 & H-C_1 & 2H-C_1-C_2 & 3H-2C_2 \\
      \hline\hline

C_1 &       -2 & 0 & 2 & 2 & 0 \\

        \hline

C_2 &   0 &    -3 & 0 & 3 & 6 \\

       \hline

H-C_1 &     2 & 0 & -1 & 0 & 3 \\
       
  \hline

2H-C_1-C_2 &      2 & 3 & 0 & -1  & 0 \\

        \hline

3H-2C_2 &      0  & 6 & 3 & 0  & -3\\
      
\hline
  \end{array}$$}

\

Let us name the extremal rays of $\NE(\psi)$ as $R_1,\dotsc,R_5$ corresponding to the classes $C_1,C_2,H-C_1,2H-C_1-C_2,3H-2C_2$ in the same order, under the natural isomorphism $\NE(X_\eta)\cong\NE(\psi)$ (see \ref{turkish}). If $R_i$ corresponds to the exceptional curve $C_i$ in $X_\eta$, and $d_i$ is the degree of $R_i$ (see on p.\ \pageref{defidegree}),
we have $d_i=-C_i^2$ by Lemma \ref{degreeR}, therefore
$d_1=2$, $d_2=3$, $d_3=d_4=1$, and $d_5=3$.

Let $X_t$ be the fiber of $\psi$ over a general closed point $t\in T$. For every $i,j=1,\dotsc,5$ we have $E_{i|X_t}\cdot_{X_t} E_{j|X_t}=C_i\cdot_{X_\eta} C_j$ by Lemma \ref{fulton} and $E_{j|X_t}=e_{j,1}+\cdots +e_{j,d_j}$ with $e_{j,h}\equiv e_j$ for all $h=1,\dotsc,d_j$, by  Lemma \ref{degreeR}. Therefore:
$$d_j E_i\cdot e_j=E_i\cdot(e_{j,1}+\cdots +e_{j,d_j})=E_i\cdot E_{j|X_t}=
E_{i|X_t}\cdot_{X_t} E_{j|X_t}=C_i\cdot_{X_\eta} C_j,$$
which allows to deduce the intersection table $E_i\cdot e_j$ in $X$. 
\end{proof}
\begin{parg}\label{dim3}
 The linear span of the classes $[E_1],\dotsc,[E_m]$ in $\Nu(X)$ has dimension $3$.
\end{parg}  
\begin{proof}
If $m=8$ this is shown in \ref{m=8}, while if  $X_\eta\cong\Bl(Y_6;2,2)$ or $X_\eta\cong\Bl(Y_3;1,1)$ it is shown in \ref{milano}. In the remaining cases, by \ref{list10} we have 
 $m\in\{5,6\}$ and four possible generic fibers: $\Bl(\pr^2_K;2,3)$, $\Bl(\pr^2_K;2,5)$, $\Bl(Y_8;3,3)$, $\Bl(Y_9;3,3)$. 
We prove the statement by applying repeatedly  \ref{old}
in the four cases, using
the intersection tables given in \ref{inttables}.

  Suppose that $X_\eta=\Bl(\pr^2_K;2,3)$, and consider the intersection table given in \ref{inttables}; we have $m=5$ and the degree of the del Pezzo surface is $4$, therefore in \ref{old} cases $(ii)$ and $(iii)$ do not occur.
  Let us first apply  \ref{old} with $i=3$ and $\{j,j+1\}=\{5,1\}$. We have $E_1\cdot e_3=2$ and $E_5\cdot e_3=3$,
  thus case $(i)$ of \ref{old} does not occur, and we conclude that the classes $[E_1],[E_2],[E_4],[E_5]$ are linearly dependent in $\Nu(X)$. Similarly, by applying \ref{old} with $i=4$ and $\{j,j+1\}=\{1,2\}$,
as $E_1\cdot e_4=2$ and $E_2\cdot e_4=3$,
we deduce that also the classes  $[E_1],[E_2],[E_3],[E_5]$ are linearly dependent, and we have the statement.

\medskip

In the remaining cases for $X_\eta$ we have $m=6$, therefore we already have the relations of linear dependence in \ref{tenerife}.

Suppose that $X_\eta=\Bl(\pr^2_K;2,5)$, and consider the intersection table given in \ref{inttables}; now the degree of the del Pezzo surface is $2$, therefore in \ref{old} case $(iii)$ does not occur.
We apply  \ref{old} with $i=1$ and $\{j,j+1\}=\{4,5\}$. We have $E_4\cdot e_1=3$ and $E_5\cdot e_1=5$,
thus cases $(i)$ and $(ii)$ of \ref{old} do not occur, and we conclude that the classes $[E_2],[E_4],[E_5],[E_6]$ are linearly dependent in $\Nu(X)$. Together with \ref{tenerife} this gives the statement.

\medskip

In the case where $X_\eta=\Bl(Y_8;3,3)$, again the degree is $2$; we proceed  as the previous case, by applying \ref{old} with $i=1$ and $\{j,j+1\}=\{4,5\}$, and together with \ref{tenerife} this gives the statement.

\medskip

Finally suppose that $X_\eta=\Bl(Y_9;3,3)$. Here the degree is $3$, therefore in \ref{old} cases $(ii)$ and $(iii)$ do not occur. We apply  \ref{old} with $i=1$ and $\{j,j+1\}=\{3,4\}$. We have $E_3\cdot e_1=2$ and $E_4\cdot e_1=3$,
thus case $(i)$ of \ref{old} does not occur, and we conclude that the classes $[E_2],[E_3],[E_4],[E_6]$ are linearly dependent in $\Nu(X)$. Together with \ref{tenerife} this gives the statement.
\end{proof}
\begin{remark}\label{long}
The argument used in the proof of \ref{dim3} can be adapted to all of the 10 surfaces in Table \ref{table1}, except $\Bl(Y_3;1,1)$: for this reason we need a different argument in \ref{milano}.
\end{remark}  
Set $L:=E_1^{\perp}\cap\cdots\cap E_m^{\perp}\subset\N(X)$.
\begin{parg}\label{Eproduct}
  For every $i=1,\dotsc,m$ we have $E_i\cong\pr^1\times B$ and $e_i=\pr^1\times\{pt\}$. Moreover, for every $i\neq j$ such that
  $R_i$ and $R_j$ are not adjacent, we have  
  $E_i\cap E_j=\{pts\}\times B$ in both $E_i$ and $E_j$, and $\N(E_i\cap E_j,X)=L$.
\end{parg}
\begin{proof}
  We have $\codim L=3$ by \ref{dim3}, so that $L=E_a^\perp\cap E_b^\perp\cap E_c^\perp$ for every $a<b<c$ (see Remark \ref{linindep}).

  Let $i,j\in\{1,\dotsc,m\}$ be such that $R_i\neq R_j$ and $R_i$ and $R_j$ are not adjacent, namely $j\neq i-1,i,i+1$. Note that $R_{i-1},R_{i+1},R_{j-1},R_{j+1}$ contain at least three distinct extremal rays, because $m\geq 5$ (see \ref{list10}).

  By \ref{cone_psi} $E_i\cap E_j$ is disjoint from all four divisors 
  $E_{i-1},E_{i+1},E_{j-1},E_{j+1}$, hence  $\N(E_i\cap E_j,X)\subset E_{i-1}^{\perp}\cap E_{i+1}^{\perp}\cap E_{i-1}^{\perp}\cap E_{i+1}^{\perp}=L$ (see Remark \ref{disjoint}); 
  this implies that   $\N(E_i\cap E_j,X)\subset E_{i}^{\perp}\cap E_{j}^{\perp}$.

  Let us also note that  $E_j\cdot R_i>0$ (see  \ref{positive}) and $\dim\N(E_i,X)=\rho_X-2$ (see \ref{cone_psi2}),
therefore by Remark \ref{pisa0}
we get $\dim\N(E_i\cap E_j,X)\geq\dim\N(E_i,X)-1=\rho_X-3$. On the other hand
$\N(E_i\cap E_j,X)\subset L$ and $\dim L=\rho_X-3$, and we conclude that $\N(E_i\cap E_j,X)=L$.
We also have
  $$\N(E_i\cap E_j,X)\subset \N(E_i,X)\cap E_j^{\perp}\subsetneq \N(E_i,X)$$
  (because $R_i\subset\N(E_i,X)$ and $R_i\not\subset E_j^{\perp}$), therefore $\N(E_i\cap E_j,X)= \N(E_i,X)\cap E_j^{\perp}$.
  Now by applying Lemma \ref{varigotti} with  $D=E_j$ and $E=E_i$ we get the statement.
\end{proof}
\begin{parg}\label{product}
We have $X\cong S\times T$ where $S=\Bl_{2 pts}\pr^2$, and $\psi\colon X\to T$ is the projection.
\end{parg}
\begin{proof}
  For every $i=1,\dotsc,m$ set $B_i:=\{pt\}\times B\subset E_i$. Then $\N(B_i,X)=L$, indeed if $j\in\{1,\dotsc,m\}$ is such that $R_j$ is not adjacent to $R_i$, then $\N(B_i,X)=\N(E_i\cap E_j,X)=L$ by  \ref{Eproduct}. In particular, since $E_i\cong \pr^1\times B$, we have:
  \begin{equation}\label{NE_E_i}
\NE(E_i,X)=R_i+\NE(B_i,X)\subset R_i+L.
\end{equation}

Recall that $m\geq 5$ (see \ref{list10}), hence $R_4$ is not adjacent to $R_1$ nor $R_2$, and we have $E_4\cdot R_i>0$ and $E_i\cdot R_4>0$ for $i=1,2$, by \ref{positive}.

Set $H:=2E_1+2E_2+3E_4$. We proceed similarly to \cite[\S 3.2.21]{codim}.
For $i=1,2$ we have
$H\cdot e_i=3E_4\cdot e_i-2>0$, and $H\cdot e_4=2E_1\cdot e_4+2E_2\cdot e_4-3>0$.

We also have $L=E_1^{\perp}\cap E_2^{\perp}\cap E_4^{\perp}\subset H^{\perp}$.

Let $C\subset X$ be an irreducible curve with $C\subset\Supp H=E_1\cup E_2\cup E_4$. Then $C\subset E_j$ for some $j\in\{1,2,4\}$, hence $[C]\in R_j+L$ (see \eqref{NE_E_i}) and $H\cdot C\geq 0$. This implies that $H$ is nef and defines a contraction $\xi\colon X\to S$ such that $\NE(\xi)=\NE(X)\cap H^{\perp}$.
$$\xymatrix{
& X\ar[dl]_{\xi}\ar[dr]^{\psi}&  \\
S && T
}$$

Let $j\in\{1,2,4\}$. Since $\N(B_j,X)=L\subset H^{\perp}$, for every curve $C'\subset B_j$ we have $H\cdot C'=0$ and hence $\xi(C')=\{pt\}$. Then
the image $\xi(B_j)$ is a point, and
$\xi(E_j)=\xi(e_j)$
is an irreducible rational curve (because $H\cdot e_j>0$). Therefore
 $\xi_{|E_j}\colon E_j\to \xi(e_j)$ factors through the
projection $E_j\to\pr^1$.
In particular $\dim\xi(\Supp H)=1$, hence $S$ is a surface, because $H$ is the pullback of a divisor on $S$ (see Lemma \ref{pullback}).

Let us show that
 \begin{equation}\label{disp}
\NE(\xi)=L\cap\NE(X).\end{equation}
We already have $\NE(\xi)=H^{\perp}\cap\NE(X)\supseteq
L\cap\NE(X)$. Conversely, let 
 $C''\subset X$ be an irreducible curve such that $\xi(C'')=\{pt\}$,
 \emph{i.e.}\ $H\cdot C''=0$.

If $C''$ is disjoint from
$\Supp H=E_1\cup E_2\cup E_{4}$, then $C''\cdot E_j=0$ for
$j=1,2,4$, hence $[C'']\in L$.  

If instead  $C''$ intersects $E_1\cup E_2\cup E_4$
  then it must be contained in it,
and we have $C''\subset E_j$ for some $j\in\{1,2,4\}$. Since $\xi_{|E_j}$ factors as
the projection onto $\pr^1$ followed by a finite map, we get $C''\subset
B_j$, and again $[C'']\in\N(B_j,X)=L$.
Therefore we have \eqref{disp}.

In particular for every $j=1,2,4$ we have $\NE(\xi)\subseteq
E_j^{\perp}$, therefore  $E_j=\xi^*(\xi(E_j))$.

\medskip

Let $\pi\colon X\to S\times T$ be the morphism induced by $\xi$ and
$\psi$. We have $\ker\psi_*=\R(R_1+R_2+R_4)$, and it is easy to see that 
$L\cap \ker\psi_*=\{0\}$.  Moreover 
$\ker\xi_*\subseteq L$ by
\eqref{disp}, therefore $\pi$ is finite. 

In particular, $\xi$ must be equidimensional, hence $S$ is smooth by \cite[Lemma 3.10(ii)]{fanos}. Let $j\in\{1,2,4\}$. Since both $S$ and 
$E_j$ are smooth, \cite[Remark 3.1.1(4)]{codim} yields that
$\xi(E_j)$ is a smooth curve,
therefore $\xi(E_j)\cong\pr^1$ and $\xi_{|E_j}$ is the
projection.

We want to apply \cite[Remark 3.2.23]{codim}, that states that 
 $\pi\colon X\to S\times T$ is an isomorphism if $\ker\psi_*\subset (K_{X/S})^{\perp}$.
For this
we  need to show that $K_{X/S}\cdot R_j=0$ for
$j=1,2,4$.
But since $E_j\cong\pr^1\times B$ and  $\xi_{|E_j}$ is the
projection onto $\pr^1$, we have 
$$K_{X/S}\cdot e_j=(K_{X/S})_{|E_j}\cdot e_j=K_{E_j/\xi(E_j)}\cdot e_j=0.$$
We conclude that $\pi$ is an isomorphism and $X\cong S\times T$. Moreover $S$ is a del Pezzo surface with $\rho_S=3$, namely $S=\Bl_{2 pts}\pr^2$.
\end{proof}
\begin{parg}
  We get a contradiction.
\end{parg}
Indeed, by \ref{product}, 
$X_\eta\cong\Bl(\pr^2_K;1,1)$,  against \ref{list10}.
This concludes the proof of Theorem \ref{main}.
\end{proof}

\section{The case with non-empty discriminant}\label{discriminant}
\noindent Let
$X$ be a smooth Fano variety with $\de_X=2$, and consider the rational conic bundle on $X$ given by Theorem \ref{main}:
 $$\xymatrix{ X\ar@{-->}[r]^{\xi}&{X'}\ar[r]^{\sigma}\ar[dr]_{\zeta}&{X_2}\ar[d]^{\ph}\\
    &&Y
  }$$


  In this section we study in more detail the setting where the elementary conic bundle $\ph$ is not smooth, and prove Theorem \ref{singular} from the Introduction.
  Our starting point is the following.
  \begin{proposition}[\cite{codimtwo}, Proposition 5.1 and its proof]\label{giugno}
In the setting of Theorem \ref{main}, suppose that the conic bundle $\ph$ is not smooth. 
Then $\xi$ is an isomorphism (so that $X\cong X'$) and there is a smooth $\pr^1$-fibration $\pi_Y\colon Y\to T$, with $T$ a smooth projective variety of dimension $n-2$, such that $\pi_Y$ is finite on $A_Y$:
 $$\xymatrix{X\ar[r]_{\sigma}\ar@/^{1pc}/[rr]^{\zeta}
    \ar@/_/[rrd]_{\psi}&{X_2}\ar[r]_-{\ph}&{Y}\ar[d]^{\pi_Y}\\
    &&T
  }$$
  The composition $\psi:=\pi_Y\circ\zeta\colon X\to T$ is an equidimensional, quasi-elementary fibration in del Pezzo surfaces. Moreover if
$E:=\Exc(\sigma)\subset X$ we have
  $\de(E)=2$.
\end{proposition}
\begin{proof}
The statement is in \cite[Proposition 5.1 and its proof]{codimtwo}, with the additional statement that  $\pi_Y$ is finite on $A_Y$. It is shown in [\emph{ibidem}, 5.13] that $\pi_Y(A_Y)=T$ and $\dim\N(A_Y,Y)=\rho_Y-1=\rho_T$. Because $\pi_Y(A_Y)=T$, we have $(\pi_Y)_{*}(\N(A_Y,Y))=\N(T)$, thus $\ker(\pi_Y)_*\not\subset\N(A_Y,Y)$, and $\pi_Y$ is finite on $A_Y$.
\end{proof}  
\begin{prg}\label{outline_singular}
  We are now ready to prove Theorem \ref{singular}; let us give an outline of the proof. First of all, the two divisors $A_Y$ and $\Delta_\ph$ of $Y$ are disjoint, and both horizontal for the $\pr^1$-fibration $\pi_Y\colon Y\to T$; we use this to show that $\pi_Y$ is in fact a $\pr^1$-bundle, and $A_Y$ is a section (see \ref{A_Ysection}).

  Then we consider the $3$-dimensional cone $\NE(\psi)$, and show that every extremal ray of the cone is of type $(n-1,n-2)^{sm}$ (see \ref{every_ray}). This allows to apply to the del Pezzo fibration $\psi\colon X\to T$ the results on generic fibers in \S\ref{fibrations}, so that the generic fiber $X_\eta$ of $\psi$  is a smooth del Pezzo surface over $K=\C(T)$ with $\rho_{X_\eta}=3$, $\NE(X_\eta)\cong\NE(\psi)$, and every contraction of $X_\eta$ extends to a contraction of $X$ over $T$ (see \ref{turkish_2}). By construction we know that there is a contraction $X_\eta\to\pr^1_K$ whose relative cone is generated by the classes of two exceptional curves $C$, $\wi{C}$ with $C^2=\wi{C}^2=-1$ (see \ref{treno}), and this implies that $X_\eta$ is one of the 4 surfaces  $\Bl(\pr^2_K;1,1)$, $\Bl(\pr^2_K;1,4)$,
$\Bl(\pr^2_K;2,2)$, or $\Bl(Y_4;1,1)$ (see \ref{list4}). We easily exclude that $X_\eta\cong\Bl(\pr^2_K;1,1)$, by showing that in this case $\ph$ should be smooth (see \ref{not_smooth}). More work is needed to exclude the case $X_\eta\cong\Bl(Y_4;1,1)$;
for this we reproduce some arguments from  \ref{milano}, \ref{Eproduct}, and \ref{product}, to show that if $X_\eta\cong\Bl(Y_4;1,1)$, then $X\cong (\Bl_{2  pts}\pr^2)\times T$, and $\psi\colon X\to T$ is the projection, which gives a contradiction (see \ref{errore}, \ref{italo}, and \ref{yoga}).
\end{prg}  
\begin{proof}[Proof of Theorem \ref{singular}]
  By assumption the conic bundle $\ph$ has non-empty discriminant divisor $\Delta_{\ph}\subset Y$. We have $A_Y\cap\Delta_{\ph}=\emptyset$ by Theorem \ref{main}, and $A_Y\sqcup \Delta_{\ph}$ is the discriminant divisor of the conic bundle $\zeta$.
  \begin{parg}\label{A_Ysection}
 In the setting of Proposition \ref{giugno}, $A_Y$ is a section of $\pi_Y$, therefore $Y=\pr_T(\ma{F})$ for some rank $2$ vector bundle $\ma{F}$ on $T$.
\end{parg}
\begin{proof}
  We note that $\Delta_{\ph}$ cannot be a section of $\pi_Y$, otherwise $\Delta_{\ph}\cong T$ is smooth, connected, and simply connected (as $T$ is rationally connected, because $X$ is). 
  However this is
impossible, because by a standard construction the conic bundle $\ph$ defines a double cover 
$\w{\Delta}_\ph\to\Delta_{\ph}$, obtained by considering the Hilbert scheme of lines in the fibers of $\ph$, see \cite[\S 1.4]{beauville} and
\cite[\S 1.17]{sarkisov}. Since $\ph$ is an elementary
contraction, this double cover is non-trivial. Moreover $\ph$ does not have non-reduced fibers, because these are precisely the fibers over $\Sing(\Delta_{\ph})$, and $\Delta_{\ph}$ is smooth. This implies that the double cover
$\w{\Delta}_\ph\to\Delta_{\ph}$
is  étale,  and we have a contradiction. Therefore $\Delta_{\ph}$ is not a section of $\pi_Y$.

Now we notice that $Y$ is rationally connected because $X$ is, hence $Y$ is simply connected.
Moreover $A_Y$ and $\Delta_{\ph}$ are both reduced effective divisors, they are disjoint, and $\pi_Y$ is finite on $A_Y$, hence $\pi_Y$ must be finite on $\Delta_{\ph}$ too. By Lemma \ref{section2}, 
we conclude that $A_Y$ is a section of $\pi_Y$.
\end{proof}
\begin{parg}\label{integral}
Every fiber of $\psi$ is irreducible and reduced.
\end{parg}
\begin{proof}
Let $t\in T$. We have $\Gamma_t:=(\pi_Y)^{-1}(t)\cong\pr^1$, and $\Gamma_t\not\subset\Delta_\ph$, hence the restriction $\zeta_{|\psi^{-1}(t)}\colon \psi^{-1}(t)\to \Gamma_t$ is a conic bundle with general fiber $\pr^1$, in particular it is flat. Therefore $\psi^{-1}(t)$ is reduced and irreducible by \cite[Ch.~III, Proposition 9.7]{hartshorne}.
\end{proof}  
\begin{prg}\label{cricca}
  Recall that $E=\Exc(\sigma)$; let $\wi{E}\subset X$ be the transform of $\ph^{-1}(A_Y)\subset Y$.
Then $\wi{E}\stackrel{\sigma}{\cong}\ph^{-1}(A_Y)$ and $E$ and $\wi{E}$ are exceptional $\pr^1$-bundles; let  $e\subset E$ and $\hat{e}\subset\wi{E}$ be the respective fibers.
We have
 $\NE(\zeta)=R_E+R_{\wi{E}}$ where
  $R_E:=\NE(\sigma)=\R_{\geq 0}[e]$ and
 $R_{\wi{E}}:=\R_{\geq 0}[\hat{e}]$ are extremal rays; finally $\NE(\psi)$ is a $3$-dimensional polyhedral cone.
\end{prg}
\begin{parg}\label{every_ray}
  Every extremal ray $R$ of $\NE(\psi)$ is of type $(n-1,n-2)^{sm}$.
\end{parg}
\begin{proof}
  This is clear if $R=R_E$ or $R=R_{\wi{E}}$; suppose otherwise. Then $R\not\subset\NE(\zeta)$, while $\zeta_*(R)=\zeta_*(\NE(\psi))=\NE(\pi_Y)$. If $F_R\subset X$ is a non-trivial fiber of the contraction of $R$, then $\zeta$ is finite on $F_R$, and $\zeta(F_R)$ is contained in a fiber of $\pi_Y$, therefore $\dim F_R=1$, and either $R$ is of type $(n-1,n-2)^{sm}$, or the contraction of $R$ is a conic bundle (see \cite[Theorem 1.2]{wisn}).
In this last case, we get a factorization of $\psi$ as 
$$\xymatrix{X\ar[r]_{\beta}\ar@/^{1pc}/[rr]^{\psi}&{X'}\ar[r]_{\gamma}&T}
$$
where $\beta$ is the contraction of $R$, and $\dim X'=n-1$.
Note that $X'$ is smooth, and $\gamma$ is of fiber type because $\dim T=n-2$. 
Since $\psi$ is quasi-elementary, $\gamma$ is quasi-elementary too
(see \cite[Lemma 3.15]{refinements}). On the other hand this is impossible, because the general fiber $F_{\gamma}\subset X'$ of $\gamma$ is a curve, thus
$\dim\N(F_{\gamma},X')=1$, while
 $\rho_{X'}-\rho_T=2$.
Therefore $R$ is of type $(n-1,n-2)^{sm}$.
\end{proof}
Consider now 
the generic fiber $X_\eta$ of $\psi$, a smooth del Pezzo surface over the field $K:=\C(T)$ of rational functions on $T$. 
\begin{parg}\label{turkish_2}
We have $\rho_{X_\eta}=3$ and $\psi(\Lo(R))=T$ for every extremal ray $R$ of $\NE(\psi)$. Moreover we have a natural isomorphism $\NE(X_\eta)\cong\NE(\psi)$, and every contraction of $X_\eta$ is the restriction of a contraction of $X$ over $T$.
\end{parg}
\begin{proof}
This follows from \ref{every_ray} and Lemmas \ref{degreeR}, \ref{restr}, and \ref{NE}.
\end{proof} 
   \begin{parg}\label{treno}
Set $C:=E_{|X_\eta}$ and $\wi{C}:=\wi{E}_{|X_\eta}$. Then $C$ and $\wi{C}$ are exceptional curves in $X_\eta$ with $C^2=\wi{C}^2=-1$. Moreover the cone $\langle [C],[\wi{C}]\rangle$ is the relative cone of a contraction $X_\eta\to\pr^1_K$.
\end{parg}
\begin{proof}
  Let $X_t\subset X$ be the fiber of $\psi$ over a general closed point $t\in T$.
We have $E_{|X_t}=e$ and $\wi{E}_{|X_t}=\hat{e}$. Indeed $\pi_Y^{-1}(t)\cong\pr^1$ meets $A_Y$ transversally in one point ($A_Y$ being a section of $\pi_Y$, see \ref{A_Ysection}), thus $(\ph\circ\pi_Y)^{-1}(t)\subset X_2$ meets $A$ transversally in one point (because $A$ is a section of $\ph_{|\ph^{-1}(A_Y)}$, see Theorem \ref{main}). Hence $E_{|X_t}$ is a reduced fiber of $\sigma$. Similarly for $\wi{E}$.
Then  Lemma \ref{degreeR} shows that $C$ and $\wi{C}$ are exceptional curves in $X_\eta$ with $C^2=\wi{C}^2=-1$.

Moreover we have $R_E+R_{\wi{E}}=\NE(\zeta)$, where $\zeta=\ph\circ\sigma\colon X\to Y$, and the generic fiber of $\pi_Y\colon Y\to T$ is $\pr^1_K$, because $\pi_Y$ is a $\pr^1$-bundle (see \ref{A_Ysection}). Therefore $\zeta$ restricts to a contraction $\zeta_\eta \colon X_\eta\to\pr^1_K$ with $\NE(\zeta_\eta)=\langle [C],[\wi{C}]\rangle$.
\end{proof}
\begin{parg}\label{list4}
  The surface $X_\eta$ is one of the 4 surfaces in Table \ref{table2}:
  $$\text{$\Bl(\pr^2_K;1,1)$, $\quad\Bl(\pr^2_K;1,4)$,
  $\quad\Bl(\pr^2_K;2,2)$, $\quad\Bl(Y_4;1,1)$.}$$
\end{parg}
\begin{proof}
  Straightforward from \ref{treno} and Lemma \ref{4surfaces}.
  \end{proof}
 \begin{parg}\label{not_smooth}
  We have $X_{\eta}\not\cong\Bl(\pr^2_K;1,1)$.
\end{parg}
\begin{proof}
  Suppose by contradiction that
  $X_{\eta}\cong\Bl(\pr^2_K;1,1)$. By \cite[Figure 7]{BFSZ} the cone $\NE(X_\eta)\cong\NE(\psi)$ is simplicial; let $R_F$ be the third extremal ray of $\NE(\psi)$, so that $\NE(\psi)=R_E+R_{\wi{E}}+R_F$. Recall that $R_F$ is of type $(n-1,n-2)^{sm}$ (see \ref{every_ray}); set $F:=\Lo(R_F)$ and let $f\subset F$ be a non-trivial fiber of the contraction of $R_F$.

We keep the same notation as in \ref{treno},  
and we set $C_2:=F_{|X_{\eta}}$. Since the contraction of the face $\langle [C],[\wi{C}]\rangle$ is of fiber type, we can assume that the blow-up $X_\eta\to\pr^2_K$ contracts the curves $\wi{C}$ and $C_2$. If $H$  is the pullback of $\ol_{\pr^2_K}(1)$, we have
  $C\sim H-\wi{C}-C_2$.
  We get the following intersection table in $X_\eta$:
  
  {\tiny
$$\begin{array}{|c||r|r|r|}
   \hline
   
\Bl(\pr^2_K;1,1)  & C & \wi{C} & C_2 \\
      \hline\hline

C &       -1 & 1 & 1 \\

        \hline

\wi{C} &   1 &    -1 & 0 \\

       \hline

C_2 &     1 & 0 & -1 \\
       
  \hline
  \end{array}$$}
  As in the proof of \ref{inttables}, we deduce
the following intersection table in $X$:
{\tiny
$$\begin{array}{|c||r|r|r|}
   \hline
   
  & E & \wi{E} & F \\
      \hline\hline

e &       -1 & 1 & 1 \\

        \hline

\hat{e} &   1 &    -1 & 0 \\

       \hline

f &     1 & 0 & -1 \\
       
  \hline
\end{array}$$}

Let us consider now the prime divisor $\sigma(F)\subset X_2$. We have $\sigma^*(\sigma(F))=F+(F\cdot e)E=F+E$,
and:
$$\sigma(F)\cdot\sigma(\hat{e})=\sigma(F)\cdot\sigma_*(\hat{e})=\sigma^*(\sigma(F))\cdot\hat{e}=(F+E)\cdot\hat{e}=1.$$
On the other hand $\sigma(\hat{e})$ is a smooth fiber of $\ph$,
and this implies that $\ph$ cannot have singular fibers, contradicting $\Delta_{\ph}\neq\emptyset$.
\end{proof}
\begin{prg}\label{setnot}
We conclude that $X_\eta$ is one of the three surfaces
 $\Bl(\pr^2_K;1,4)$,
$\Bl(\pr^2_K;2,2)$, or $\Bl(Y_4;1,1)$, and  $\NE(X_\eta)\cong\NE(\psi)$ has $4$ extremal rays (see Table \ref{table2}).
Let  $R_F$ and $R_{\wi{F}}$ be the other two extremal rays of $\NE(\psi)$, of type $(n-1,n-2)^{sm}$ (see \ref{every_ray}) and with loci $F$ and $\wi{F}$ respectively, such that
$R_E+R_{F}$, $R_{\wi{E}}+R_{\wi{F}}$, and $R_F+R_{\wi{F}}$ are faces of $\NE(\psi)$ (see Figure \ref{figura_psi}).
Note that $F$ and $\wi{F}$ are exceptional $\pr^1$-bundles, like $E$ and $\wi{E}$; let $f\subset F$ and $\hat{f}\subset\wi{F}$ be fibers of the $\pr^1$-bundles. Let us also denote by $d_F$ and $d_{\wi{F}}$ the degrees of the extremal rays $R_F$ and $R_{\wi{F}}$ respectively (see Lemma \ref{degreeR}).

\begin{figure}[h]\caption{A section of $\NE(\psi)$.}\label{figura_psi}

\bigskip
  
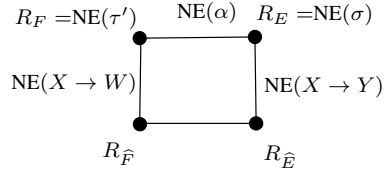
    \end{figure}
  \end{prg}
  \begin{parg}\label{tables}
We have
following the intersection tables:

  {\tiny
  $$\begin{array}{|c||r|r|r|r|}
   \hline
   
 X_\eta=\Bl(\pr^2_K;1,4) & E & \wi{E} & F & \wi{F} \\
      \hline\hline

e &       -1 & 1 & 0 &  4 \\

        \hline

\hat{e} &   1 &    -1 & 4 & 0 \\

       \hline

f &   0 & 1 & -1  & 1 \\
       
  \hline

\hat{f} &       1 & 0 & 1 & -1 \\ 

        \hline
  \end{array}
  \qquad
  \begin{array}{|c||r|r|r|r|}
   \hline
   
 X_\eta=\Bl(\pr^2_K;2,2)  & E & \wi{E} & F & \wi{F} \\
      \hline\hline

e &       -1 & 1 & 0&  2 \\

        \hline

\hat{e} &   1 &    -1 & 2 & 0 \\

       \hline

f &      0 & 1 & -1  & 0 \\
       
  \hline

\hat{f} &      1 & 0 & 0 & -1 \\

        \hline
  \end{array}$$

\smallskip
  
  $$\begin{array}{|c||r|r|r|r|}
   \hline
   
 X_\eta=\Bl(Y_4;1,1)  & E & \wi{E} & F & \wi{F} \\
      \hline\hline

e &       -1 & 1 & 0 &  2 \\

        \hline

\hat{e} &   1 &    -1 & 2 & 0 \\

       \hline

f &     0 & 2 & -1  & 1 \\ 
       
  \hline

\hat{f} &     2 & 0 & 1 & -1 \\

        \hline
    \end{array}$$}
  Moreover 
$$d_F=d_{\wi{F}}=\begin{cases}
4\quad\text{if }X_\eta=\Bl(\pr^2_K;1,4);\\
  2\quad\text{if }X_\eta=\Bl(\pr^2_K;2,2);\\
  1\quad\text{if } X_\eta=\Bl(Y_4;1,1).
                 \end{cases}$$
\end{parg}
\begin{proof}
From \cite[Figures 10, 15, 38]{BFSZ} we get the degrees $d_F$ and $d_{\wi{F}}$, and the intersection tables for exceptional curves in the three surfaces $X_\eta$; then we deduce the intersection tables above using Lemmas \ref{fulton} and \ref{degreeR}, as in the proof of \ref{inttables}.
\end{proof}
\begin{prg}\label{description}
  We have $E\cap F=\emptyset$, $\wi{E}\cap\wi{F}=\emptyset$, and $\de(\wi{E})=\de(F)=\de(\wi{F})=2$. We also have  $\N(E,X)\cap\ker\psi_*=\R[e]$, and similarly for the other divisors $\wi{E},F,\wi{F}$.

  Finally $F\cap\wi{F}=\emptyset$ if $X_\eta=\Bl(\pr^2_K;2,2)$.
\end{prg}
\begin{proof}
  We show that $E\cap F=\emptyset$ and $\N(E,X)\cap\ker\psi_*=\R[e]$. We have $\psi(E)=T$, thus $\psi_*(\N(E,X))=\N(T)$. On the other hand $\de(E)=2$ (see Proposition \ref{giugno}), thus $\dim\N(E,X)=\rho_X-2=\rho_T+1$, therefore $\dim(\N(E,X)\cap\ker\psi_*)=1$, and
  $\N(E,X)\cap\ker\psi_*=\R[e]$. In particular we deduce that $[f]\not\in\N(E,X)$. Now if $E\cap F\neq\emptyset$, since $E\cdot f=0$ (see \ref{tables}), $E$ should contain some curve numerically equivalent to $f$, a contradiction. 

\medskip 
  
  Recall that $\wi{E}\cdot e>0$ (see \ref{tables}), so that $\N(E,X)=\R[e]+\N(E\cap\wi{E},X)$ (see Remark \ref{pisa0}), and we see that
  $\dim\N(E\cap\wi{E},X)\in\{\rho_X-3,\rho_X-2\}$.
 We also have 
$E\cdot\hat{e}>0$ and 
 $[\hat{e}]\not\in\N(E,X)$ (by the first part of the proof), in particular $[\hat{e}]\not\in\N(E\cap\wi{E},X)$, so that $\N(\wi{E},X)=\R[\wi{e}]\oplus\N(E\cap\wi{E},X)$, and 
 $\dim\N(\wi{E},X)\in\{\rho_X-2,\rho_X-1\}$.

 We show that $\de(\wi{E})=2$.
 Otherwise, by what precedes, we have $\dim\N(\wi{E},X)=\rho_X-1=\rho_T+2$, $\dim\N(E\cap\wi{E},X)=\rho_X-2$, and $[e]\in\N(E\cap\wi{E},X)\subset\N(\wi{E},X)$. Since $\psi(\wi{E})=T$, we have $\psi_*(\N(\wi{E},X))=\N(T)$ and hence $\dim (\N(\wi{E},X)\cap\ker\psi_*)=2$. Therefore $\N(\wi{E},X)\cap\ker\psi_*=\R[e]\oplus\R[\hat{e}]$, and $[\hat{f}]\not\in \N(\wi{E},X)$ (see Figure \ref{figura_psi}). On the other hand $\wi{E}\cdot \hat{f}=0$ by \ref{tables}, and this implies that $\wi{E}\cap \wi{F}=\emptyset$, otherwise
 $\wi{E}$ should contain some curve numerically equivalent to $\hat{f}$.
 In particular we get $\N(\wi{E},X)\subset \wi{F}^{\perp}$ (see Remark \ref{disjoint}), and therefore
  $\N(\wi{E},X)= \wi{F}^{\perp}$. Now recall that in $Y$
we have $A_Y\cap\Delta_{\ph}=\emptyset$ by Theorem \ref{main}, hence $\zeta^{-1}(A_Y)\cap\zeta^{-1}(\Delta_{\ph})=\emptyset$ in $X$, and $\zeta^{-1}(A_Y)=E\cup\wi{E}$. In particular $\wi{E}$ is disjoint also from the divisor $G:=\zeta^{-1}(\Delta_{\ph})$, and as above we get $\N(\wi{E},X)= G^{\perp}$, and finally $G^{\perp}=\wi{F}^{\perp}$. Then there exists $\lambda\in\Q_{>0}$ such that $[G]=\lambda[\wi{F}]$. However this is impossible, because $G\cap E=\emptyset$ and hence $G\cdot e=0$, while $\wi{F}\cdot e\neq 0$.
We conclude that $\de(\wi{E})=2$.

\medskip

  Now consider the exceptional $\pr^1$-bundle $\wi{E}$ with the pair of disjoint divisors $E,F$; note that $E\cdot\hat{e}>0$ and $F\cdot\hat{e}>0$ by \ref{tables}, and we can apply Lemma \ref{pisa}. We deduce that $\N(\wi{E}\cap F,X)$ has codimension one in $\N(\wi{E},X)$, therefore $\dim\N(\wi{E}\cap F,X)=\rho_X-3$. Since $\wi{E}\cdot f>0$ again by \ref{tables}, we have
$\N(F,X)=\N(\wi{E}\cap F,X)+\R[f]$ (see Remark \ref{pisa0}). Therefore $\dim\N(F,X)\leq \rho_X-2$ and finally equality holds because $\de_X=2$.

  The rest of the statement is shown in a similar way. If  $X_\eta=\Bl(\pr^2_K;2,2)$ then $F\cdot\hat{f}=\wi{F}\cdot f=0$ by  \ref{tables}, and we show as above that $F\cap\wi{F}=\emptyset$.
\end{proof}
We are left to show that    $X_{\eta}\not\cong\Bl(Y_4;1,1)$; we will proceed by contradiction.
 \begin{parg}\label{errore}
 Suppose  that
  $X_{\eta}\cong\Bl(Y_4 ;1,1)$. Then there exist  prime divisors $D,\wi{D}\subset X$ such that $D\cap(E\cup F)=\emptyset$ and  $\wi{D}\cap
(\wi{E}\cup \wi{F})=\emptyset$.
\end{parg}
\begin{proof}
Using the description in \cite[Figure 38]{BFSZ} and the isomorphism $\NE(\psi)\cong\NE(X_\eta)$, we see that the contractions of both $2$-dimensional faces $R_E+R_F$ and $R_{\wi{E}}+R_{\wi{F}}$ of $\NE(\psi)$ (see Figure \ref{figura_psi})
give, over the generic point, blow-ups $X_\eta\to Y_4$.

We proceed
as in the proof of \ref{milano}, and consider the contraction $\alpha\colon X\to Z$ of one of these faces, $R_E+R_F$ or $R_{\wi{E}}+R_{\wi{F}}$, so that we get a diagram:
   $$\xymatrix{X\ar[r]_{\alpha}\ar@/^{1pc}/[rr]^{\psi}
    &{Z}\ar[r]_{\pi}&T
  }$$
  where $Z$ is smooth, $\alpha$ is the blow-up of two disjoint smooth, irreducible subvarieties of codimension $2$ in $Z$, on which $\pi$ is finite, and the general closed fiber $Z_t$ of $\pi$ is a del Pezzo surface of degree $4$.
The exceptional divisors of $\alpha$ are $E\cup F$ or $\wi{E}\cup\wi{F}$, depending on the face  that we have chosen.

Proceeding as in the proof of \ref{milano}, we use a $(-1)$-curve $\ell$ in the general closed fiber $Z_t$ to construct prime divisors $D,\wi{D}\subset X$ as in the statement.
\end{proof}
 \begin{parg}\label{italo}
 Suppose  that
  $X_{\eta}\cong\Bl(Y_4 ;1,1)$. Then $E\cong F\cong\wi{E}\cong\wi{F}\cong\pr^1\times B$, with the curves $e,f,\hat{e},\hat{f}$ being $\pr^1\times\{pt\}$, $E\cap\wi{E}=\{pts\}\times B$ in both $E$ and $\wi{E}$, and
  $F\cap\wi{F}=\{pts\}\times B$ in both $F$ and $\wi{F}$. Moreover   $$L:=\N(E\cap\wi{E},X)=\N(F\cap\wi{F},X)= E^{\perp}\cap\wi{E}^{\perp}\cap F^{\perp}\cap\wi{F}^{\perp}.$$
\end{parg}  
\begin{proof}
  We proceed similarly to \ref{Eproduct}.
  Consider the divisor $\wi{E}$, and recall that $\de(\wi{E})=2$ (see \ref{description}). We have $E\cdot \hat{e}>0$ and $F\cdot\hat{e}>0$ (see \ref{tables}), therefore both $E$ and $F$ intersect $\wi{E}$.
Let moreover $D\subset X$ be as in \ref{errore}.
Then $D$ intersects $\wi{E}$ too, by Remark \ref{easy}. Since $D,E,F$ are pairwise disjoint, we have $D\cdot\hat{e}>0$, and moreover $\dim\N(D\cap\wi{E},X)\geq\rho_X-3$ by Remark \ref{pisa0}, and similarly for the other intersections $E\cap\wi{E}$ and $F\cap\wi{E}$.
By \cite[Lemma 3.1.7]{codim} there exists a linear subspace $L\subset \N(X)$, of codimension $3$, such that
$$L=\N(E\cap\wi{E},X)=\N(F\cap \wi{E},X)=\N(\wi{E},X)\cap E^{\perp}\subset F^{\perp}.$$
We  argue in the same way for the other divisors $\wi{F},E,F$ (using also $\wi{D}$ as in \ref{errore}), in particular by exchanging the roles of $E$ and $\wi{E}$ we get
$$L=\N(E\cap\wi{E},X)\subset\wi{E}^{\perp}.$$
By Lemma \ref{varigotti} this implies that $\wi{E}\cong\pr^1\times B$, $\wi{e}=\pr^1\times\{pt\}$, and $E\cap\wi{E}=\{pts\}\times B$.
Moreover $L\subset E^{\perp}\cap F^{\perp}\cap \wi{E}^{\perp}$, and for dimensional reasons we have equality (see Remark \ref{linindep}).
By repeating the argument for the other divisors we get the statement.
\end{proof}
\begin{parg}\label{yoga}
 Suppose  that
  $X_{\eta}\cong\Bl(Y_4 ;1,1)$.
Then $X\cong S\times T$ where $S=\Bl_{2 pts}\pr^2$,  and $\psi\colon X\to T$ is the projection.
\end{parg}
\begin{proof}
We proceed similarly to \ref{product}. Set $H:=E+F+\wi{E}+\wi{F}$. By \ref{tables} we see that $H\cdot e=H\cdot \wi{e}=H\cdot f=H\cdot\wi{f}=2$. Moreover $L=E^{\perp}\cap\wi{E}^{\perp}\cap F^{\perp}\cap\wi{F}^{\perp}\subset H^{\perp}$.

Let $C\subset X$ be an irreducible curve with $C\subset\Supp H=E\cup F\cup\wi{E}\cup\wi{F}$. Suppose that $C\subset E$, the other cases being analogous, and set $B_E:=\{pt\}\times B\subset E$. Then $[C]\in\NE(E,X)=R_E+\NE(B_E,X)\subset R_E+L$, hence $H\cdot C\geq 0$. We conclude that $H$ is nef and defines a contraction $\xi\colon X\to S$ 
 such that $\NE(\xi)=\NE(X)\cap H^{\perp}$.
$$\xymatrix{
& X\ar[dl]_{\xi}\ar[dr]^{\psi}&  \\
S && T
}$$

As in the proof of \ref{product}, we see that $\xi_{|E}\colon E\to\xi(e)$ factors through the projection $E\to\pr^1$, and similarly for the other divisors $\wi{E},F,\wi{F}$; in particular $\dim\xi(\Supp(H))=1$ and $\dim S=2$. We also see that $\NE(\xi)=L\cap\NE(X)$, and  that $\xi$ and $\psi$ induce a finite morphism $\pi\colon X\to S\times T$.
Finally as in the proof of \ref{product} we see that $\pi$ is an isomorphism.
\end{proof}
By \ref{yoga},
we conclude that $X_\eta\not\cong\Bl(Y_4;1,1)$. Indeed if $X\cong S\times T$, then $X_\eta\cong \Bl(\pr^2_K;1,1)$, and we get a contradiction. Therefore by \ref{list4} and \ref{not_smooth} $X_\eta$ 
is 
 $\Bl(\pr^2_K;1,4)$ or
$\Bl(\pr^2_K;2,2)$, and this concludes the proof of Theorem \ref{singular}.
\end{proof}
\begin{remark}
We will study in detail the case $X_\eta=\Bl(\pr^2_K;1,4)$ in the next sections. For the other case  $X_\eta=\Bl(\pr^2_K;2,2)$, we describe in the figure below the generic fibers of the targets of the contractions of the $2$-dimensional faces of $\NE(\psi)$, using the isomorphism $\NE(\psi)\cong\NE(X_\eta)$ and the description in \cite[Figure 15]{BFSZ}. Here $Y_8$ is a del Pezzo surface over $K=\C(T)$ with $\rho_{Y_8}=1$ and degree $8$.
\begin{center}
  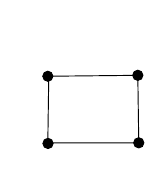
\end{center}
 \end{remark}

\section{The case $X_\eta\cong\Bl(\pr^2_K;1,4)$}\label{section_rho=6}
\noindent
\begin{prg}\label{outline_rho=6}
In this section we prove Theorem \ref{rho=6} from the Introduction; let us give an outline of the proof.

First of all we consider the map $\pi_Y\circ\ph\colon X_2\to T$ and its second factorization in elementary contractions: $X_2\stackrel{\tau}{\to} Z\stackrel{\pi_Z}{\to} T$ (see diagram \eqref{firenze}). We show that
$Z$ is smooth and 
$\alpha:=\tau\circ\sigma\colon X\to Z$ is the blow-up of two disjoint smooth, irreducible, codimension $2$ subvarieties $A_Z,B_Z$ of $Z$, where $A_Z=\tau(A_2)$ (see \ref{alpha}; recall that $A_2$ is the center of the blow-up $\sigma$).
We also show that $\pi_Z\colon Z\to T$ is a $\pr^2$-bundle, and $A_Z$ is a section, while $B_Z$ is finite of degree $4$ onto $T$ (see \ref{centers} -- \ref{gianfranco}).
The rest of the statement follows by studying carefully the different factorizations and the involved morphisms.  To show that $Z\cong\pr_T(\ol\oplus\ol(D)\oplus\ol(2D))$, we use some results by Casnati and Ekedahl on degree $4$ covers, applied to $\pi_{Z|B_Z}\colon B_Z\to T$ (see \ref{trace}, \ref{D1}, and \ref{coker}).
\end{prg}
We keep the same notation as in Section \ref{discriminant}.
\begin{parg}\label{alpha}
  We have a commutative diagram:$\,$\footnote{In this Section's diagrams, for the reader's convenience, for some blow-up map we specify the exceptional divisors, \emph{e.g.}\ $\sigma=\sigma_E$.}
  \begin{equation}\label{firenze}
\xymatrix{ X\ar@/^{1pc}/[rr]^{\alpha=\alpha_{E,F}}\ar[r]_{\sigma=\sigma_E}&{X_2}\ar[d]_{\ph}
  \ar[r]_{\tau}&{Z}\ar[d]^{\pi_Z} \\
 & {Y}\ar[r]^{\pi_Y}&{T}
}
\end{equation}
where $Z$ is smooth and
$\alpha$ is the blow-up of two disjoint smooth, irreducible subvarieties 
$A_Z,B_Z\subset Z$ of codimension two, where $A_Z$ is the image of $A_2\subset X_2$; the exceptional divisors of $\alpha$ are $E$ (over $A_Z$) and $F$ (over $B_Z$).
\end{parg}
\begin{proof}
We have $\NE(\ph\circ\sigma)=R_E+R_{\wi{E}}$; let $\alpha\colon X\to Z$ be the contraction over $T$ with $\NE(\alpha)=R_E+R_{F}$ (see Figure \ref{figura_psi}), then $\alpha$ factors through $\sigma$ because $\NE(\sigma)=R_E$. Since both rays $R_E$ and $R_{F}$ are of type $(n-1,n-2)^{sm}$ (see \ref{every_ray}), with disjoint loci $E$ and $F$ (see  \ref{description}), we get the statement.
 \end{proof}
  \begin{parg}\label{centers}
    $A_Z$ is a section of $\pi_Z$.
 \end{parg}  
\begin{proof}
  Since $A_2$ is a section of $\ph_{|\ph^{-1}(A_Y)}$ (see Theorem \ref{singular}), and $A_Y$ is a section of $\pi_Y$ (see \ref{A_Ysection}), we see that $A_2$ is a section of $\pi_Y\circ\ph\colon X_2\to T$. Moreover $A_2=\sigma(E)$ is disjoint from the exceptional divisor $\sigma(F)$ of the blow-up $\tau\colon X_2\to Z$ of $B_Z$, therefore its image $A_Z\subset Z$ is a section of $\pi_Z$.
\end{proof}
 \begin{parg}\label{vbE}
We have $Z=\pr_T(\ma{E})$ where $\ma{E}$ is a rank $3$ vector bundle on $T$, sitting in an extension:
\begin{equation}\label{sabri}
  0\la \ma{N}^{\vee}_{A_Z/Z}\la\ma{E}\la\ol_T\la 0
\end{equation}
given by the section $A_Z$ of $\pi_Z$. Let $H_Z$ be a tautological divisor in $Z$; we have $\ol_Z(H_Z)_{|A_Z}\cong\ol_T$.
\end{parg}
\begin{proof}
 The map $\pi_Z\colon Z\to T$ is an elementary contraction with generic fiber $\pr^2_K$, thus the fiber over a general closed point is $\pr^2$. 

 We note that $\pi_Z$ is equidimensional with irreducible fibers. 
 Indeed this holds for $\psi=\pi_Z\circ\alpha$ (see \ref{integral}), and every fiber of $\pi_Z$ is the image, under $\alpha$, of a fiber of $\psi$.
 
We claim that $\pi_Z$ is smooth. Indeed it is a flat morphism, between smooth varieties, with general fiber $\pr^2$; if it were not smooth, by \cite[Theorem 4]{carolinaflat} it should have reducible fibers, which is not the case. 
Since moreover $\pi_Z$ has a section, we conclude that  
 $Z=\pr_T(\ma{E})$ for some rank $3$ vector bundle $\ma{E}$ on $T$, and we can choose $\ma{E}$ such that the section $A_Z$ gives the exact sequence \eqref{sabri}.
\end{proof}
\begin{parg}\label{gianfranco}
Set  $\wi{E}_Z:=\alpha(\wi{E})\subset Z$.
We have $\wi{E}\stackrel{\alpha}{\cong}\wi{E}_Z$, the divisor $\wi{E}_Z\subset Z$ is  smooth, contains both $A_Z$ and $B_Z$, and $\pi_{Z|\wi{E}_Z}\colon \wi{E}_Z\to T$ is a smooth 
conic bundle. Moreover $\pi_B:=\pi_{Z|B_Z}\colon B_Z\to T$ is finite of degree $4$.
\end{parg}
For $t\in T$ general, in $Z_t:=\pi_Z^{-1}(t)\cong\pr^2$, $\wi{E}_Z\cap Z_t$ is the conic through the $5$ points $(A_Z\cup B_Z)\cap Z_t$, see Figure \ref{figura_Z_t}.

\begin{figure}[h]\caption{The projective plane $Z_t\cong\pr^2$ for $t\in T$ general.}\label{figura_Z_t}

\bigskip
  
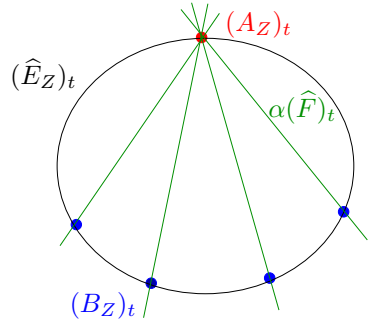
    \end{figure}

 \begin{proof}
   We have $\wi{E}\cdot e=\wi{E}\cdot f=1$ (see  \ref{tables}), hence  $\wi{E}_Z$ contains both $A_Z$ and $B_Z$. Moreover
$\N(\wi{E},X)\cap\ker\psi_*=\R[\hat{e}]$ by  \ref{description}, and $\NE(\alpha)=\langle [e],[f]\rangle\subset\ker\psi_*$ (see \ref{alpha}), therefore $\N(\wi{E},X)\cap\NE(\alpha)=\{0\}$, and $\alpha$ is finite on $\wi{E}$. We conclude that
$\wi{E}\stackrel{\alpha}{\cong}\wi{E}_Z$ and $\wi{E}_Z$ is smooth.

We also have $\wi{E}\stackrel{\sigma}{\cong}\ph^{-1}(A_Y)\subset X_2$ (see \ref{cricca}), and $A_Y$ is a section of $\pi_Z$, therefore
 $\pi_{Z|\wi{E}_Z}\colon \wi{E}_Z\to T$ is isomorphic to $\ph_{|\ph^{-1}(A_Y)}\colon\ph^{-1}(A_Y)\to A_Y$ (see diagram \eqref{firenze}), and is a smooth $\pr^1$-fibration. Let $\hat{e}_Z:=\alpha(\hat{e})$ be its fiber.
 Since $-K_X\cdot\hat{e}=1$, $E\cdot\hat{e}=1$, and $F\cdot\hat{e}=4$ (see \ref{tables}), we have
  $$-K_Z\cdot\hat{e}_Z=-K_Z\cdot\alpha_*(\hat{e})=
  \alpha^*(-K_Z)\cdot\hat{e}=(-K_X+E+F)\cdot\hat{e}=6,$$ therefore $\hat{e}_Z$ is a smooth conic in a fiber of $\pi_Z$, and
we conclude that $\wi{E}_Z\subset Z$ is a smooth conic bundle.
   
   Now  $\pi_{B}=\pi_{Z|B_Z}\colon B_Z\to T$ is generically finite of degree $4$, because the generic point of $B_Z$ has degree $4$ in $Z_\eta\cong\pr^2_K$.
   Since $\pi_{Z|\wi{E}_Z}$ is finite on $A_Z$ and $A_Z\cap B_Z=\emptyset$, $B_Z$ cannot contain any fiber of $\pi_{Z|\wi{E}_Z}$, and $\pi_Z$ must be finite on $B_Z$ too.
\end{proof} 
\begin{prg}
  Let $t\in T$ (a closed point), and consider the fibers over $t$ of
$\pi_Y\colon Y\to T$, $\pi_Y\circ\ph\colon X_2\to T$, and $\pi_Z\colon Z\to T$:
  $$\xymatrix{ {(X_2)_t}\ar[d]_{\ph_t}\ar[r]^{\tau_t}&{Z_t\cong\pr^2} \\
 {Y_t\cong\pr^1}&
}$$
Here $\tau_t\colon (X_2)_t\to\pr^2$ is the blow-up of the zero-dimensional subscheme $(B_Z)_t\subset\pr^2$, of length $4$, $\ph_t$ is the conic bundle given by the pencil of conics through $(B_Z)_t$, and $(\Delta_{\ph})_t\subset\pr^1$ is the discriminant of such conic bundle. If $t$ is general, then $(B_Z)_t$ is reduced and contains $4$ points in general linear position,\footnote{Because the blow-up of $\pr^2$ along $(B_Z)_t$ is a del Pezzo surface.} therefore $(\Delta_{\ph})_t$ is reduced of length $3$. Thus $\pi_{Y|\Delta_\ph}\colon\Delta_\ph\to T$ has degree $3$.

We remark that the finite covers $\pi_{B}\colon B_Z\to T$ (of degree 4)  and
$\pi_{Y|\Delta_\ph}\colon\Delta_\ph\to T$ (of degree 3) determine each other, we refer the reader to \cite{casnati} for a study of the relation among them.
\end{prg}
\begin{prg}\label{trace}
  The direct image $(\pi_B)_*\ol_{B_Z}$ is a rank $4$ locally free sheaf on $T$; moreover the natural inclusion $\pi_B^{\#}\colon\ol_T\to (\pi_B)_*\ol_{B_Z}$ splits via the trace map
  $$\frac{1}{4}\text{Tr}_{B_Z/T}\colon (\pi_B)_*\ol_{B_Z}\to \ol_T,$$
  see \cite[\S 6.3.D]{lazII}. Then $\ker\text{Tr}_{B_Z/T}\cong\coker\pi_B^{\#}$  is a rank $3$ locally free sheaf on $T$ such that:
  $$(\pi_B)_*\ol_{B_Z}\cong \ol_T\oplus\ker\text{Tr}_{B_Z/T}.$$ We set $\ma{G}:=(\ker\text{Tr}_{B_Z/T})^{\vee}$,  
  this is called the bundle associated to the covering $\pi_B$, or also the Tschirnhausen module of $\pi_B$.
\end{prg} 
\begin{parg}\label{D1}
  There exists a divisor $D_1$ on $T$ such that $\ma{G}\cong\ma{E}\otimes \ol_T(D_1)$ and $H_{Z|B_Z}\sim K_{B_Z}-\pi_B^*(D_1+K_T)$ (see \ref{vbE} and \ref{gianfranco}).
\end{parg}  
\begin{proof}
  It is shown in \cite[Theorem 1.3]{casnek} (see also \cite[Proposition 6.3.49]{lazII}) that there is an embedding $B_Z\hookrightarrow\pr_T(\ma{G})$ such that $\pi_B$ is the restriction of the $\pr^2$-bundle map, and  $\ol_{\pr_T(\ma{G})}(1)_{|B_Z}\cong\omega_{B_Z/T}$ \cite[Theorem 2.1(ii)]{casnek}.
Moreover
such an embedding of $\pi_B\colon B_Z\to T$ in a $\pr^2$-bundle over $T$ is unique up to isomorphism, in the following sense: if $\pi_B$ factors as $p\circ\iota$ where $\iota\colon B_Z\hookrightarrow \pr$ is an embedding, $p\colon \pr\to T$ is a $\pr^2$-bundle, and $\iota_t\colon (B_Z)_t\hookrightarrow \pr_t=\pr^2$ is a non-degenerate, arithmetically Gorenstein embedding for every $t\in T$, then $\pr\cong\pr_T(\ma{G})$, and this isomorphism respects the embeddings of $B_Z$ and the $\pr^2$-bundle structures [\emph{ibidem}, Theorem 2.1, Step B, Step C].

Now we have
$B_Z\subset Z=\pr_T(\ma{E})$, and for every $t\in T$, $(B_Z)_t\subset Z_t=\pr^2$ is the base locus of a pencil of conics whose general member is smooth; therefore $(B_Z)_t$ is not contained in a line, and it is an arithmetically Gorenstein subscheme of $\pr^2$, being the complete intersection of two conics.
We deduce that $\pr_T(\ma{G})\cong\pr_T(\ma{E})$ and hence $\ma{G}\cong\ma{E}\otimes \ol_T(D_1)$ for some divisor $D_1$ on $T$. 

  We also have $\ol_Z(H_Z)=\ol_{\pr_T(\ma{E})}(1)\cong\ol_{\pr_T(\ma{G})}(1)\otimes\pi_Z^*\ol_T(-D_1)$, and restricting to $B_Z$ we get $H_{Z|B_Z}\sim \omega_{B_Z/T}-\pi_B^*(D_1)=K_{B_Z}-\pi_B^*(D_1+K_T)$.
\end{proof}  
\begin{prg}\label{pantelleria}
In the blow-up $\alpha\colon X\to Z$ of $A_Z$ and $B_Z$, let us first blow-up $A_Z$ by
$\sigma'\colon X_2'\to Z$, and then the transform  $B_2'\subset X_2'$ of $B_Z$:
 $$\xymatrix{   X\ar[rd]^{\alpha}\ar[r]^{\tau'=\tau'_{F}}\ar[d]_{\sigma=\sigma_E}&{X_2'}\ar[d]^{\sigma'=\sigma'_{E_2'}}
    \\
    {X_2}\ar[r]_{\tau}&  {Z}
  }$$
The exceptional divisor of $\tau'$ is $F$, and 
that of $\sigma'$ is $E_2':=\tau'(E)\subset X_2'$ (see Figure \ref{figura_X_2} on page \pageref{figura_X_2}).
\end{prg}
\begin{parg}\label{alternative}
 We have a commutative diagram:
\begin{equation}\label{calcio}
 \xymatrix{ X\ar[rd]^{\alpha}\ar[r]^{\tau'=\tau'_{F}}\ar[d]_{\sigma=\sigma_E}&{X_2'}\ar[rd]^{\ph'}\ar[d]^{\sigma'}&
    \\
    {X_2}\ar[r]^{\tau}\ar[dr]_{\ph}&  {Z}\ar[dr]^{\pi_Z} &W\ar[d]^{\pi_W}\\
 & {Y}\ar[r]_{\pi_Y}&{T}
}\end{equation}
where $W$ is smooth of dimension $n-1$, and  $\pi_W\colon W\to T$ and  $\ph'\colon X_2'\to W$ are smooth $\pr^1$-fibrations (see Figure \ref{figura_psi_2}).
\end{parg}

\begin{figure}[h]\caption{A section of $\NE(\psi)$.}\label{figura_psi_2}

\bigskip
  
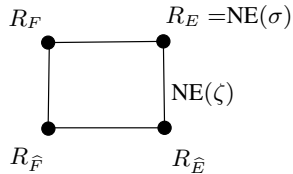
    \end{figure}

\begin{proof}
Since $\pi_Z$ is a $\pr^2$-bundle and $A_Z$ is a section of $\pi_Z$, the composition $\pi_Z\circ\sigma'\colon X_2'\to T$ is a $K$-negative, smooth morphism with fiber the Hirzebruch surface $\mathbb{F}_1$, and has a different factorization as $X_2'\stackrel{\ph'}{\to} W\stackrel{\pi_W}{\to} T$ where $W$ is smooth of dimension $n-1$ and both $\ph'$ and $\pi_W$ are smooth $\pr^1$-fibrations.
We have $\NE(\ph'\circ\tau')=R_F+R_{\wi{F}}$. 
\end{proof}
\begin{parg}\label{caldaia}
  Set  $\wi{E}_2':=\tau'(\wi{E})\subset X_2'$. Then  $\wi{E}_2'$ is a section of $\ph'$ (see Figure \ref{figura_X_2}).
\end{parg}
\begin{proof}
  We have $\wi{E}\cdot f=1$, $\wi{E}\cdot \hat{f}=0$, and $F\cdot\hat{f}=1$ (see \ref{tables}).
The curve $\tau'(\hat{f})$ is a fiber of $\ph'$ and
$$\wi{E}'_2\cdot \tau'(\hat{f})=(\tau')^*(\wi{E}_2')\cdot\hat{f}
=\bigl(\wi{E}+(\wi{E}\cdot f)F\bigr)\cdot\hat{f}=(\wi{E}+F)\cdot \hat{f}=1.$$
 
Moreover $\ph'$ must be finite on $\wi{E}_2'$, because $\NE(\ph'\circ\tau')=R_F+R_{\wi{F}}$ and $\N(\wi{E},X)\cap\ker\psi_*=\R[\hat{e}]$ (see \ref{description}), therefore $\N(\wi{E},X)\cap\NE(\ph'\circ\tau')=\{0\}$, $\ph'\circ\tau'$ is finite on $\wi{E}$, and  $\ph'$ is finite on $\wi{E}_2'=\tau'(\wi{E})$.
 We conclude that $\wi{E}_2'$ is a section of $\ph'$.
\end{proof}

\begin{figure}[h]\caption{The surface $(X_2')_t$ for $t\in T$ general.}\label{figura_X_2}

\bigskip
  
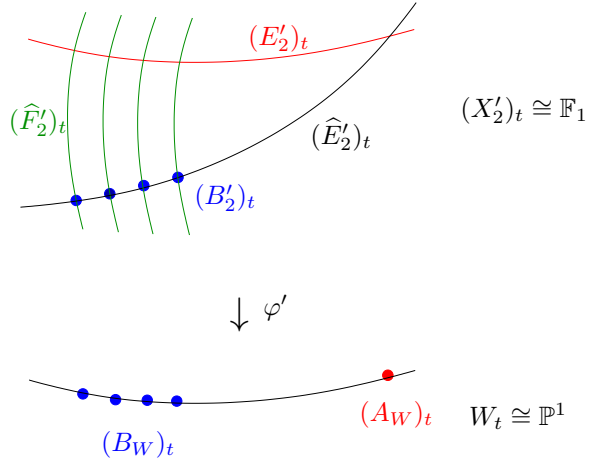
    \end{figure}

\begin{parg}\label{xi}
 We have an isomorphism $\xi:=\ph'\circ(\sigma'_{|\wi{E}_2})^{-1}\colon \wi{E}_Z\cong W$, such that $\pi_W\circ\xi=\pi_{Z|\wi{E}_Z}$.
$$\xymatrix{
  {\wi{E}_2'\subset X_2'}\ar[r]^-{\ph'} \ar[d]_-{\sigma'} & W\ar[d]^{\pi_W}\\
  {\wi{E}_Z\subset Z}\ar[ru]|{\,\xi\,}\ar[r]^-{\pi_Z}&T
  }$$
\end{parg}
\begin{proof}
  We know that $\alpha_{|\wi{E}}\colon \wi{E}\to\wi{E}_Z$ is an isomorphism (see \ref{gianfranco}), and $\alpha=\sigma'\circ\tau'$,
  therefore $\sigma'_{|\wi{E}_2'}\colon \wi{E}_2'\to\wi{E}_Z$ is an isomorphism too.
Moreover $\wi{E}_2'\subset X_2'$ is a section of $\ph'$ (see \ref{caldaia}), therefore $\ph'_{|\wi{E}_2'}\colon \wi{E}_2'\to W$ is an isomorphism, and we get the statement.
\end{proof}  
\begin{parg}\label{W1}
  Set $A_W:=\xi(A_Z)\subset W$ and $B_W:=\xi(B_Z)\subset W$. Then $A_W$ and $B_W$ are smooth and disjoint, 
  $A_W$ is a section of $\pi_W\colon W\to T$, and  $\pi_{W|B_W}\colon B_W\to T$ is a finite morphism of degree $4$, isomorphic to $\pi_B\colon B_Z\to T$
   (see Figure \ref{figura_X_2}).
\end{parg}
\begin{proof}
The statement follows from \ref{xi}, \ref{alpha}, \ref{centers}, and \ref{gianfranco}.
\end{proof}  
   \begin{parg}\label{W_T_Fano}
$W$ and $T$ are Fano.
\end{parg}
\begin{proof}
  Let us consider the composition $\ph'\circ\tau'\colon X\to W$; recall that $\ph'\colon X_2'\to W$ is a smooth $\pr^1$-fibration, and $\tau'\colon X\to X_2'$ is the blow-up of $B_2'\subset X_2'$, transform of $B_Z\subset Z$ under $\sigma'$ (see \ref{pantelleria} and \ref{alternative}).
Since $B_Z\subset\wi{E}_Z$, we have $B_2'\subset\wi{E}_2'$, and $\ph'(B_2')=B_W$ (see \ref{xi} and \ref{W1}). Moreover $\wi{E}_2'$ is a section of $\ph'$ (see \ref{caldaia}), therefore $B_2'$ is a section of $\ph'_{|(\ph')^{-1}(B_W)}$. 
We conclude that
$\ph'\circ\tau'$
  is a conic bundle, and has smooth fibers over $W\smallsetminus B_W$, reducible fibers over $B_W$, and no non-reduced fibers. Hence $W$ is Fano by \cite[Proposition 4.3]{wisn}, and since $\pi_W$ is smooth (see \ref{alternative}), $T$ is Fano too.
\end{proof}
\begin{parg}\label{W}
  We have $W\cong\pr_{T}(\ol\oplus\ol(D))$ where $D$ is a divisor  on $T$ such that  $-K_T+D$ is ample and $[D]\in\Eff(T)\smallsetminus\{0\}$. 
  Moreover, if $H_W$ is a tautological divisor, we have $H_{W|A_W}\sim 0$ and $B_W\sim 4 H_W$.  
\end{parg}
\begin{proof}
  Recall that $\pi_W\colon W\to T$ is a smooth $\pr^1$-fibration, $A_W$ is a section of $\pi_W$, and $A_W\cap B_W=\emptyset$ (see \ref{alternative} and \ref{W1}).
  Let $D$ be a divisor on $T$ such that $\ol_{W}(A_W)_{|A_W}\cong\ol_{T}(-D)$, and set $D':=-D$.
Since $W$ is Fano,
we deduce from  Lemma \ref{sectionsFano} that $-K_T+D=-K_{T}- D'$ is ample, and that $W\cong\pr_{T}(\ol\oplus\ol(D'))$, with the section $A_W$ corresponding to the projection onto $\ol_{T}(D')$. By tensoring the vector bundle with $\ol_T(D)$, we get $W\cong\pr_{T}(\ol(D)\oplus\ol)$ 
  with $A_W$ corresponding to the projection onto $\ol_{T}$.
  If $H_W$ is a tautological divisor, we have $H_{W|A_W}\sim 0$ and  $B_{W|A_W}= 0$, thus $B_W\sim\lambda H_W$ for some $\lambda$, and since $B_W$ has degree $4$ onto $T$, we get 
  $B_W\sim 4 H_W$.

  Finally the existence of a smooth connected divisor $B_W\in |4 H_W|$ implies that
 $[D]\in\Eff(T)\smallsetminus\{0\}$, see Remark \ref{Deffective}.
\end{proof}
\begin{parg}\label{coker}
We have $\ma{G} \cong\ol_T(D)\oplus\ol_T(2D)\oplus\ol_T(3D)$ (see \ref{trace}).
\end{parg}
\begin{proof}
Since $\pi_B\colon B_Z\to T$ is isomorphic to the cover $\pi_{W|B_W}\colon B_W\to T$ (see \ref{W1}), we can compute $\ma{G}\cong(\coker\pi_B^{\#})^{\vee}$ using the embedding $\iota\colon B_W\hookrightarrow W=\pr_T(\ol\oplus\ol(D))$. We still denote by $\pi_B$ the restriction $\pi_{W|B_W}$.
  
Consider the exact sequence on $W$:
$$0\la \ol_W(-B_W)\la \ol_W\stackrel{\iota^{\#}}{\la}\iota_*\ol_{B_W}\la 0$$
and take direct images via $\pi_W$. 
We have
$$(\pi_W)_*\ol_W\cong\ol_T\quad\text{and}\quad
(\pi_W)_*\iota_*\ol_{B_W}=(\pi_B)_*\ol_{B_W}$$
and under these isomorphisms the natural restriction morphism $\iota^{\#}\colon  \ol_W\to \iota_*\ol_{B_W}$ is sent, via $(\pi_W)_*$, to
the morphism $\pi_B^{\#}\colon \ol_T\to(\pi_B)_*\ol_{B_W}$.
We also have
$$(\pi_W)_*\ol_W(-B_W)=0\quad\text{and}\quad R^1(\pi_W)_*\ol_W=0,$$
so we get an exact sequence on $T$:
$$0\la \ol_T\stackrel{\pi_B^{\#}}{\la} (\pi_B)_*\ol_{B_W}\la R^1(\pi_W)_*\ol_W(-B_W)\la 0.$$
 We conclude that $R^1(\pi_W)_*\ol_W(-B_W)\cong \coker \pi_B^{\#}\cong \ma{G}^{\vee}$.

 Recall that $W\cong\pr_T(\ol\oplus\ol(D))$, $\omega_{W/T}=\ol_W(-2H_W+\pi_W^*(D))$,
 $\ol_W(B_W)\cong\ol_W( 4H_W)$ (see \ref{W}), and by relative Serre duality we have:
\begin{align*}
  \ma{G}&\cong \bigl(R^1(\pi_W)_*\ol_W(-B_W)\bigr)^{\vee}\cong
(\pi_W)_*\bigl(\ol_W(B_W)\otimes\omega_{W/T}\bigr)\\
&\cong (\pi_W)_*\ol_W\bigl(2H_W+\pi_W^*(D)\bigr)\cong  (\pi_W)_*\ol_W(2H_W)\otimes\ol_T(D)\\
 & \cong\Sym^2\bigl(\ol_T\oplus\ol_T(D)\bigr)\otimes\ol_T(D) \cong\ol_T(D)\oplus\ol_T(2D)\oplus\ol_T(3D).\qedhere
\end{align*}
\end{proof}
\begin{parg}\label{tautologicals}
Recall the isomorphism $\xi\colon \wi{E}_Z\to W$ from \ref{xi}. We have $H_{Z|\wi{E}_Z}\sim\xi^*(2H_W)$ (recall that $H_Z$ is a tautological divisor in $Z$, see \ref{vbE}).
\end{parg}
\begin{proof}
  Set $M:=(\xi^{-1})^*H_{Z|\wi{E}_Z}$; we show that $M\sim 2 H_W$. We have $M\sim\lambda H_W+\pi_W^*(M_T)$, where $\lambda\in\Z$ and $M_T$ is a divisor on $T$. Moreover, if $\Gamma\subset W$ is a fiber of  $\pi_W$, then $\xi^{-1}(\Gamma)=\hat{e}_Z$ is a conic in a fiber of the $\pr^2$-bundle $\pi_Z\colon Z\to T$ (see \ref{gianfranco}), and we have
  $$\lambda=M\cdot\Gamma=H_{Z|\wi{E}_Z}\cdot \hat{e}_Z=2.$$
  We also have $H_{Z|A_Z}\sim 0$ (see \ref{vbE}), thus  $(H_{Z|\wi{E}_Z})_{|A_Z}\sim 0$, and $M_{|A_W}\sim 0$. Since  $H_{W|A_W}\sim 0$ (see \ref{W}), we get
  $$0\sim M_{|A_W}\sim 2H_{W|A_W}+\pi_W^*(M_T)_{|A_W}\cong M_T,$$
and we get the statement.
\end{proof}   
\begin{parg}
  We have $\ol_T(D_1)\cong\ol_T(D)$ and
$\ma{E}\cong \ol_T\oplus \ol_T(D)\oplus\ol_T(2D)$ (see \ref{vbE} and \ref{D1}).
 \end{parg}
\begin{proof}
  We have $K_W-\pi_W^*(K_T)\sim -2H_W+\pi_T^*(D)$, $B_W\sim 4 H_W$ (see \ref{W}), and
  $$K_{B_W}-\pi_B^*(K_T)\sim \bigl(K_W+B_W-\pi_W^*(K_T)\bigr)_{|B_W}
  \sim 2H_{W|B_W}+\pi_B^*(D).$$
  On the other hand by \ref{tautologicals} and \ref{D1} we also have $$2H_{W|B_W}\sim (\xi^{-1})^*(H_{Z|B_Z})
\sim (\xi^{-1})^*\bigl(K_{B_Z}-\pi_B^*(D_1+K_T)\bigr)=
K_{B_W}-\pi_B^*(D_1+K_T)$$
namely $K_{B_W}-\pi_B^*(K_T)\sim 2H_{W|B_W}+\pi_B^*(D_1)$,
so that $$\bigl(\pi_W^*\ol_T(D_1-D)\bigr)_{|B_W}=\pi_B^*\ol_T(D_1-D)\cong\ol_{B_W}.$$

  Now we notice that
  since $B_W$ dominates $T$, $\N(B_W,W)$ dominates $\N(T)$ under $(\pi_W)_*$, hence $\dim\N(B_W,W)\geq \rho_T=\rho_W-1$, and $\de(B_W)\leq 1$. By duality $\de(B_W)=\dim\ker r$ where $r\colon\Nu(W)\to \Nu(B_W)$ is the restriction. Since $A_W\cap B_W=\emptyset$ (see \ref{W1}), we have $[A_W]\in\ker r$, and we deduce that $\ker r=\R[A_W]$ and $\pi_W^*(D_1-D)\equiv \lambda A_W$ for some $\lambda\in \R$. By taking the intersection with a fiber of $\pi_W$ we see that $\lambda=0$, $\pi_W^*\ol_T(D_1-D)\cong\ol_W$, and finally $\ol_T(D_1-D)\cong\ol_T$, therefore $\ol_T(D_1)\cong\ol_T(D)$.

  Then \ref{D1} and \ref{coker} give $\ma{E}\cong\ma{G}\otimes\ol_T(-D)\cong\ol_T\oplus \ol_T(D)\oplus\ol_T(2D)$.
\end{proof}
\begin{parg}\label{portosanto}
The sequence \eqref{sabri} splits and $\ma{N}^{\vee}_{A_Z/Z}\cong\ol_T(D)\oplus\ol_T(2D)$.
\end{parg}
\begin{proof}
  Consider the surjective morphism $\phi\colon \ma{E}\cong\ol_T\oplus\ol_T(D)\oplus\ol_T(2D)\to\ol_T$ in \eqref{sabri}. We have $$
  \Hom\bigl(\ol_T\oplus\ol_T(D)\oplus\ol_T(2D),\ol_T\bigr)\cong H^0\bigl(T,\ol_T\oplus\ol_T(-D)\oplus\ol_T(-2D)\bigr)$$
and $H^0(T,\ol_T(-D))=H^0(T,\ol_T(-2D))=0$ because $[D]\in\Eff(T)\smallsetminus\{0\}$ (see \ref{W}). Therefore $\phi=\lambda\pi$ where $\lambda\in\C\smallsetminus\{0\}$ and $\pi\colon\ma{E}\to\ol_T$ is the projection onto the first summand, the sequence splits, and $\ma{N}^{\vee}_{A_Z/Z}=\ker\phi\cong \ol_T(D)\oplus\ol_T(2D)$.
\end{proof}  
\begin{parg}
  $-K_T-2D$ is ample.
\end{parg}
\begin{proof}
Since $E\cap F=\emptyset$ and $F=\Exc(\tau')$ (see \ref{pantelleria}), we have $-K_{X|E}\stackrel{\tau'}{\cong} -K_{X_2'|E_2'}$ ample. Moreover $E_2'$ is the exceptional divisor of the blow-up $\sigma'\colon X_2'\to Z$ of $A_Z$, therefore by \ref{portosanto} we have $E_2'=\pr_T(\ol(D)\oplus\ol(2D))$ with tautological class $\eta=\ol_{X_2'}(-E_2')_{|E_2'}$. We also have $K_{E_2'}=\pi^*(K_T+3D)-2\eta$, thus
$-K_{X_2'|E_2'}=-K_{E_2'}+E'_{2|E_2'}=\pi^*(-K_T-3D)+\eta$ is the tautological class for $\pr_T(\ol(-K_T-2D)\oplus\ol(-K_T-D))$. We deduce that $-K_T-2D$ is ample.
\end{proof}
\begin{prg}
We have  $\de_T\leq \de_X=2$ by \cite[Remark 3.3.18]{codim}.
\end{prg}  
This concludes the proof of Theorem \ref{rho=6}.
\section{Reconstructing $X$ from $(T,D)$}\label{reconstructing}
\noindent 
In this section we prove Proposition \ref{construction}, by showing how a Fano variety $X$ as in Section \ref{section_rho=6} can be reconstructed from the Fano variety $T$ and the divisor $D$ on $T$, provided that the conditions 
given in Proposition \ref{construction} are satisfied.

Let 
$T$ be a smooth Fano variety of dimension $n-2$, and $D$ a divisor on $T$.
\begin{prg}\label{defW}
  Set
$$W:=\pr_T(\ol\oplus\ol(D))\stackrel{\pi_W}{\la} T$$
and let $H_W\subset W$ be the section  corresponding to the projection $\ol\oplus\ol(D)\twoheadrightarrow\ol(D)$, so that $H_W$ is a tautological divisor.
\end{prg}
\begin{assumption}\label{4H}
  We assume that there exists a smooth connected divisor $B_W\in |4H_W|$.
\end{assumption}
  \begin{remark}\label{bsptfree}
    If $D$ is base-point-free and $D\not\sim 0$, then Assumption \ref{4H} is verified.
    
    Indeed   $\ol_T\oplus\ol_T(D)$ is  globally generated, therefore $H_W$ and $4H_W$ are base-point-free, and there exists a smooth divisor $B_W\in|4H_W|$. If $B_W$ is not connected, then the linear system $|4H_W|$ defines a morphism from $W$ to a curve, that factors through a contraction $\theta\colon W\to\pr^1$ sending $B_W$ to points, and finite on the fibers of $\pi_W$. Then $H_W$ must be a fiber of $\theta$, and is a section of $\pi_W$: this implies that the induced finite morphism $(\pi_W,\theta)\colon W\to T\times\pr^1$ has degree one, therefore  $W\cong T\times\pr^1$ and $\ol_T(D)\cong\ol_T$, a contradiction. We conclude that  $B_W$ is  connected.
  \end{remark}
   \begin{remark}\label{Deffective}
     Assumption \ref{4H} implies that $[D]\in\Eff(T)$ and $D\not\sim 0$.

     Indeed, since $\pi_{W*}\ol_W(4H_W)\cong\Sym^4(\ol_T\oplus\ol_T(D))$, $h^0(W,4H_W)>0$ implies that $h^0(T,mD)>0$ for some $m\in\{0,1,2,3,4\}$. Moreover, if $D\sim 0$, then $W=\pr^1\times T$ and $H_W=\{pt\}\times T$, and there is no smooth connected divisor in $|4H_W|$.
\end{remark}

\begin{prg}\label{ref}
  Let $A_W\subset W$ be the section of $\pi_W$ corresponding to the projection $\ol\oplus\ol(D)\twoheadrightarrow\ol$, so that $A_W\cap H_W=\emptyset$ and $A_W\sim H_W-\pi_W^*(D)$. It is easy to see that $A_W\cap B_W=\emptyset$.

  We note that $A_W$ is a tautological divisor for $W=\pr_T(\ol(-D)\oplus\ol)$, thus $H^0(W,A_W)\cong H^0(T,\ol(-D)\oplus\ol)=\C$ (as $H^0(T,-D)=0$ by Remark \ref{Deffective}), and $\dim |A_W|=0$.
\end{prg}
The following lemma is left to the reader.
\begin{lemma}\label{BW}
Set $\pi_{B}:=\pi_{W|B_W}\colon B_W\to T$. Then $\pi_{B}$ is finite of degree $4$,  $H_{W|B_W}\sim \pi_{B}^*(D)$, $-K_{W|B_W}=\pi_{B}^*(-K_T+D)$, and   $-K_{B_W}=\pi_{B}^*(-K_T-3D)$.
\end{lemma}
\begin{prg}\label{defZ}
We set $$Z:=\pr_T(\ol\oplus\ol(D)\oplus\ol(2D))\stackrel{\pi_Z}{\la} T.$$
Let $A_Z\subset Z$ be the section of $\pi_Z$ corresponding to the projection $\ol\oplus\ol(D)\oplus\ol(2D)\twoheadrightarrow\ol$, so that $$\ma{N}_{A_Z/Z}\cong \ol_T(-D)\oplus\ol_T(-2D).$$ Set also $H_Z:=\pr_T(\ol(D)\oplus\ol(2D))\hookrightarrow Z$; we have $H_Z\cap A_Z=\emptyset$ and $H_Z$ is a tautological divisor for $Z$.
\end{prg}
\begin{lemma}\label{Ehat}
  We have $A_Z\subset\text{Bs}|2(H_Z-\pi_Z^*(D))|$. The general divisor $\wi{E}_Z\in |2(H_Z-\pi_Z^*(D))|$ is smooth, and $\pi_{Z|\wi{E}_Z}\colon \wi{E}_Z\to T$ is a smooth conic bundle.
\end{lemma}
\begin{proof}
Note that $H_Z-\pi_Z^*(D)$ is a tautological divisor for $Z=\pr_T(\ol(-D)\oplus\ol\oplus\ol(D))$, so that
  \begin{align*}
    H^0\bigl(Z, 2(H_Z-\pi_Z^*(D))\bigr) & \cong  H^0\bigl(T,\text{Sym}^2(\ol(-D)\oplus\ol\oplus\ol(D))\bigr)\\
& \cong    H^0(T,\ol(-2D)\oplus\ol(-D)\oplus \ol\oplus \ol \oplus\ol(D)\oplus\ol(2D)).
  \end{align*}
  We have $H^0(T,\ol(-2D))=H^0(T,\ol(-D))=0$, because $[D]\in\Eff(T)\smallsetminus\{0\}$ by Remark \ref{Deffective}.

  Let $U\subset T$ be an open subset such that $\ol_T(D)_{|U}\cong\ol_U$; then $\pi_T^{-1}(U)\cong U\times\pr^2_{(x_0:x_1:x_2)}$ and $A_Z\cap\pi_T^{-1}(U)\cong U\times\{(1:0:0)\}$.    A non-zero section in $H^0(Z, 2(H_Z-\pi_Z^*(D)))$ gives a divisor $\wi{E}_U \subset U\times\pr^2_{(x_0:x_1:x_2)}$ with equation:
  $$\lambda x_0x_2+\mu x_1^2+fx_1x_2+gx_2^2=0,$$
  where $\lambda,\mu\in \C$, $f\in H^0(T,\ol(D))$, $g\in  H^0(T,\ol(2D))$. We see that $\wi{E}_U$ always contains $A_Z\cap \pi_T^{-1}(U)$, therefore $A_Z\subset\text{Bs}|2(H_Z-\pi_Z^*(D))|$.
Moreover if $\lambda\neq 0$ and $\mu\neq 0$ then every fiber of $\pi_{Z|\wi{E}_U}\colon \wi{E}_U\to U$ is a smooth conic, so that $\pi_{Z|\wi{E}_U}$ is a smooth morphism, and $\wi{E}_U$ is smooth. Therefore, for $\lambda\neq 0$ and $\mu\neq 0$, $\wi{E}_Z$ is smooth and $\pi_{Z|\wi{E}_Z}\colon\wi{E}_Z\to T$ is a smooth conic bundle. 
\end{proof}

Let $\sigma'\colon X_2'\to Z$ be the blow-up of $A_Z$ and $E_{2}'\subset X_2'$ the exceptional divisor (the notation is chosen to match with that of Section \ref{section_rho=6}). 
\begin{lemma}\label{X_2}
  We have a commutative diagram:
  $$\xymatrix{ {X_2'}\ar[r]^{\sigma'}\ar[d]_{\ph'}& Z\ar[d]^{\pi_Z}\\ 
W\ar[r]^{\pi_W}&T
    }$$
where $\ph'$ is a $\pr^1$-bundle and  $X_2'=\pr_W(\ol\oplus\ol(H_W+\pi_W^*(D)))$, with tautological divisor  the transform $H_{2}'\subset X_2'$ of $H_Z\subset Z$.
Moreover $E_{2}'\cap H_{2}'=\emptyset$ and $E_{2}'\sim H_{2}'-(\ph')^*(H_W+\pi_W^*(D))$.
\end{lemma}  
\begin{proof}
  The composition $\pi_Z\circ \sigma'\colon X_2'\to T$ is a $K$-negative smooth morphism with fiber $\mathbb{F}_1$, and it admits a second factorization as
  $X_2'\stackrel{\ph'}{\to}\overline{W}\to T$ where both morphisms are smooth $\pr^1$-fibrations. The exceptional divisor $E_{2}'$ of $\sigma'$ is a section of $\ph'$ and $E_{2}'\cong \pr_{A_Z}(\ma{N}^{\vee}_{A_Z/Z})\cong\pr_T(\ol(D)\oplus\ol(2D))\cong W$ (see \ref{defZ}), so we can take $\overline{W}=W$. Moreover  $\ol_{X_2'}(-E_{2}')_{|E_{2}'}$ is the tautological line bundle of $\pr_{A_Z}(\ma{N}^{\vee}_{A_Z/Z})$, and we get
  $\ol_{X_2'}(-E_{2}')_{|E_{2}'}\cong \ol_W(H_W+\pi_W^*(D))$ (see \ref{defW}).

Let $H_{2}'\subset X_2'$ be the transform of $H_Z\subset Z$, and recall that $H_Z\cap A_Z=\emptyset$ (see \ref{defZ}). Then $H_{2}'$ is another section of $\ph'$, disjoint from $E_{2}'$, so that $\ph'\colon X_2'\to W$ is the projectivization of a decomposable rank $2$ vector bundle on $W$. Write $X_2'=\pr_W(\ol\oplus L)$ with $L\in\Pic(W)$, where the section $E_{2}'$ corresponds to the projection onto $\ol_W$, and $H_{2}'$ to that onto $L$; note that $H_2'$ is a tautological divisor. Then $L\cong \ma{N}^{\vee}_{E_{2}'/X_2'}\cong \ol_W(H_W+\pi_W^*(D))$, and we get the statement.
\end{proof}
\begin{lemma}\label{defEhat2}
  Let $\wi{E}_Z\subset Z$ be as Lemma \ref{Ehat}, and $\wi{E}_{2}'\subset X_2'$ its transform.
  Then $\wi{E}_2'$ is a section of $\ph'$, and
  \begin{equation}\label{coffee}
    \wi{E}_{2}' \sim H_{2}'+(\ph')^*\bigl(H_W-\pi_W^*(D)\bigr)\sim E_2'+(\ph')^*(2H_W)
  \end{equation}
  (see Figure \ref{figura_X_2} on page \pageref{figura_X_2}).
  \end{lemma}
\begin{proof}
  Since $A_Z\subset \wi{E}_Z$ and $\wi{E}_Z$ is smooth, using Lemma \ref{X_2} we have 
\begin{equation*}
  \begin{split}
  \wi{E}_{2}' & \sim(\sigma')^*(\wi{E}_Z)-E_2'
  \sim(\sigma')^*\bigl(2H_Z-2\pi^*_Z(D)\bigr)-E_2'\\
 & \sim 2H_2'-2(\ph')^*\pi_W^*(D)-
  \bigl(H_{2}'-
  (\ph')^*(H_W+\pi_W^*(D))\bigr) \\
             &  = H_{2}'+(\ph')^*\bigl(H_W-\pi_W^*(D)\bigr) \sim E_2'+(\ph')^*(2H_W).
  \end{split}
   \end{equation*}
Note that fiberwise over $T$, $(\wi{E}_Z)_t\subset Z_t\cong\pr^2$ is a smooth conic, and its transform $(\wi{E}_2')_t\subset X_2'$ is a smooth rational curve, horizontal for $\ph'$. Hence $\ph'$ is finite on $\wi{E}_{2}'$,
and
$\wi{E}_{2}'$ is a section of $\ph'$.
\end{proof}
Recall from Assumption \ref{4H} that $B_W$ is a smooth connected divisor in $|4H_W|$.
We set $\wi{F}_2':=(\ph')^{-1}(B_W)\subset X_2'$ and $B_2':=\wi{E}_2'\cap \wi{F}_{2}'$ (see Figure \ref{figura_X_2} on page \pageref{figura_X_2});
then $\wi{F}_2'$ is a $\pr^1$-bundle over $B_W$, and $B_2'$ is a section (because $\wi{E}_2'$ is a section of $\ph'$, see  Lemma \ref{defEhat2}), therefore $B_2' \cong B_W$, in particular $B_2'$ is smooth and irreducible. We also note that, by \eqref{coffee}:
\begin{equation}\label{relation}
2\wi{E}_2'\sim 2E_2'+(\ph')^*(4H_W)\sim 2E_2'+(\ph')^*(B_W)=2E_2'+\wi{F}_2'.
\end{equation}
\begin{lemma}\label{A'}
We have $\ph'(E_2'\cap\wi{E}_2')=A_W$ and $E_{2}'\cap B_2'=\emptyset$.
\end{lemma}
\begin{proof}
  Using \eqref{coffee}, $E_2'\cap H_2'=\emptyset$ (see Lemma \ref{X_2}), and \ref{ref}, 
 under the isomorphism $E_{2}'\cong W$ induced by $\ph'$
 we get
 \begin{align*}
(\wi{E}_{2}')_{|E_{2}'} & \sim \bigl(H_{2}'+(\ph')^*(H_W-\pi_W^*(D))\bigr)_{|E_{2}'}=(\ph')^*(H_W-\pi_W^*(D))_{|E_{2}'}\\ & \stackrel{\ph'}{\cong} H_W-\pi_W^*(D)\sim A_W.\end{align*}
Since $h^0(W,A_W)=1$ (see \ref{ref}), this shows that  $\ph'(E_2'\cap\wi{E}_2')=A_W$. Moreover
$(\wi{F}_{2}')_{|E_{2}'}=(\ph')^{-1}(B_W)_{|E_{2}'}$, thus  $\ph'(\wi{F}_2'\cap E_2')=B_W$. Since $E_2'$ is a section of $\ph'$, and $A_W\cap B_W=\emptyset$ (see \ref{ref}), we get
$\emptyset=(E_2'\cap\wi{E}_2')\cap (\wi{F}_2'\cap E_2')=E_2'\cap\wi{E}_2'\cap \wi{F}_2'=E_2'\cap B_2'$.
\end{proof}
\begin{prg}\label{defX}
Let $\tau'\colon X\to X_2'$ be the blow-up of $B_2'$; then $X$ is smooth with dimension $n$ and $\rho_X=\rho_T+3$.
 $$\xymatrix{ X\ar@/^{1pc}/[rr]^{\alpha}\ar[r]_{\tau'}&{X_2'}\ar[r]_{\sigma'}\ar[d]_{\ph'}& Z\ar[d]^{\pi_Z}\\ 
&W\ar[r]^{\pi_W}&T
    }$$
Set $F:=\Exc(\tau')\subset X$ and let $E,\wi{E},\wi{F}\subset X$ be the transforms of $E_{2}',\wi{E}_{2}',\wi{F}_{2}'\subset X_2'$ respectively.
Set $\alpha:=\sigma'\circ\tau'\colon X\to Z$, and $\psi:=\pi_Z\circ\alpha\colon X\to T$. Note that $\psi$ is quasi-elementary (see \S\ref{notation}), because $\pi_W\circ\ph'$ is, and $B_2'$ dominates $T$.
\end{prg}
\begin{lemma}\label{-K}
We have $-K_X=\psi^*(-K_T)+E+\wi{E}+\frac{1}{4}(F+\wi{F})$ and $2E+\wi{F}\sim 2\wi{E}+F$.
\end{lemma}  
\begin{proof}
  We note that, by Lemma \ref{X_2} and \eqref{coffee}, in $X_2'$ we have $E_{2}'+\wi{E}_{2}'\sim 2(H_{2}'-(\ph')^*(\pi_W^*(D)))$; moreover
  $-K_W=  \pi_W^*(-K_T-D)+2H_W$ (see \ref{defW}). By Lemma \ref{X_2} we get
  \begin{align*}
    -K_{X_2'}& =(\ph')^*(-K_W -H_W-\pi_W^*(D))+2H_{2}'\\
            & =  (\ph')^*\bigl(\pi_W^*(-K_T-D)+2H_W -H_W-\pi_W^*(D)\bigr)+2H_{2}'\\
            & =  (\ph')^*(\pi_W^*(-K_T))+ 2\bigl(H_{2}'-(\ph')^*(\pi_W^*(D))\bigr)+(\ph')^*(H_W)\\
    & \sim 
    (\ph')^*(\pi_W^*(-K_T))+E_{2}'+\wi{E}_{2}'+(\ph')^*(H_W)
    \\ & \sim (\ph')^*(\pi_W^*(-K_T))+E_{2}'+\wi{E}_{2}'+\frac{1}{4} \wi{F}_{2}',
  \end{align*}
  because $\wi{F}_{2}'=(\ph')^*(B_W)\sim 4(\ph')^*(H_W)$ (see Assumption \ref{4H}).
Finally $-K_X=(\tau')^*(-K_{X_2'})-F$, $(\tau')^*(E_{2}')=E$, $(\tau')^*(\wi{E}_{2}')=\wi{E}+F$, and $(\tau')^*(\wi{F}_{2}')=F+\wi{F}$, give the first statement. The second statement follows similarly from \eqref{relation}.
\end{proof}
\begin{lemma}\label{E}
  The following are equivalent:
  \begin{enumerate}[$(i)$]
    \item
      $-K_{X|E}$ is ample;
      \item
$-K_{X|\wi{E}}$ is ample;
  \item $-K_T-2D$ and $-K_T-D$ are ample.
\end{enumerate}
  \end{lemma}
\begin{proof}
 Recall that $\tau'\colon X\to X_2'$ is the blow-up of $B_2'$. Since $E_{2}'\cap B_2'=\emptyset$ (see Lemma \ref{A'}), $-K_{X|E}$ is ample if and only if  $-K_{X_2'|E_{2}'}$ is ample.
Recall also that $E_2'=\Exc(\sigma')$ where $\sigma'\colon X_2'\to Z$ is the blow-up of $A_Z$, hence $E_{2}'\cong \pr_{A_Z}(\ma{N}^{\vee}_{A_Z/Z})\cong\pr_T(\ol(D)\oplus\ol(2D))$ (see \ref{defZ}), and  $\ol_{X_2'}(-E_{2}')_{|E_{2}'}$ is the tautological line bundle.
Then by adjunction
$$-K_{X_2'|E_{2}'}=-K_{E_{2}'}+E_{2|E_{2}'}'=\pi^*(-K_T-3D)-E_{2|E_{2}'}'$$
(where $\pi\colon E_2'\to T$ is the $\pr^1$-bundle)
and this is a tautological divisor for $\pr_T(\ol(-K_T-2D)\oplus\ol(-K_T-D))$. This shows the equivalence of $(i)$ and $(iii)$.

  We also have, using \eqref{coffee} and Lemma \ref{A'}:
  \begin{align*} -K_{X|\wi{E}} &=-K_{\wi{E}}+\wi{E}_{|\wi{E}}=-K_{\wi{E}}+\bigl((\tau')^*(\wi{E}_2')-F\bigr)_{|\wi{E}}
    \cong -K_{\wi{E}_2'}+\wi{E}_{2|\wi{E}_2'}'-B_2'  \\ &\sim
  -K_{\wi{E}_2'}+\bigl(E_2'+(\ph')^*(2H_W)\bigr)_{|\wi{E}_2'}-B_2' \stackrel{\ph'}{\cong}
  -K_W+ A_W+2H_W-B_W  \\ & \sim
                           2H_W+\pi_W^*(-K_T-D)+\bigl(H_W-\pi_W^*(D)\bigr)+2H_W-4H_W
                           \\ & = H_W+\pi_W^*(-K_T-2D)
  \end{align*}
where the relations in $W$ follow from \ref{defW}, Assumption \ref{4H}, and \ref{ref}.
  Then we see that $-K_{X|\wi{E}}$
 is a tautological divisor for 
 $\pr_T(\ol(-K_T-2D)\oplus \ol(-K_T-D))$, and this shows the equivalence of $(ii)$ and $(iii)$.
\end{proof}
\begin{lemma}\label{F}
 The following are equivalent:
  \begin{enumerate}[$(i)$]
    \item
      $-K_{X|F}$ is ample;
      \item
$-K_{X|\wi{F}}$ is ample;
  \item  $-K_T+ D$ and $-K_T- D$ are ample.
\end{enumerate}
\end{lemma}
\begin{proof}
 Since $B_2'=\wi{E}_{2}'\cap \wi{F}_{2}'$, we have $\ma{N}_{B_2'/X_2'}=\ol_{B_2'}(\wi{E}_{2|B_2'}')\oplus\ol_{B_2'}(\wi{F}_{2|B_2'}')$ and $F=\pr_{B_2'}(\ol(-\wi{E}_{2|B_2'}')\oplus\ol(-\wi{F}_{2|B_2'}'))$ with tautological class $-F_{|F}$. Then $-K_{X|F}=-K_{F}+F_{|F}=\pi^*(-K_{B_2'}+\wi{E}_{2|B_2'}'+\wi{F}_{2|B_2'}')-F_{|F}$ (where $\pi\colon F\to B_2'$ is the $\pr^1$-bundle), and this is a tautological divisor for
  $\pr_{B_2'}(\ol(-K_{B_2'}+\wi{E}_{2|B_2'}')\oplus\ol(-K_{B_2'}+\wi{F}_{2|B_2'}'))$.

  Recall that $\wi{E}_{2}'$ is a section of $\ph'$, and
  by \eqref{coffee} and Lemma \ref{A'} we have
  $$\wi{E}_{2|\wi{E}_{2}'}'\sim E_{2|\wi{E}_2'}'+(\ph')^*_{|\wi{E}_2'}(2H_W)\stackrel{\ph'}{\cong} A_W+2H_W.$$
  Thus $\wi{E}_{2|B_2'}'=(\wi{E}_{2|\wi{E}_{2}'}')_{|B_2'}\stackrel{\ph'}{\cong} (A_W+2H_W)_{|B_W}
  \sim 2\pi_{B}^*(D)$ (see Lemma \ref{BW}),
  and using again Lemma \ref{BW}:
  $$-K_{B_2'}+\wi{E}_{2|B_2'}'\stackrel{\ph'}{\cong} -K_{B_W}+2\pi_{B}^*(D) \sim\pi_{B}^*(-K_T-D).$$
  Similarly, $\wi{F}_{2|\wi{E}_{2}'}'\cong B_W$, $\wi{F}_{2|B_2'}'\cong B_{W|B_W}\sim 4H_{W|B_W}\sim 4\pi_{B}^*(D)$, and 
$-K_{B_2'}+\wi{F}_{2|B_2'}'\cong \pi_{B}^*(-K_T+D)$. Since
 $\pi_{B}$ is a finite map, this shows the equivalence of $(i)$ and $(iii)$. 
We also note that $F\cong\pr_{B_W}(\ol\oplus\ol(2\pi_B^*(D)))$.
 
Moreover, similarly as before, using  that  $X_2'=\pr_W(\ol\oplus\ol(H_W+\pi_W^*(D)))$ (Lemma \ref{X_2})  and
    \eqref{coffee} we get:
  \begin{align*}
    -K_{X|\wi{F}} & \cong -K_{X_2'|\wi{F}_{2}'}-B_2'=(-K_{X_2'}-\wi{E}_2')_{|\wi{F}_{2}'}\\
                  &  = \Bigl(2H_2'+(\ph')^*\bigl(-K_W-H_W-\pi_W^*(D)\bigr)-H_2'-(\ph')^*\bigl(H_W-\pi_W^*(D)\bigr)\Bigr)_{|\wi{F}_{2}'}
    \\ & = \bigl(H_{2}'+(\ph')^*(-K_W-2H_W)\bigr)_{|\wi{F}_{2}'}.\end{align*}
Since $\wi{F}_2'=\pr_{B_W}(\ol\oplus\ol(H_W+\pi_W^*(D))_{|B_W})$, with tautological divisor $H_{2|\wi{F}_2'}'$,
  this is a tautological divisor for:  $$\wi{F}_{2}'=\pr_{B_W}\bigl(\ol(-K_W-2H_W)_{|B_W}\oplus\ol\bigl(-K_W-H_W+\pi_W^*(D)\bigr)_{|B_W}\bigr).$$

  Now using Lemma \ref{BW} we have
  $(-K_W-2H_W)_{|B_W}\sim\pi_B^*(-K_T-D)$ and  $(-K_W-H_W+\pi_W^*(D))_{|B_W}\sim\pi_{B}^*(-K_T+D)$, and again by the finiteness of $\pi_B$, we get the equivalence of $(ii)$ and $(iii)$.
\end{proof}
\begin{proposition}
  The following are equivalent:
  \begin{enumerate}[$(i)$]
    \item
      $X$ is Fano;
      \item 
        $-K_T- 2D$ and $-K_T+D$ are ample.
        \end{enumerate}
      \end{proposition}
      Note that if  $-K_T- 2D$ and $-K_T+D$ are ample, then
      $-K_T+\lambda D$ is ample for every $\lambda\in\R$ with $-2\leq\lambda\leq 1$.
      \begin{proof}
If $X$ is Fano, then  $-K_T- 2D$ and $-K_T+D$ are ample by Lemmas \ref{E} and \ref{F}.

   Conversely, assume $(ii)$. We show that     
  $-K_X$ is strictly nef, namely that $-K_X\cdot C>0$ for every 
irreducible curve
$C\subset X$. If $C$ is contained in one of the divisors $E,\wi{E},F,\wi{F}$, then $-K_X\cdot C>0$ by Lemmas \ref{E} and \ref{F}.

  Suppose instead that $C$ is not contained in any of the divisors $E,\wi{E},F,\wi{F}$. Then 
  $$\bigl(E+\wi{E}+\frac{1}{4}(F+\wi{F})\bigr)\cdot C\geq 0,$$
  and $\psi^*(-K_T)\cdot C=-K_T\cdot\psi_*(C)\geq 0$ because $T$ is Fano, thus 
  by Lemma \ref{-K} we get $-K_X\cdot C \geq 0$.

  Moreover
  $-K_X\cdot C =0$ would imply that $C$ is contracted by $\psi$ and disjoint from the exceptional divisors $E,F$ of $\alpha\colon X\to Z$. Then $\alpha(C)\subset Z$ must be contained in a fiber of the $\pr^2$-bundle $\pi_Z\colon Z\to T$, and $-K_Z\cdot\alpha(C)=-K_X\cdot C=0$, a contradiction.
  Therefore $-K_X\cdot C>0$.

\bigskip
  
We show that $-K_X$ is big. In $Z$ the divisor $P_Z:=H_Z+\pi_Z^*(-K_T-D)$ is tautological for $\pr_T(\ol(-K_T-D)\oplus\ol(-K_T)\oplus\ol(-K_T+D))$ (see \ref{defZ}) and the divisors $-K_T-D$, $-K_T$, and $-K_T+D$ are ample, so that $P_Z$ is ample on $Z$.  Therefore there  exist $m,c_1,c_2\in\R_{>0}$ such that
$P_X:=m\alpha^*(P_Z)-c_1E-c_2F$ is ample on $X$.

Using \eqref{coffee}, we have
\begin{align*}
\alpha^*(P_Z) & =\alpha^*(H_Z)+\alpha^*\pi_Z^*(-K_T-D)=(\tau')^*(H_2')+\psi^*(-K_T)
-\psi^*(D)\\
& \sim(\tau')^*\bigl(E_2'+(\ph')^*(H_W+\pi_W^*(D))\bigr)+\psi^*(-K_T)
-\psi^*(D)\\
& \sim E+\frac{1}{4}(\tau')^*(\ph')^*(B_W)+\psi^*(D)+\psi^*(-K_T)
-\psi^*(D)\\
&=E+\frac{1}{4}(\tau')^*(\wi{F}_2')+\psi^*(-K_T)
                =E+\frac{1}{4}\wi{F}+\frac{1}{4}F+\psi^*(-K_T).
\end{align*}
Then:
$$
P_X\equiv m\psi^*(-K_T)+(m-c_1)E+\frac{1}{4}m\wi{F}+\bigl(\frac{1}{4}m-c_2\bigr)F$$
and by Lemma \ref{-K} we get
 $$ -mK_X  = m\psi^*(-K_T)+mE+m\wi{E}+\frac{1}{4}mF+\frac{1}{4}m\wi{F}\\
      = P_X+c_1E+m\wi{E}+c_2F.$$
This shows that $-K_X$ is big, so that it is ample by the base-point-free theorem.
\end{proof}
\begin{lemma}
  Suppose that $X$ is Fano. We have the following:
  \begin{enumerate}[$(a)$]
   \item $\de(E)=\de(\wi{E})=\de(F)=\de(\wi{F})=2$; 
\item
     $\delta_X\geq 2$;
   \item $\delta_X=2$ if and only if $\delta_T\leq 2$;
   \item if $\delta_X\geq 3$, then  $\delta_X=\delta_T$.
     \end{enumerate}
\end{lemma}
\begin{proof}
  The restriction $\psi_{|E}\colon E\to T$ is a $\pr^1$-bundle, and $E\cap\wi{E}$ is a section, so it gives a morphism $s\colon T\to E$ such that $\psi\circ\iota\circ s=\text{Id}_T$, where $\iota\colon E\hookrightarrow X$ is the inclusion. Then $\psi_*\circ\iota_*\circ s_*=\text{Id}_{\N(T)}$.

  We show that $\iota_*\colon\N(E)\to\N(X)$ is injective. Let $\gamma\in\ker\iota_*$; we can write $\gamma=\lambda[e]_E+s_*(\gamma_0)$, where $\lambda\in\R$, $\gamma_0\in\N(T)$, and $e\subset E$ is a fiber of the $\pr^1$-bundle. By applying $\iota_*$ we see that $\gamma_0=0$, and then $0=\iota_*(\gamma)=\lambda[e]_X$, so that $\lambda=0$ and $\gamma=0$.
  
Therefore  $\de(E)=\codim\N(E,X)=\rho_X-\rho_E=(\rho_T+3)-(\rho_T+1)=2$. The argument for $\wi{E}$ is analogous.

  Consider now $B_2'=\wi{E}_2'\cap\wi{F}_2'\subset X_2'$. We have $B_2'\cap E_2'=\emptyset$ (Lemma \ref{A'}). We also have $\wi{F}_2'=(\ph')^{-1}(B_W)$ and $A_W\cap B_W=\emptyset$ (see \ref{ref}), therefore $G:=(\ph')^{-1}(A_W)\subset X_2'$ is a prime divisor disjoint from $B_2'$. We have $E_2'\neq G$ ($E_2'$ is a section of $\ph'$) and $E_2'=\Exc(\sigma')$, hence the classes $[E_2'],[G]\in\Nu(X_2')$ are linearly independent, and the hyperplanes $(E_2')^{\perp},G^{\perp}\subset\N(X_2')$ are distinct. Since $\N(B_2',X_2')\subset (E_2')^{\perp}\cap G^{\perp}$ (see Remark \ref{disjoint}), we get $\dim\N(B_2',X_2')\leq \rho_{X_2'}-2=\rho_T$. On the other hand $\pi_W(\ph'(B_2'))=\pi_W(B_W)=T$, therefore $\pi_{W*}((\ph')_*(\N(B_2',X_2')))=\N(T)$, and $\dim\N(B_2',X_2')=\rho_T$. This implies that $\dim\N(F,X)=\dim\N(B_2',X_2')+1=\rho_T+1=\rho_X-2$, and $\de(F)=2$. The argument for $\wi{F}$ is analogous.
  This gives $(a)$; moreover $(b)$ is straightforward from $(a)$.

The existence of the surjective morphism $\psi\colon X\to T$ implies that $\delta_T\leq\delta_X$, see \cite[Remark 3.3.18]{codim}. Therefore $\delta_X=2$ implies that $\delta_T\leq 2$. Conversely, we show that
$\delta_X\geq 3$ implies $\delta_T\geq \delta_X\geq 3$, which
gives both $(c)$ and $(d)$.

Thus let us assume that $\delta_X\geq 3$, and
let $P\subset X$ be a prime divisor with $\delta(P)=\delta_X$, so that $\dim\N(P,X)\leq\rho_X-3$. Note that $\psi$ cannot be finite on $P$, thus $\N(P,X)\cap\ker\psi_*\neq 0$, and $\dim\psi_*(\N(P,X))\leq\rho_X-4$. Since $\rho_T=\rho_X-3$, we have $\psi_*(\N(P,X))\subsetneq\N(T)$, and we conclude that
 $\psi(P)\subsetneq T$. Since $\psi$ is flat, we deduce that $P_T:=\psi(P)$ is a prime divisor in $T$, and
  $P$ contains some fiber $F_0$ of $\psi$.

  We have $\N(F_0,X)=\ker\psi_*$ because $\psi$ is quasi-elementary (see \ref{defX}), thus $\ker\psi_*\subset\N(P,X)$, and we get $\codim\N(P_T,T)=\codim\N(P,X)$ and hence $\delta_T\geq \delta(P_T)=\delta(P)=\delta_X$.
\end{proof}


\begin{lemma}\label{madeira}
  Suppose that $X$ is Fano. The morphism $\psi\colon X\to T$ has another factorization:
$$\xymatrix{ X\ar@/^{1pc}/[rr]^{\alpha}\ar[r]_{\sigma}&{X_2}\ar[r]_{\tau}\ar[d]_{\ph}& Z\ar[d]^{\pi_Z}\\ 
&Y\ar[r]^{\pi_Y}&T
}$$
where:
\begin{enumerate}[$(a)$]
\item $\tau\colon X_2\to Z$ blows-up $B_Z:=\sigma'(B_2')\subset Z$ and $\sigma\colon X\to X_2$ blows-up the transform $A_2\subset X_2$ of $A_Z$;
\item $Y=\pr_T(\ma{F})\stackrel{\pi_Y}{\to} T$ is a $\pr^1$-bundle, and $A_Y:=\ph(A_2)$ is a section of $\pi_Y$, corresponding to
  an extension:
  $$0\la \ol_T(2D)\la \ma{F}\la\ol_T\la 0;$$
\item $\ph$ is a conic bundle with non-empty discriminant locus $\Delta_{\ph}\subset Y$, $\Delta_{\ph}\cap A_Y=\emptyset$, and $\Delta_{\ph}\sim 3 A_Y+6\pi_Y^*(D)$.
\end{enumerate}  
\end{lemma}
\begin{proof}
  Let us factor $\alpha\colon X\to Z$ by first blowing-up $B_Z=\sigma'(B_2')$ and then the transform $A_2$ of $A_Z$; we obtain:
  $$\xymatrix{ X\ar@/_/[rrd]_{\psi}\ar@/^{1pc}/[rr]^{\alpha}\ar[r]_{\sigma}&{X_2}\ar[r]_{\tau}& Z\ar[d]^{\pi_Z}\\ 
&&T
}$$
where $\sigma$ and $\tau$ are as in statement $(a)$.

Consider the composition $\pi_Z\circ\tau\colon X_2\to T$; it is a quasi-elementary contraction (because $B_Z$ dominates $T$) with relative Picard number two, equidimensional, with irreducible fibers. It is also $K$-negative, indeed it is finite on $A_2$, and if an irreducible  curve $\Gamma\subset X_2$ is not contained in $A_2$, and $\Gamma_X\subset X$ is its transform, we have $E\cdot\Gamma_X\geq 0$, $-K_X\cdot\Gamma_X>0$, hence $$-K_{X_2}\cdot\Gamma=-K_{X_2}\cdot\sigma_*(\Gamma_X)=(-K_X+E)\cdot\Gamma_X>0.$$

Let $R$ be the extremal ray of $\NE(X_2)$ such that $\NE(\pi_Z\circ\tau)=\NE(\tau)+R$. Let $F_R$ be an irreducible component of a non-trivial fiber of the contraction of $R$, and  $F_0$  the fiber of 
$\pi_Z\circ\tau$ containing it. We have $\dim\N(F_0,X_2)=2$ while $\dim\N(F_R,X_2)=1$, therefore $F_R\subsetneq F_0$. Since $\dim F_0=2$ and $F_0$ is irreducible, we get $\dim F_R=1$.

We conclude that  $R$ is a $K$-negative extremal ray whose contraction has fibers of dimension $\leq 1$. 
By \cite[Theorem 1.2]{wisn} we see that  either $R$ is of type $(n-1,n-2)^{sm}$, or its contraction is a conic bundle. Moreover $\Lo(R)$ must dominate $T$ because $\pi_Z\circ\tau$ is quasi-elementary.

Let us consider the generic point $\eta=\Spec\C(T)$ of $T$, and the generic fiber $(X_2)_\eta$ of  $\pi_Z\circ\tau$. It is the blow-up of $Z_\eta=\pr^2_K$, where $K=\C(T)$, along the point $(B_Z)_\eta$, of degree $4$. By the classification of elementary links (see \cite[Theorem 2.6]{isk96} or \cite[Propositions 5.1 and 5.5]{BFSZ}), the second elementary contraction of  $(X_2)_\eta$ is a conic bundle. Then by Lemma \ref{NE} we see that the 
 contraction $\ph\colon X_2\to Y$ of $R$ is a conic bundle too, and $Y$ is smooth of dimension $n-1$. We get also a contraction $\pi_Y\colon Y\to T$, with general fiber $\pr^1$, such that $\pi_Z\circ\tau=\pi_Y\circ\ph$.

We note that $A_2$ is a section for $\pi_Z\circ\tau$, hence $A_Y:=\ph(A_2)\subset Y$ is a section of $\pi_Y$. This implies first of all that $\pi_Y$ is equidimensional and hence a conic bundle, and in fact that it is a $\pr^1$-bundle.
We have $Y=\pr_T(\ma{F})$ where $\ma{F}$ is a rank $2$ vector bundle on $T$ sitting in an extension:
  $$0\la\ma{L}\la\ma{F}\la\ol_T\la 0$$
  where $\ma{L}\in\Pic(T)$ is such that $\ma{L}\cong \ol_Y(-A_Y)_{|A_Y}$.

  Let $\wi{E}_2\subset X_2$ be the transform of  $\wi{E}_Z\subset Z$, and notice that $A_2\subset\wi{E}_2$. Fiberwise over $T$, $\wi{E}_Z\cap Z_t$ is the smooth conic through the $5$ points $(A_Z\cup B_Z)\cap Z_t$ (see Figure \ref{figura_Z_t} on page \pageref{figura_Z_t}).
  The conic bundle $\ph_{|(X_2)_t}\colon (X_2)_t\to\pr^1$ is given by 
pencil of conics through the $4$ points $B_Z\cap Z_t$, and 
  $\wi{E}_2\cap (X_2)_t$ is a smooth fiber, 
  hence $\ph(\wi{E}_2)=A_Y$ and $\wi{E}_2=\ph^*(A_Y)$. 
  Now we have $$\ol_{X_2}(\wi{E}_2)_{|\wi{E}_2}\cong\ph^*\ol_Y(A_Y)_{|\wi{E}_2} =(\ph_{|\wi{E}_2})^*\ol_Y(A_Y)_{|A_Y}\cong(\ph_{|\wi{E}_2})^*(\ma{L}^{\vee}),$$
  and since $A_2\cap\Exc(\tau)=\emptyset$, using \ref{defZ} and Lemma \ref{Ehat} we get
  $$\ma{L}\cong \ol_{X_2}(-\wi{E}_2)_{|A_2}\cong \ol_Z(-\wi{E}_Z)_{|A_Z}
  \cong\ol_Z\bigl(-2H_Z+\pi_Z^*(2D)\bigr)_{|A_Z}\cong\ol_T(2D),$$
so we have $(b)$.

  For $t\in T$ a general closed point, the $4$ points $B_Z\cap Z_t$ are distinct, and must be in general linear position, otherwise one can find a line in $Z_t$ whose transform in $X$ has non-positive anticanonical degree. Therefore the pencil of conics through $B_Z\cap Z_t$ contains exactly $3$ singular conics, and $\ph$ has $3$ singular fibers over $\pi_Y^{-1}(t)\cong\pr^1$. This implies that $\Delta_{\ph}\neq\emptyset$ and that $\Delta_{\ph}$ has generically degree $3$ onto $T$.

  We also have $A_Y\cap\Delta_{\ph}=\emptyset$, because $\ph_{|\wi{E}_2}\colon
  \wi{E}_2\to A_Y$ is smooth, having $A_2$ as a section. Write $\Delta_{\ph}\sim 3 A_Y+\pi_Y^*(M)$, $M$ a divisor on $T$, and recall that $A_{Y|A_Y}\sim -2D$. Then $0=\Delta_{\ph|A_Y}\sim 3A_{Y|A_Y}+\pi_Y^*(M)_{|A_Y}\cong -6D+M$, hence $M\sim 6D$, and we have $(c)$.
\end{proof}
\begin{remark}
 Suppose that $X$ is Fano. By Lemma \ref{madeira} $(b)$ and $(c)$, and Lemma \ref{sectionsFano},  we see that $Y$ is Fano if and only if $-K_T+2D$ is ample; in this case we have $Y\cong\pr_T(\ol\oplus\ol(2D))$.
\end{remark}  
\begin{remark}
 Suppose that $X$ is Fano. If $D$ is nef, then  $-K_T+2D$ is ample, therefore $Y$ is Fano and $Y\cong\pr_T(\ol\oplus\ol(2D))$.
\end{remark}  
\begin{remark}
  Suppose that $X$ is Fano.
We have $E\cong\wi{E}\cong W=\pr_T(\ol\oplus\ol(D))$, and $F\cong\wi{F}\cong\pr_{B_W}(\ol\oplus\ol(2\pi_B^*(D)))$ (see the proof of Lemma \ref{F}).
  The sequence
  $$X\stackrel{\sigma}{\la}X_2\stackrel{\ph}{\la} Y$$
  is a special MMP of type $(b)$ for $-\wi{F}$. Moreover the sequence
  $$X\stackrel{\tau'}{\la}X_2'\stackrel{\ph'}{\la} W$$
is a special MMP of type $(b)$ for  $-\wi{E}$.
\end{remark}  
  \begin{remark}\label{lori}
    Let $T$ and $D$ be as in Proposition \ref{construction}. Then in particular $-K_T\pm D$ are both ample, therefore for every irreducible curve $C\subset T$ with $-K_T\cdot C=1$ we must have $D\cdot C=0$, namely $[C]\in D^{\perp}\subset\N(T)$.
Since $D\not\equiv 0$ (see Remark \ref{Deffective}), this implies that the classes of curves of anticanonical degree one are contained in a hyperplane of $\N(T)$.
    This is quite a restriction on $T$, for instance it implies that $\NE(T)$ must have at least one extremal ray of length $\geq 2$.  
\end{remark}
\section{The $4$-dimensional case}\label{dim4}
\noindent In this section we apply the previous results to the $4$-dimensional case, and get the following.
\begin{corollary}\label{4fold}
There exists a family of Fano $4$-folds $X$ with $\de_X=2$ as in Theorem \ref{singular}$(i)$. Such family is unique, and is obtained as in Proposition \ref{construction} from $(T,D)=(\pr^2,\ol_{\pr^2}(1))$; its invariants are given in Table \ref{t1}.
{\footnotesize
\begin{table}[h]
 $\begin{array}{||c|c|c|c|c|c|c|c||}
   \hline\hline
   
   \rho_{X} & K_{X}^4 & b_3(X)& h^{2,2}(X) & h^{1,3}(X) &  h^0(-K_{X}) & c_2(X)\cdot K_X^2 & \chi(T_X)\\
      \hline

       4 & 160 & 0  & 24 & 1 & 40 & 148 & -22\\
      
\hline\hline
  \end{array}$
  
\bigskip
  
  \caption{}\label{t1}
\end{table}}
\end{corollary}
Explicitly, $\alpha\colon X\to Z=\pr_{\pr^2}(\ol\oplus \ol(1)\oplus\ol(2))$ is the blow-up of two disjoint smooth, irreducible surfaces $A_Z,B_Z\subset Z$, 
where $A_Z$ is the section of the $\pr^2$-bundle with normal bundle $\ma{N}_{A_Z/Z}\cong\ol_{\pr^2}(-1)\oplus\ol_{\pr^2}(-2)$.  Here $Z$ is toric, and 
$\Bl_{A_Z}Z$ is the toric Fano $4$-fold $D_3$ in \cite{bat2}.
Moreover, $W=\pr_{\pr^2}(\ol\oplus\ol(1))$ is the blow-up of $\pr^3$ at a point, and $B_W\subset W$ is the transform of a smooth quartic surface in $\pr^3$ not containing the blown-up point; therefore $B_Z\cong B_W$ is a quartic K3 surface. 
\begin{proof}[Proof of Corollary \ref{4fold}]
  Consider the pair $(T,D)=(\pr^2,\ol_{\pr^2}(1))$. Then the assumptions of Proposition \ref{construction} are satisfied (note that $D$ is base-point-free), therefore there exists a family of Fano $4$-folds as in the statement.  The invariants of $X$ are computed from the explicit description of $X$ as a blow-up of $Z=\pr_{\pr^2}(\ol\oplus \ol(1)\oplus\ol(2))$ given in \S \ref{reconstructing}; see for instance \cite[Lemma 3.2]{delta3_4folds} and \cite[Lemma 6.25]{vb} for the relevant Riemann-Roch formulas.
  
  Conversely, let $X$ be a Fano $4$-fold as  in Theorem \ref{singular}$(i)$. Then by Theorem \ref{rho=6} we get a smooth del Pezzo surface $T$ with a divisor $D$ such that $[D]\in\Eff(T)$, $D\not\equiv 0$, and $-K_T-2D$ and $-K_T+D$ are ample. Moreover every $(-1)$-curve of $T$ must have intersection zero with $D$, and  classes of $(-1)$-curves cannot generate $\N(T)$ (see Remark \ref{lori}). This implies immediately that
$\rho_T\leq 2$; moreover, if $T\cong\mathbb{F}_1$, then  $D=\sigma^*\ol_{\pr^2}(a)$ with $a>0$, and one can check that $-K_T-2D$ is never ample. Similarly, if $T\cong\pr^1\times\pr^1$, for every effective, non-zero $D$ the divisor $-K_T-2D$ is not ample. 
  We are left with $T\cong\pr^2$, and $-K_T-2D$ ample implies that $D$ is a line.
\end{proof}

\section{Open questions and special cases}\label{questions}
\noindent While the structure of Fano varieties with Lefschetz defect $\de_X\geq 3$ is well understood (see Theorems \ref{geq 4} and \ref{delta=3}), we still have a few open questions on the case $\de_X=2$.

Let us consider first  the case of non-empty discriminant; then by Theorem \ref{singular} there is a del Pezzo fibration $\psi\colon X\to T$ with $T$ smooth and $\rho_X-\rho_T=3$.
\begin{question}
In the setting of Theorem \ref{singular}, does case $(ii)$ with generic fiber $\Bl(\pr^2_{\C(T)};2,2)$ actually occur? If so, is the base $T$ of del Pezzo fibration  always Fano? Is it possible to give for  case $(ii)$ a structure theorem and an explicit construction like Theorem \ref{rho=6} and Proposition \ref{construction} for case $(i)$?
\end{question} 

Let us consider now the case where $\Delta_{\ph}=\emptyset$, so that by Theorem \ref{main} we get a diagram:
  $$\xymatrix{ X\ar@{-->}[r]^{\xi}&{X'}\ar[r]^{\sigma}\ar[dr]_{\zeta}&{X_2}\ar[d]^{\ph}\\
    &&Y
  }$$
  where $\ph$ is a smooth $\pr^1$-fibration.
  \begin{question}
    Is $Y$ always weak Fano, i.e.\ is $-K_Y$ nef and big?

    Note that in presence of flips, $Y$ does not need to be Fano, as toric examples show (see for instance toric Fano $4$-folds of type $J$ or $R$ in \cite{bat2}). However, in the toric case, $Y$ is always weak Fano. Let us also note that in any case $Y$ is of Fano type, and $-K_Y$ is big.
  \end{question}  
\begin{remark}[the case where $\ph$ is smooth and without  flips]\label{linate}
  Suppose now that, in the setting above, there are no flips, so that $X=X'$ and $X$ itself has a conic bundle structure:
  $$\xymatrix{X\ar@/^{1pc}/[rr]^{\zeta}\ar[r]_{\sigma}&{X_2}\ar[r]_{\ph}&Y
    }$$
    where $\ph$ is a smooth $\pr^1$-fibration and $\sigma$ is the blow-up of a smooth, codimension $2$ subvariety $A_2\subset X_2$, a section of $\ph_{|\ph^{-1}(A_Y)}$, where $A_Y=\ph(A_2)\subset Y$. Then the conic bundle $\zeta$ has no non-reduced fibers, hence $Y$ is Fano by \cite[Proposition 4.3]{wisn}.
    
    In this case, we would like to understand whether $X$ can be reconstructed from some data on the Fano variety $Y$.

\smallskip
    
    In \cite{pier} this setting is studied under the assumption that  $A_2$ is contained in a global section $S\subset X_2$ of $\ph$; note that in this case
    $X_2=\pr_Y(\ma{E})$ for some rank $2$ vector bundle on $Y$.
 In [\emph{ibidem}, Theorem 1.3 and Proposition 2.3]   explicit conditions on $A_Y$ and $\ma{E}$ are given in order to obtain  a Fano variety $X$ as above. In particular, for $n=4$, by choosing  suitable decomposable vector bundles on  Fano $3$-folds $Y$, this yields 144 distinct families  of Fano $4$-folds $X$ with $\de_X=2$ and $\rho_X\in\{4,5\}$ [\emph{ibidem}, Theorem 1.4]. 
 However, an example is given to show that in general $A_2$ does not need to be contained in a global section [\emph{ibidem}, p.\ 24].
 \end{remark}
   \begin{remark}[the case $\rho_X=3$]\label{rho3}
   Let $X$ be a Fano variety with $\de_X=2$; then the Picard number of $X$ is at least  $3$. When moreover $\rho_X=3$, it is shown in \cite[Theorem 3.8]{minimal} that the setting described above always holds, namely $X=\Bl_{A_2}X_2$ where $X_2=\pr_Y(\ma{E})$, $\ma{E}$ is a decomposable rank 2 vector bundle on a Fano variety $Y$ with dimension $n-1$ and $\rho_Y=1$, and $A_2$ is contained in a section of $\ph\colon X_2\to Y$. Using this structure, Fano $4$-folds with $\de_X=2$ and $\rho_X=3$ have been classified in \cite{saverio}; they form 28 families.

     Suppose that $n\geq 3$.   Let us also note that we have $\de(E)=1$
for $E:=\Exc(\sigma)$, differently from the case $\Delta_{\ph}\neq\emptyset$, where $\de(E)=2$ (see Proposition \ref{giugno}).
Indeed $\sigma(E)=A_2$ and $\dim A_2=n-2>0$, therefore $\dim\N(E,X)>1$. On the other hand $\N(E,X)$ does not contain the class of a general fiber of $\zeta$ (see \cite[\S 2.3]{codimtwo}), therefore $\N(E,X)\subsetneq\N(X)$, and finally $\dim\N(E,X)=2$. 
   \end{remark}
   \begin{remark}\label{typeb}
     Let $X$ be a smooth Fano variety with $\de_X=2$ and $\rho_X=3$, as in Remark \ref{rho3}, of dimension $n\geq 3$. Let $D\subset X$ be a prime divisor with $\de(D)=2$, namely $\dim\N(D,X)=1$. Then every special MMP for $-D$ is of type $(b)$.

     Indeed, if there were a special MMP of type $(a)$, by \cite[\S 2.2]{codimtwo} we would get an exceptional $\pr^1$-bundle $E_1\subset X$ such that $D\cdot e_1>0$ ($e_1\subset E_1$ a fiber of the $\pr^1$-bundle) and $\dim\N(D\cap E_1,X)=\rho_X-3=0$, but this is impossible because $D\cap E_1\neq \emptyset$ and hence $\dim (D\cap E_1)\geq n-2\geq 1$.
   \end{remark}  
    
\small
\providecommand{\noop}[1]{}
\providecommand{\bysame}{\leavevmode\hbox to3em{\hrulefill}\thinspace}
\providecommand{\MR}{\relax\ifhmode\unskip\space\fi MR }
\providecommand{\MRhref}[2]{%
  \href{http://www.ams.org/mathscinet-getitem?mr=#1}{#2}
}
\providecommand{\href}[2]{#2}

\end{document}